\documentclass[a4paper]{amsart}
\usepackage{amssymb}
\usepackage{amsmath}
\usepackage{mathrsfs}
\usepackage{amsthm}
\usepackage{cleveref}
\usepackage[numbers]{natbib}
\usepackage{dsfont}
\usepackage{enumitem}
\usepackage{expdlist}
\usepackage{mathtools}
\usepackage{comment}
\theoremstyle{definition} \newtheorem{lemma}{Lemma}
\theoremstyle{definition} \newtheorem{definition}[lemma]{Definition}
\theoremstyle{definition} \newtheorem{theorem}[lemma]{Theorem}
\theoremstyle{definition} \newtheorem{Question}{Question}
\theoremstyle{definition} \newtheorem{example}[lemma]{Example}
\theoremstyle{remark} \newtheorem*{remark}{Remark}
\theoremstyle{definition} \newtheorem{corollary}[lemma]{Corollary}
\theoremstyle{definition} 
\theoremstyle{remark} 
\theoremstyle{remark} \newtheorem*{acknowledgements}{Acknowledgements}

\newcommand{\scra}[2]{\mathscr{A}_{#1,#2}}

\newcommand{\df}{\mathrel{\mathop:}=}
\newcommand{\meas}[1]{\lambda \left(#1\right)}

\newcommand{\one}{\log}
\newcommand{\two}{\log \log}
\newcommand{\three}{\log \log \log}

\title{Variations, approximation, and low regularity in one dimension}
\author{Richard Gratwick}
\address{School of Mathematics, James Clerk Maxwell Building,	The King's Buildings,	Peter Guthrie Tait Road,	Edinburgh,	EH9 3FD, UK.}
	\email{R.Gratwick@ed.ac.uk}
\date{\today}
\thanks{The research leading to these results has received funding from the European Research Council under the European Union's Seventh Framework Programme (FP/2007-2013) / ERC Grant Agreement n.~291497.}
\begin{document}

\begin{abstract}
We investigate the properties of minimizers of one-dimensional variational problems when the Lagrangian has no higher smoothness than continuity.  An elementary approximation result is proved, but it is shown that this cannot be in general of the form of a standard Lipschitz ``variation''.  Part of this investigation, but of interest in its own right, is an example of a nowhere locally Lipschitz minimizer which serves as a counter-example to any putative Tonelli partial regularity statement.  Under these low assumptions we find it nonetheless  remains possible to derive necessary conditions for minimizers, in terms of approximate continuity and equality of the one-sided derivatives.
\end{abstract}
\maketitle
\section{Introduction}
\label{sec:intro}
The basic problem of the one-dimensional calculus of variations is the minimization of the functional 
\[ 
\mathscr{L}(u) = \int_{a}^{b} L(t, u(t), u'(t))\, dt
\]
over some class of functions $u \colon [a,b] \to \mathbb{R}^n$ with fixed boundary conditions.  The integrand $L \colon [a,b] \times \mathbb{R}^n \times \mathbb{R}^n\to \mathbb{R}$ is known as the {\it Lagrangian}.  Tonelli~\cite{Tonelli-Fondamenti-2, Tonelli-1934} presented rigorous existence results for minimizers of such a problem, demonstrating the need to work on the function space of absolutely continuous functions, or what is now known also as the Sobolev space $W^{1,1}((a,b) ;\mathbb{R}^n)$.  In particular such functions are only differentiable almost everywhere.   Defining the functional $\mathscr{L}$ on this space, Tonelli developed the {\it direct method} of the calculus of variations to deduce the existence of minimizers when certain conditions are imposed on the Lagrangian.  The key assumptions are the conditions of convexity and superlinearity: i.e.\ that the function $p \mapsto L(t, y, p)$ is convex for each $(t,y)$, and that there exists some $\omega \colon \mathbb{R} \to \mathbb{R}$ satisfying $\omega(\|p\|) / \|p\| \to \infty$ as $\|p\| \to \infty$ such that $L(t, y, p) \geq \omega(\|p\|)$ for all $(t, y, p)$.  Some minimal smoothness of the Lagrangian is also required, for example continuity suffices.  The subject of this paper is what can happen at this level of regularity, i.e.\ when the Lagrangian is assumed only to be continuous. 

The penalty paid for an abstract existence theorem is that one must work in a suitable function space, and therefore can only assert that the minimizer is $W^{1,1}$.  A significant question is then whether it is possible to assert {\it a priori} any higher regularity of minimizers.  Assuming appropriate growth conditions and $C^k$-regularity of the Lagrangian, one may prove $C^k$-regularity of the minimizers (see for example~\cite{Buttazzo-Giaquinta-Hildebrandt}).  For scalar-valued functions $u$, Tonelli~\cite{Tonelli-Fondamenti-2} provided a partial regularity theorem, asserting that $C^{\infty}$-regularity of the Lagrangian and strict convexity in $p$ implies that any minimizer $u$ is $C^{\infty}$ on an open set of full measure.  Clarke and Vinter~\cite{Clarke-Vinter-1985-regularity} gave an analogous statement for vector-valued functions.

The assumption of strict convexity may not be weakened, but several authors have weakened the smoothness assumption on the Lagrangian.  Clarke and Vinter imposed only a local Lipschitz condition in $(y, p)$.  In the scalar case, Sych\"ev~\cite{Sychev-1993} imposed a local H\"older condition, and Cs\"ornyei et al.~\cite{Csornyei-etal-2008} imposed a local Lipschitz condition in $y$, locally uniformly in $(t,p)$. In the vectorial case, Ferriero~\cite{Ferriero-2012, Ferriero-2013} allowed this Lipschitz constant to vary as an integrable function of $t$.  Recalling that no control of the modulus of continuity is required for the existence theorem, Gratwick and Preiss~\cite{Gratwick-Preiss-2011} gave a counter-example of a continuous Lagrangian which admits a minimizer non-differentiable on a dense set.  So we are faced with the possibility of situations where minimizers over $W^{1,1}$ exist, but partial regularity results fail to hold.  Section~\ref{sec:tonelli-cex} presents a new counter-example illustrating this, with a minimizer having upper and lower derivatives $\pm \infty$ at a dense set of points.  

A standard technique to prove necessary conditions of minimizers is to compute the first variation, i.e.\ to consider the limiting behaviour of the function $\gamma \mapsto \mathscr{L} (u + \gamma w)$ as $\gamma \to 0$.  Following this path in the classical situation leads us to the Euler-Lagrange equation and other necessary conditions.  In our low-level regularity situation, assuming only continuity of the Lagrangian, it is not immediately clear how such small perturbations behave.  Ball and Mizel~\cite{Ball-Mizel-1985} gave examples of polynomial Lagrangians for which $\mathscr{L}(u + \gamma w) =  \infty$ for a certain class of smooth functions $w$.  In our case, when we do not have a partial regularity theorem, and must therefore admit the possibility of minimizers which are nowhere locally Lipschitz, it is not even immediately clear that it is possible to approximate the minimum value by any other trajectories at all.

The possibility of a complete failure of approximation is not absurd when one considers the possible presence of the {\it Lavrentiev phenomenon}~\cite{Lavrentiev-1926}, in which situation the energy of Lipschitz functions with the required boundary conditions is bounded away from the minimum value.  That this can occur not only for polynomial integrands~\cite{Mania-1934} but even for strictly convex and superlinear polynomial integrands~\cite{Ball-Mizel-1985} should warn us that we are wise to be wary of what might happen when we consider Lagrangians which satisfy only the bare continuity assumption.  Ball and Mizel~\cite{Ball-Mizel-1985} gave another example of bad behaviour to keep us on guard: the {\it repulsion property}~\cite{Ball-2001}, whereby it can happen that $\mathscr{L} (u_n) \to \infty$ for any sequence of admissible Lipschitz  functions $u_n$ which converge uniformly to the minimizer.

Nevertheless, a general approximation result can be proved, indeed without great difficulty.  This is the content of theorem~\ref{thm:JB} in section~\ref{sec:variations}.  In this section we go on to investigate how fruitful it may be to consider computing the variation as suggested above, and discover that in general it will not get us very far: there exist examples (theorem~\ref{thm:no var}), even superlinear and strictly convex examples (theorem~\ref{thm:no var coercive}), of continuous Lagrangians where the addition of any Lipschitz variation to a minimizer results in an infinite value for the integral.  We also investigate the relationship between approximation in this sense and the Lavrentiev phenomenon.  We find in this section that we can make good use of the counter-example to partial regularity described in section~\ref{sec:tonelli-cex}, using in an essential way the main new feature of this example, viz the fact that the minimizer is nowhere locally Lipschitz.  

The technique used to construct the basic approximation in theorem~\ref{thm:JB} is then put to repeated use in section~\ref{sec:regularity}, where we pursue the question of whether any necessary conditions can be derived of minimizers in our setting.  Having lost any hope of a general partial regularity statement, we are left wondering whether it might be the case that an arbitrary $W^{1,1}$ function can be a minimizer of a variational problem with a continuous Lagrangian.  Under the assumption of {\it strict} convexity,  we are able to show that, although the derivative of a minimizer need not exist at every point, at those points at which the derivative does exist, the derivative is approximately continuous.  As a corollary of this, we can then show that when each one-sided derivative exists at a point, the two derivatives must in fact be equal.  Such statements extend, suitably interpreted, to cases where the derivatives are infinite, and we may be more precise when infinite derivatives are confined to one component.

\subsection{Notation and terminology}
Throughout we fix $[a,b] \subseteq \mathbb{R}$, $n \geq 1$ and the euclidean norm $\| \cdot\|$ on $\mathbb{R}^n$.  We shall consider $[a,b] \times \mathbb{R}^n \times \mathbb{R}^n$ to be equipped with the norm given by the maximum of the norms of the three components. The supremum norm of a real-{} or vector-valued function shall be denoted by $\| \cdot \|_{\infty}$, and the support of a such a function shall be denoted by $\mathrm{spt}$.  For a set $E \subseteq [a,b]$ we denote the Lebesgue measure of the set by $\meas{E}$, and the characteristic function of the set by $\mathds{1}_{E}$.

A {\it Lagrangian} shall be a function $L = L(t, y, p) \colon [a,b] \times \mathbb{R}^n \times \mathbb{R}^n \to \mathbb{R}$. Conditions on the Lagrangians shall be discussed at the relevant points, but in particular we demand that they are continuous, but never impose any stronger smoothness condition or prescribe any modulus of continuity.   For a function $v \in W^{1,1}((a,b); \mathbb{R}^n)$ (if $n=1$ we will usually suppress the notating of the target space), we let 
\[
\mathscr{L}(v) = \int_a^b L(t, v(t), v'(t))\, dt.
\]
Recall that {\it superlinearity} is the condition that for some $\omega \colon \mathbb{R} \to \mathbb{R}$ satisfying $\omega(\|p\|) / \|p\| \to \infty$ as $\|p\| \to \infty$, we have for all $(t, y, p) \in [a,b] \times \mathbb{R}^n\times \mathbb{R}^n$ that $L(t, y, p) \geq \omega(\|p\|)$.  For $A, B \in \mathbb{R}^n$ we let 
\[
\scra{A}{B} \df \{ v \in W^{1,1}((a,b); \mathbb{R}^n) : v(a) = A, v(b) = B\}.
\]
\section{Failure of partial regularity}
\label{sec:tonelli-cex}
In this section we present a counter-example to a putative partial regularity theorem in the manner of Tonelli for continuous Lagrangians.  A first example of this kind was produced by Gratwick and Preiss~\cite{Gratwick-Preiss-2011}, exhibiting a Lipschitz minimizer which was non-differentiable on a dense set.  The following example produces a minimizer with upper and lower Dini derivatives of $\pm \infty$ at a dense set of points, i.e.\ the derivative fails to exist at these points in as dramatic a way possible.  That we have both Lipschitz and non-Lipschitz examples is worth emphasizing.  The Lipschitz example serves to disillusion us should we be inclined to suspect, as can be the case, that \textit{a priori} knowledge of boundedness of the derivative of a minimizer implies some higher regularity.  The non-Lipschitz case is remarkable in that intuitively one does not expect superlinear Lagrangians to have minimizers with infinite derivatives at many points, far less minimizers with difference quotients oscillating arbitrarily largely.  

This example was first presented by Gratwick~\cite[Example~2.35]{Gratwick-2011} as an application of a general construction scheme.
\begin{definition}
	The upper and lower Dini derivatives, $\overline{D}v(t)$ and $\underline{D}v(t)$ respectively, of a function $v \in W^{1,1}(a,b)$ at a point $t \in (a,b)$ are given by
	\[
	\overline{D}v(t) \df \limsup_{s \to t} \frac{ v(s) - v(t)}{s-t},\ \text{and}\ \underline{D}v(t) \df \liminf_{s \to t}\frac{v(s) - v(t)}{s-t}.
	\]
\end{definition}
\begin{theorem}
	\label{thm:thesis}
	There exist $T> 0$, $w \in W^{1,2}(-T, T)$, and a continuous $\phi \colon [-T, T] \times \mathbb{R} \to [0, \infty)$ such that 
	\[
	\mathscr{L} (u) = \int_{-T}^{T} \left( \phi(t, u(t) - w(t)) + (u'(t))^2\right) \, dt
	\]
	defines a functional on $W^{1,1}(-T, T)$ with a continuous Lagrangian such that $w$ is a minimizer of $\mathscr{L}$ over $\mathscr{A}_{w(-T),w(T)}$, but $\overline{D}w(t) = + \infty$ and $\underline{D}w(t) = - \infty$ for a dense set of points $t \in [-T, T]$.
\end{theorem}
The remainder of this section is devoted to a proof of this theorem.
\subsection{Construction of the minimizer}
Let $T \in (0, e^{-e^2}/2)$ be small enough such that for any $t \in [-T, T] \setminus\{0\}$, 
\begin{equation}
(2|t|)^{1/3} \two 1/2|t|\leq \left(\frac{1}{\one 1/2|t|}\right)^{\! 1/3} \two 1/ 2|t|  \leq1/125. \label{T0small}
\end{equation}
Given any sequence of points in $(-T, T)$, we can construct a Lagrangian and minimizer $w$ with the set of non-differentiability points of $w$ containing this sequence.  The construction is essentially inductive, and hinges on the fact that a certain function $\tilde{w}$  is non-differentiable at $0$, with difference quotients oscillating between arbitrarily large positive and negative values, but minimizes a problem with a continuous Lagrangian.  This basic Lagrangian is of the form $(t, y, p)  \mapsto \tilde{\phi}(t, y - \tilde{w}(t)) + p^2$ for a ``weight function'' $\tilde{\phi} \colon [-T, T] \times \mathbb{R} \to [0, \infty)$, i.e.\ such that $\tilde{\phi}(\cdot, 0) = 0$ and $|y|\mapsto \tilde{\phi}(t,y)$ is increasing, so $(t, y) \mapsto \tilde{\phi} (t, y - \tilde{w}(t))$ penalizes functions which stray from $\tilde{w}$.  This summand of the Lagrangian then takes its minimum value along the graph of $\tilde{w}$, and assigns larger values to functions $u$ the further their graph lies from that of $\tilde{w}$. 

\label{DPidea} We sketch the main ideas behind the proof that $\tilde{w}$ minimizes this  ``basic'' problem
\[
\mathscr{A}_{\tilde{w}(-T), \tilde{w}(T)} \ni u \mapsto \int_{-T}^T \left(\tilde{\phi}(t, u(t)-\tilde{w}(t)) + (u'(t))^2\right) \, dt.
\]
So suppose for now that $\tilde{u} \in \mathscr{A}_{\tilde{w}(-T), \tilde{w}(T)}$ is a minimizer for this problem.  

If $\tilde{u}(0)=\tilde{w}(0)$, it suffices to argue separately on $[-T,0]$ and $[0, T]$.  We consider $[0,T]$.   Note for any two functions $\bar{u}, \bar{w} \colon [-T, T] \to \mathbb{R}$, we have that
\begin{equation}
\label{triv}
(\bar{u})^2 - (\bar{w})^2 = (\bar{u} - \bar{w})^2 + 2 (\bar{u} - \bar{w})\bar{w} \geq 2 (\bar{u} - \bar{w})\bar{w}. 
\end{equation}
Assuming $\tilde{w}$ is smooth enough that we can integrate by parts, our key argument is the following: 
\begin{align*}
\int_0^{T} \big(\tilde{\phi}(t, \tilde{u}-\tilde{w}) + (\tilde{u}')^2 \big)- \int_0^{T} (\tilde{w}')^2 
& = \int_0^T \left(  \big( (\tilde{u}')^2 - ( \tilde{w}')^2 \big) + \tilde{\phi}(t, \tilde{u} - \tilde{w})\right) \\
&\geq \int_0^{T} \big(2(\tilde{u}'-\tilde{w}')\tilde{w}' + \tilde{\phi}(t, \tilde{u} -\tilde{w})\big)\\
& =[2(\tilde{u}-\tilde{w})\tilde{w}']_0^{T}  \\ 
&\phantom{=} +{}\int_0^{T}\big( \tilde{\phi}(t, \tilde{u} - \tilde{w}) -2(\tilde{u} - \tilde{w})\tilde{w}'' \big) \\
& \geq \int_0^{T} \big(\tilde{\phi}(t, \tilde{u} - \tilde{w})\big) - 2 |\tilde{u} - \tilde{w}|| \tilde{w}''|\big),
\end{align*}
since the boundary terms vanish by assumption.  Hence choosing $\tilde{\phi} (t, y) \geq 2|\tilde{w}''(t)| |y|$, this final expression is non-negative, implying that $\tilde{w}$ is indeed a minimizer with respect to its own boundary conditions.  However, this inequality for $\tilde{\phi}$ cannot be enforced for all values of $(t,y)$, since $|\tilde{w}''(t)| \to \infty$ as $t \to 0$; this is the whole point of the example.  Since we only need this inequality for values of $y = \tilde{u}-\tilde{w}$, we enforce the inequality for only values of $(t,y)$ which lie in the (slightly expanded) convex hull of the graph of $\tilde{w}$, the shape of which is given by a function $g$.  We choose the oscillations of $\tilde{w}$ as we approach to $0$ to be so slow that $|\tilde{w}''(t) g(t)| \to 0$ as $t \to 0$.  So it is possible to construct a well-defined continuous function $\tilde{\phi}$ so that $\tilde{\phi}(t, y) \geq 2 |w''(t)| |y|$ for $|y| \leq c |g(t)|$, for some constant $c$.  To exploit this definition, we then need to establish that  $|\tilde{u}(t) - \tilde{w}(t) | \leq c| g(t)|$.  It suffices to establish that $|\tilde{u}(t)| \leq c |g(t)|$.  This is a consequence of the assumption that $\tilde{u}$ is a minimizer and that $|g|$ is concave on $(0,T)$.  Since $\tilde{u}$ and $\tilde{w}$ agree at $0$ and $T$, any interval on which $\tilde{u}$ lies outside the convex hull of $\tilde{w}$ must be a proper subinterval of $(0, T)$. By the concavity of $g$ on such  an interval, we may find an affine function which lies strictly between $\tilde{u}$ and $g$, and hence $\tilde{u}$ and $\tilde{w}$.  We then consider the competitor function in the minimization problem defined by replacing $\tilde{u}$ with this affine function.  Since affine functions minimize convex functionals, this strictly decreases the gradient term in the integrand. Since the affine function lies closer to $\tilde{w}$ than $\tilde{u}$, this replacement cannot increase the $\tilde{\phi}(t, \cdot - \tilde{w}(t))$ term.  Hence we get a contradiction: $\tilde{u}$ must lie inside the (expanded) convex hull of $\tilde{w}$.

This argument cannot be performed in the case when $\tilde{u}(0) \neq \tilde{w}(0)$, and there is no \textit{a priori} reason why this might not occur.  In this case, we compare $\tilde{u}$ not with $\tilde{w}$ but with a new function we obtain by replacing $\tilde{w}$ with a linear function $\tilde{l}$ on an interval around $0$.  This forces another requirement on the (slow) speed of the oscillations of $\tilde{w}$, since we incur two errors in making this replacement.  We need to control the difference in the $L^2$-norm of the gradients of $\tilde{w}$ and $\tilde{l}$, and the difference in the gradients at the endpoints of this interval, since these latter terms appear as boundary terms when performing the integration by parts inside and outside the interval.  The oscillations of $\tilde{w}$ are carefully chosen so that these errors are controlled by a continuous function of the discrepancy $|\tilde{u}(0)  - \tilde{w}(0)|$, which by a Lipschitz estimate on $\tilde{u}$ in this situation, is comparable to the length of the interval on which we substitute $\tilde{l}$.

This immediately gives us a one-point example of non-differentiability of a minimizer, which already suffices to provide a counter-example to any Tonelli-like partial regularity result.  Other points of non-differentiability are included by inserting translated copies of $\tilde{w}$ into the original $\tilde{w}$, and  passing to the limit, $w$, say.  The final Lagrangian is of the form $(t,y,p) \mapsto \phi(t, y - w(t)) + p^2$, where $\phi$ is a sum of suitably modified translated and truncated copies $\tilde{\phi}_n$ of $\tilde{\phi}$, each of  which penalizes functions which stray from $w$ in a neighbourhood of $x_n$.  Many of the technicalities of the following construction are related to guaranteeing the existence and appropriate properties of $w$ and $\phi$, and are in some sense secondary to the main points of the proof.  As indicated by the sketch of the argument above, we need to understand the first and second derivatives of $w$, and the shape of its convex hull as seen from each point of singularity $x_n$. This demands a number of conditions in the inductive construction of $w$, ensuring that while the function $w$ oscillates as required around $x_n$, elsewhere the derivatives do not interfere in a significant way with the basic argument performed around $x_n$.

Define $g, \tilde{w} \colon \mathbb{R} \to \mathbb{R}$ by  
\begin{equation*}
g(t) = 
\begin{cases}
t \two 1/|t| & t \neq 0,\\
0 & t = 0;
\end{cases}
\ \textrm{and}\ 
\tilde{w}(t) = 
\begin{cases}
g(t) \sin \three 1/|t| & t \neq 0, 
\\0 & t= 0.
\end{cases}
\end{equation*}
Then
\begin{equation}
\tilde{w} \in C^{\infty}(\mathbb{R} \backslash \{0\}),\label{alpha}
\end{equation}
and in particular $\tilde{w}''$ is bounded away from $0$ and $\tilde{w}'$ satisfies the fundamental theorem of calculus on closed intervals not including $0$.
Note that for $t \neq 0$, 
\begin{equation} 
\tilde{w}'(t) =( \two 1/|t| )(\sin \three 1/|t|) - \left(\frac{\sin \three 1/|t| + \cos \three 1/|t|}{\one 1/|t|}\right),\label{lamda}
\end{equation} which is an even function.  We have chosen $T > 0$ small enough such that $1/\one 1/|t| \leq 1 \leq 2 \leq \two 1/ |t| $ for all $t \in [-2T, 2T]\setminus \{0\}$, and so for such $t$ we have that  
\begin{equation}
|\tilde{w}'(t)| \leq \two 1/|t| + \frac{2}{\one 1/|t|}  \leq 3 \two 1/|t|, \label{|tw'|}
\end{equation}
and that 
\begin{equation*}
|\tilde{w} ''(t)|  \leq \frac{2}{|t|\one1/|t|}\left( \frac{ 1}{(\one 1/|t|)( \two 1/|t|)} + \frac{1}{\one1/|t|} + 1 \right) \leq \frac{6}{|t| \one 1/|t|},
\end{equation*}
and hence it follows that
\begin{align}
|g(t)\tilde{w}''(t)| \leq \frac{6 \two 1/|t|}{\one 1/|t|}&\to 0 \ \textrm{as}\ 0 < |t| \to 0, \label{wn''cont}
\end{align}
which is a key fact discussed above which encapsulates one sense in which the oscillations of $\tilde{w}$ are sufficiently slow.

The following functions give us for each $t \in [-T, T]$ the exact coefficients  we shall eventually need in our weight function $\tilde{\phi}$. We define $\psi^1, \psi^2 \colon \mathbb{R} \to [0, \infty)$ by 
\[
\psi^1 (t) = 
\begin{cases}
\frac{1812}{|t| (\one 1/|t|)^{1/3}} &  t \neq 0,\\
0 & t = 0;
\end{cases}
\
\textrm{and}
\
\psi^2 (t)  = 
\begin{cases}
3 + 2|\tilde{w}''(t)| & t \neq 0, \\ 
0 & t=0; 
\end{cases}
\] 
and $\psi \colon \mathbb{R} \to [0, \infty)$ by $\psi(t) = \psi^1(t) + \psi^2(t)$.  Note that by~\eqref{wn''cont},
\begin{equation}
\label{psi2}
t \mapsto g(t)\psi(t) \ \textrm{defines a continuous function on $\mathbb{R}$ with value $0$ at $0$}.
\end{equation}
We may therefore define a constant $C\in (1, \infty)$ by 
\[
C \df 1 + \sup_{t \in [-T, T]}  5|g(t)| \psi (t).
\]

Let $\{ x_n\}_{n=0}^{\infty}$ be a sequence in $(-T, T)$, with $x_0 = 0$. For each $n \geq 0$ define the translated functions $\tilde{w}_n, g_n, \psi_n^1, \psi_n^2, \psi_n \colon [-T, T]\to \mathbb{R}$ by composing the respective function with the translation $t \mapsto (t- x_n)$, thus $\tilde{w}_n (t) = \tilde{w} ( t- x_n)$, etc. 

For each $n \geq 1$, we define $\sigma_n \in (0,1)$ by 
\[
\sigma_n \df \min_{0 \leq i \leq n-1} |x_i - x_n| / 2.
\]
Observe for future reference that 
\begin{equation}
\label{sigmause}
|t - x_n | \leq \sigma_n\ \textrm{implies that}\ |t - x_i| \geq \sigma_n\ \textrm{for all}\ 0 \leq i \leq n -1,
\end{equation}
for otherwise we should have for some $0 \leq i \leq n-1$ that 
\[
|x_i - x_n| \leq |x_i - t| + |t - x_n| < 2 \sigma_n,
\]
which contradicts the definition of $\sigma_n$.

We want to construct a sequence of absolutely continuous functions $w_n$, where for each $0 \leq i \leq n$, up to the addition of a scalar, $w_n = \tilde{w}_i$ on a neighbourhood of $x_i$, thus $w_n$ is singular at $x_i$.  We first define a decreasing sequence $T_n \in (0,1)$ and hence intervals $Y_n \df [x_n - T_n, x_n + T_n]$.  In the inductive construction of $w_n$ we shall modify $w_{n-1}$ only on $Y_n$.  A requirement that these intervals be small and decreasing in measure is the first step towards guaranteeing that the $w_n$ converge to some limit function.

Define a sequence $K_n \in [1, \infty)$ by setting $K_0 = 1$ and so that for $n\geq 1$, we have 
\begin{align}  
\sum_{i=0}^{n-1} (|\tilde{w}''_i (t)| + |\tilde{w}_i'(t)| +1)&\leq K_n\ \textrm{whenever $|t - x_i | \geq \sigma_n$ for all $0 \leq i \leq n-1$};\label{Kn}\\
\shortintertext{and} 
K_n &\geq 1+ K_{n-1}. \label{Kn2} 
\end{align}

Also for $n \geq 0$ define a sequence $\theta_n \in [1, \infty)$ by setting $\theta_0 = 1$ and for $n\geq 1$ setting
\begin{equation}
\label{thetan}
\theta_n = 10K_n \sigma_n^{-1}.
\end{equation}
The scaling constant $\theta_n$ is an unimportant technicality, which just permits some useful estimates, and is chosen so that the graph of $w_n$ always lies inside the ``multi-graph'' of $w_n (x_i) \pm \theta_i |g_i|$, for all $0 \leq i \leq n$; see~\ref{lipwn} below for a precise statement.  Little conceptual understanding would be lost by regarding the $\theta_n$ as constant, and e.g.\ equal to $1$.

For $n\geq 0$ we define $T_n \in (0,1)$ by setting $T_0 = T$ and for $n \geq 1$ inductively defining $T_n$ such that the following conditions hold:
\begin{enumerate}[label=(T:\arabic*)]
	\item $T_n \leq |x_n \pm T| \sigma_n T_{n-1}/2$; and \label{t1}
	\item $|g_n(t) \psi_n (t)| \leq 2^{-n} / 5\theta_n $ for $t \in Y_n$\label{t4}.
\end{enumerate}
Note that \ref{t4} is possible by~\eqref{psi2}.   Since we will only modify $w_{n-1}$ on $Y_n$ to construct $w_n$, we only need to add more weight to our Lagrangian for $t \in Y_n$.  Recalling that we are always working with translations of the same basic function $\tilde{\phi}$ (which we will define explicitly later), we know that we can choose the intervals $Y_n$ small enough so that summing all the extra ``weights'' we need, we still converge to a continuous function.
That the intervals of modification are small enough in this sense is the reason behind these conditions on $T_n$. We observe  that~\ref{t1} guarantees in particular that
\begin{equation}
T_n \leq 2^{-n}\ \textrm{for all $n\geq0$}.\label{iv}
\end{equation} 
Condition~\ref{t1} also guarantees that the points in $Y_n$ are far from the previous $x_i$, in a certain sense.  That $T_n \leq \sigma_n$ implies that~\eqref{sigmause} holds in particular on $Y_n$, i.e.\ that 
\begin{equation}
\label{Ynfarfromxi}
t \in Y_n \ \textrm{implies that}\ |x_i - t| \geq \sigma_n \ \textrm{for all}\ 0 \leq i \leq n-1. 
\end{equation}
This stops the subintervals we later consider from overlapping.

We emphasize that these values of $T_n$ are constructed independently of the later constructed $w_n$; the inductive construction of these functions will require us to pass further down the sequence of $T_n$ than induction would otherwise allow, as we now see.
For $n \geq 0$, find $m_n \geq n$ such that 
\begin{equation}
\label{mn}
2^{-m_n} \leq \frac{T_{n+1}^2}{32}.
\end{equation}
Choose a small open cover  $G_n \subseteq [-T, T]$ of the points $\{x_i\}_{i=0}^{m_n}$ such that 
\begin{equation}
\label{Fnmeas}
\lambda(G_n) \leq \frac{T_{n+1}^2 }{16C},
\end{equation}
and choose 
$M_n \in (1, \infty)$ such that 
\begin{equation}
\label{Mn}
\sum_{i=0}^{m_n}\left(\max\{\psi_i (t), \psi ( T_i)\}\right) \leq M_n\ \textrm{whenever $ t \in [-T, T]\setminus G_n$}.
\end{equation}

We note also that since $g$ is strictly increasing and $g_i(x_i) = g(0) = 0$, for each $n \geq 0$ there exists $\eta_n \in (0,1) $ such that for all $0 \leq i \leq n-1$,
\begin{equation}
\label{etan}
g_i(t) \geq \eta_n\ \textrm{whenever}\ |x_i - t| \geq \sigma_n.
\end{equation}

Let $R_0 = T$ and for $n\geq 1$ inductively construct decreasing numbers $R_n \in (0,T_n)$, progressively smaller fractions of the corresponding $T_n$, such that:
\begin{enumerate}[label=(R:\arabic*)]
	\item 
	\[
	\int_{-R_n}^{R_n} |\tilde{w}'|^2 \leq \frac{T_n^4}{1024\left( 1+ \|\tilde{w}'\|_{L^2(-T, T)}\right)^2};
	\]
	\label{r1} and
	\item 
	\[
	g(R_n) \leq \frac{2^{-n}R_{n-1}T_n^5 \eta_n}{(344 \cdot 512) \left( 1+ \|\tilde{w}'\|_{L^2(-T, T)}\right)^2  K_n^2  M_{n-1}}.
	\]
	\label{r3}
\end{enumerate}
In~\ref{r3}, for brevity we enforce one single inequality, relating the smallness of $R_n$ to all the other construction constants with which we shall have cause to compare it.  On no one application will we need the precise right-hand side as an upper bound.  Rather at various points we need upper bounds which are most conveniently combined in this one expression.
Now define progressively smaller subintervals $Z_n \df [x_n - R_n, x_n + R_n]$ of $Y_n$.  These intervals are those on which we aim to insert a copy of $\tilde{w}_n$ into $w_{n-1}$.  The $Z_n$ must be a very much smaller subinterval of $Y_n$ to allow the estimates we require to hold; the point of this stage in the construction is that we now let the derivative of $w_n$ oscillate arbitrarily highly on $Z_n$, so we have to make the measure of this set very small to have any control over the convergence of $w_n$ in $W^{1,2}(-T,T)$.

The next lemma gives us the very delicate construction of the sequence of functions $w_n$ by which we shall ultimately define our minimizer $w$.  The basic key facts are that $w_n$ oscillates precisely like $\tilde{w}_n = \tilde{w}( \cdot - x_n)$ in a neighbourhood of $x_n$ no larger than $Z_n$; that $w_n$ equals $w_{n-1}$ off $Y_n$; that on $Y_n \setminus Z_n$ both the first and second derivatives of $w_n$ are controlled in terms of those of $w_{n-1}$, in precise ways which are necessary for the inductive construction and the convergence of the $w_n$; and that the graph of $w_n$ lies within that of $w_n (x_i) \pm 2 \theta_i |g_i|$ for each $0 \leq i \leq n$.
\begin{lemma}
	\label{wnlemma}  
	There exists a sequence of $w_n \in W^{1,2}(-T, T)$ satisfying, for $ n \geq 0$:
	\begin{enumerate}[label=(\thelemma.\arabic*)]
		\item $w_n (t)=  \tilde{w}_n(t) + \rho_n$ when $t \in [x_n - \tau_n, x_n + \tau_n]$, for some $\tau_n \in (0, R_n]$, and some $\rho_n \in \mathbb{R}$;\label{wn=twn}
		\item $w_n'$ exists and is locally Lipschitz on $(-T, T) \setminus \{x_i\}_{i=0}^n$;\label{wn'lip}
		\item $|w_n (t) - w_n (x_i)| \leq (2 - 2^{-n})\theta_i|g_i (t)|$ for all $t \in [-T, T]$ and for all $ 0 \leq i \leq n$;\label{lipwn}
		\item $|w_n'(t)| \leq K_{n+1}$ when $|t - x_{n+1}| \leq \sigma_{n+1}$, in particular on $Y_{n+1}$\label{wn'bd};
		\item $w_n''$ exists almost everywhere and $|w_n''(t)| \leq K_{n+1}$ for almost every $t$ such that $|t - x_{n+1}| \leq \sigma_{n+1}$, in particular on $Y_{n+1}$;
		\label{wn''bd}
	\end{enumerate}
	and for $n \geq 1$:
	\begin{enumerate}[resume,label=(\thelemma.\arabic*)]
		\item $w_n = w_{n-1}$ off $Y_n$;\label{wn=wn-1}
		\item $\|w_n - w_{n-1}\|_{\infty} \leq 5K_n g(R_n)$
		; \label{cvg}
		\item $w_n (x_i) = w_{n-1} (x_i)$ for all $0 \leq i \leq n$; \label{wnfixxi}
				\item $\|w_n' - w_{n-1}'\|_{L^2(-T, T)} \leq \frac{T_n^2}{16\left( 1+ \|\tilde{w}'\|_{L^2(-T, T)}\right)}$; \label{intwn'}
		\item $|w_n'(t)| \leq |w_{n-1}'(t)| + 2^{-n}$ for  almost every $ t\notin  [x_n - \tau_n, x_n + \tau_n]$; and\label{wn'-wn-1'} 
		\item $|w_n''(t)| \leq |w_{n-1}''(t)| + 2^{-n}$ for almost every $ t\notin  [x_n - \tau_n, x_n + \tau_n]$\label{wn''-wn-1''}. 
	\end{enumerate}
\end{lemma}
\begin{proof} 
	For the case $n = 0$, we easily check that setting $w_0 = \tilde{w}_0$ satisfies all the required conditions. Condition~\ref{wn=twn} is trivial for $\tau_0 = T$ and $\rho_0 = 0$; and~\ref{wn'lip} follows from~\eqref{alpha}. \ref{lipwn} is evident from the definition of $\tilde{w}$. Condition~\ref{lipwn} is evident from the definition of $\tilde{w}$, since $w_0 (x_0) = \tilde{w} (0) = 0$ and $\theta_0 \geq 1$. 
		Conditions~\ref{wn'bd} and~\ref{wn''bd} are given precisely by~\eqref{Kn}, since the inequality in~\eqref{Kn} holds for $|t- x_{1}| \leq \sigma_{1}$, as observed in~\eqref{sigmause}.
	
	Suppose for $n \geq 1$ that we have constructed $w_i$ as claimed for all $0 \leq i \leq n-1$.  We demonstrate how to insert a copy of $\tilde{w}_n$ into $w_{n-1}$.  We introduce in this proof a number of variables, e.g.~$m$,  which only appear in this inductive step.  Although they do of course depend on $n$, we do not index them as such, since they are only used while $n$ is fixed.
	
	Condition~\ref{t1} implies that $T_n \leq \sigma_n \leq |x_n - x_i|/2$ for all $0 \leq i \leq n-1$, and so that $x_i \notin Y_n$. Thus $w_{n-1}'$ exists and is Lipschitz on $Y_n$ by inductive hypothesis~\ref{wn'lip}.  Let $m \df w_{n-1}'(x_n)$, so $|m| \leq K_n$ by inductive hypothesis~\ref{wn'bd}.  On some yet smaller subinterval $[x_n - \tau_n, x_n + \tau_n]$ of $Z_n$ we aim to replace $w_{n-1}$ with a copy of $\tilde{w}_n$, connecting this with $w_{n-1}$ off $Y_n$ without increasing too much either the first or second derivatives, hence the choice of $R_n$ as very much smaller than $T_n$.  Moreover we want to preserve a continuous first derivative.  Hence we displace $w_{n-1}$ by a $C^1$ function---dealing with each side of $x_n$ separately---so that on each side of $x_n$ we approach $x_n$ on an affine function of gradient $m$ (a different function on each side, in general), which we then connect up with $\tilde{w}_n$ at a point where $\tilde{w}_n'=m$.   Because we need careful control over the first and second derivatives, it is easiest to construct explicitly the cut-off function we in effect use.
	
	Now, $\limsup_{t \downarrow 0}\tilde{w}'(t) = + \infty$ and $\liminf_{t \downarrow 0}\tilde{w}'(t) = -\infty$, so, recalling that $\tilde{w}'$ is a continuous and even function, we can  find $\tau_n \in (0, R_n]$ such that $ \tilde{w}' (\pm \tau_n) = m$. 
	
	We now construct the cut-off functions $\chi_-$ and $\chi_+$ that we will use on the left and right of $x_n$ respectively.  Additional constants and functions used in the construction are labelled similarly.
	
	Let $\delta_{\pm} \df m-w_{n-1}'(x_n \pm R_n)$, so by inductive hypothesis~\ref{wn''bd} we see that 
	\begin{equation}
	\label{moddelta}
	|\delta_{\pm}| = |w_{n-1}'(x_n) - w_{n-1}'(x_n \pm R_n)| \leq \|w_{n-1}''\|_{L^{\infty}(Y_n)} R_n 
	\leq K_n R_n. 
	\end{equation}  
	Define 
	\[
	c_{\pm}\df w_{n-1}(x_n) - w_{n-1}(x_n \pm R_n) +  \tilde{w} ( \pm \tau_n) \pm m (R_n - \tau_n).
	\]
	The point is that the functions $t \mapsto m(t - (x_n \pm R_n)) + w_{n-1}(x_n \pm R_n)+ c_{\pm}$ are affine functions with gradient $m$ which take the values $w_{n-1}(x_n \pm R_n) + c_{\pm}$ at $t = (x_n \pm R_n)$ and the values $w_{n-1}(x_n) + \tilde{w} ( \pm \tau_n)$ at $t = (x_n \pm \tau_n)$.
	By an application of the mean value theorem, the definition of $\tilde{w}$,  and the inductive hypothesis~\ref{wn'bd}, we have, recalling that $|m| \leq K_n$, that
	\begin{align} 
	|c_{\pm}| 
	& \leq  |w_{n-1}(x_n) - w_{n-1}(x_n \pm R_n)|+  | \tilde{w} ( \pm \tau_n)| + |m| |R_n -\tau_n| \nonumber \\
	& \leq \|w_{n-1}'\|_{L^{\infty}(Y_n)}R_n + g(\tau_n) + K_nR_n \nonumber \\
	& \leq K_n R_n + g(\tau_n) +K_n R_n \nonumber\\
	&\leq 3K_n g(R_n),
	\label{gamma} 
	\end{align}
	using also that $g(R_n) \geq R_n$ and that $K_n \geq 1$. Now let 
	\[
	d_{\pm} \df \frac{4}{T_n}\left(\pm \frac{\delta_\pm}{2}(R_n - T_n/2) - c_{\pm}\right),
	\]
	and define the piecewise affine functions $q_{\pm} \colon [-T, T] \to \mathbb{R}$ by stipulating  
	\[
	q_{\pm}(x_n \pm T_n) = 0 = q_{\pm} (x_n \pm T_n/2), \ q_{\pm}(x_n \pm 3T_n/4) = \pm d_{\pm},
	\]
	and 
	\[
	q_-(t) = 
	\begin{cases} 
	0 & t \leq x_n - T_n, \\
	\delta_- & t \geq x_n - R_n,\\
	\textrm{affine} & \textrm{otherwise};
	\end{cases}
	\ \textrm{and}\
	q_+(t) = 
	\begin{cases}
	\delta_+ & t \leq x_n + R_n,\\
	0 & t \geq x_n + T_n,\\
	\textrm{affine} & \textrm{otherwise}.
	\end{cases}
	\] 
	These $q_{\pm}$ will be the derivatives of the cut-off functions we will use. 
	So by definition of $d_{\pm}$,
	\begin{align}
	\int_{-T}^{x_n - R_n} q_{-}(t) \, dt  
	& = \int_{x_n - T_n}^{x_n - R_n} q_{-}(t) \, dt = \frac{1}{2}\left(-\frac{T_n d_{-}}{2} + (T_n/2 - R_n)\delta_{-}\right) =
	c_{-}, \label{intfl}
	\shortintertext{and}
	\int_{x_n + R_n}^{T}q_{+}(t) \, dt 
	& = \int_{x_n + R_n}^{x_n + T_n}q_+(t) \,dt = \frac{1}{2}\left(\delta_+ (T_n/2 - R_n) + \frac{d_+ T_n}{2}\right) =  -c_+.\label{intfr}
	\end{align}
	We need bounds on the first and second derivatives of the cut-off functions we will use, so we establish appropriate bounds on $q_{\pm}$ and $q_{\pm}'$.  Now, $\|q_{\pm}\|_{\infty} = \max\{ |\delta_{\pm}|, |d_{\pm}|\}$. 
	Note that~\eqref{moddelta}  and~\eqref{gamma} imply, using again that $g(R_n) \geq R_n$, that
	\begin{align}
	|d_{\pm}|  \leq \frac{4}{T_n}\left(\frac{|\delta_{\pm}|}{2}(T_n/2-R_n) + |c_{\pm}| \right)
	& \leq \frac{4}{T_n}\left( \frac{T_n K_n R_n}{4} + 3K_n g(R_n)\right) \nonumber\\
	&= K_nR_n + \frac{12K_n g(R_n)}{T_n}\nonumber\\
	& \leq \frac{13K_n g(R_n)}{T_n}.
	\label{iia}
	\end{align}  
	So, comparing with~\eqref{moddelta}, we have that
	\begin{equation}
	\label{epsilonlb}
	\|q_{\pm}\|_{\infty} \leq \frac{13 K_n g(R_n)}{T_n}. 
	\end{equation}  
	Also, $q_{\pm}'$ exists almost everywhere and satisfies $\|q_{\pm}'\|_{L^{\infty}(-T, T)} = \max\{\frac{4 |d_{\pm}| }{T_n}, \frac{|\delta_{\pm}|}{T_n/2 - R_n}\}$. Note firstly by~\eqref{iia} and~\ref{r3} that 
	\[
	\frac{4|d_{\pm}|}{T_n} \leq \frac{4}{T_n} \left(\frac{13 K_n g(R_n)}{T_n}\right) = \frac{52K_n g(R_n)}{T_n^2} \leq 2^{-n},
	\]
	and secondly that since~\ref{r3} in particular implies that $R_n \leq T_n/4$, using~\eqref{moddelta} and~\ref{r3} we see that 
	\[
	\frac{|\delta_{\pm}|}{(T_n / 2) - R_n} \leq \frac{4 K_n R_n}{T_n} \leq 2^{-n}.
	\]
	Hence
	\begin{equation}
	\label{zetal}
	\|q_{\pm}'\|_{L^{\infty}(-T, T)} 
	\leq 2^{-n}.
	\end{equation}
	
	We can now define our cut-off functions $\chi_{\pm} \colon [-T, T] \to \mathbb{R}$ by 
	\[
	\chi_- (t) = \int_{-T}^t q_-(s) \, ds,\ \textrm{and}\ \chi_+(t) = c_+ - \delta_+((x_n + R_n) - (-T))+ \int_{-T}^t q_+(s)\, ds.
	\]
	Then $\chi_{\pm} \in C^1 (-T, T)$ are such that $\chi_{\pm}' = q_{\pm}$ everywhere, $\chi_{\pm}'' = q_{\pm}'$ almost everywhere, and, by~\eqref{intfl} and~\eqref{intfr}, and the definition of $q_{\pm}$, we have that
	\[
	\chi_{\pm}(x_n \pm T_n) = 0, \ \chi_{\pm} (x_n \pm R_n) = c_{\pm}, \ \chi_{\pm}'(x_n \pm R_n) = q_{\pm}(x_n \pm R_n) = \delta_{\pm}. 
	\]
	
	We can now define $w_n \colon [-T, T]\to \mathbb{R}$ by 
	\[
	w_n(t) = 
	\begin{cases}
	w_{n-1}(t) + \chi_{-} (t) & t \leq x_n - R_n, \\
	m(t - (x_n - R_n)) + w_{n-1}(x_n - R_n)  + c_{-} & x_n - R_n < t <x_n - \tau_n, \\
	w_{n-1}(x_n) +  \tilde{w}_n (t)  & x_n - \tau_n \leq t \leq x_n + \tau_n, \\
	m(t - (x_n + R_n)) + w_{n-1}(x_n + R_n) + c_{+} & x_n + \tau_n < t < x_n + R_n, \\
	w_{n-1}(t) + \chi_+ (t) & x_n + R_n \leq t.
	\end{cases}
	\]
	We see that $w_n$ is continuous by construction.  Condition~\ref{wn=twn} is immediate, with $\tau_n$ as defined, and $\rho_n = w_{n-1}(x_n)$.  We note that since, by the definitions of $q_{\pm}$,  $\chi_{-}(t) = 0$ for $t \leq x_n - T_n$, and $\chi_{+}(t) = 0$ for $t \geq x_n + T_n$, we have that $w_n = w_{n -1}$ off $Y_n$, as required for~\ref{wn=wn-1}.  To check~\ref{wnfixxi}, we let $ 0 \leq i \leq n$.  If $ i \leq n-1$, then $x_i \notin Y_n$ since $T_n \leq \sigma_n$, so $w_n (x_i) = w_{n-1}(x_i) $ by~\ref{wn=wn-1}, which we have just checked for $n$.  We see directly from the construction that $w_n (x_n) = w_{n-1}(x_n)$ since $\tilde{w}_n (x_n) = 0$, as required for the full result.
	
	We see that $w_n'$ exists off $\{x_i\}_{i=0}^n$ by inductive hypothesis~\ref{wn'lip} and by construction: the values of $\delta_{\pm}$ and $\tau_n$ were chosen precisely so that the derivatives agree across the joins in the definition of $w_n$.  The derivative is given by
	\[
	w_n'(t) = 
	\begin{cases} 
	w_{n-1}'(t) + q_{-} (t) & t \in [ -T, x_n - R_n] \setminus \{ x_i \}_{i =0}^{n-1}, \\
	m & x_n -R_n < t < x_n - \tau_n, \\
	\tilde{w}_n'(t)& x_n - \tau_n \leq t < x_n, \ x_n < t \leq x_n + \tau_n, \\
	m & x_n + \tau_n < t < x_n + R_n, \\
	w_{n-1}'(t) + q_{+} (t) & t \in [x_n + R_n , T]  \setminus \{ x_i \}_{i =0}^{n-1}. 
	\end{cases}
	\] 
	This is locally Lipschitz on $(-T, T) \setminus \bigcup_{i=0}^n\{x_i\}$ by inductive hypothesis~\ref{wn'lip} on $w_{n-1}'$ and since $q_{\pm}$ are Lipschitz, and by~\eqref{alpha}, as required for condition~\ref{wn'lip}.
	
	We have constructed $w_n$ so that except on $[x_n - \tau_n, x_n + \tau_n]$, the derivative is comparable with that of $w_{n-1}$, wherever both exist, which is everywhere except for the points $\{x_i\}_{i=0}^{n}$.  For $t \notin Z_n \cup \{x_i\}_{i = 0}^{n-1}$, we see by~\eqref{epsilonlb} that 
	\[
	|w_n '(t) - w_{n-1}'(t)| \leq | q_{\pm}(t)| \leq \frac{13K_n g(R_n)}{T_n}.
	\]
	For $t \in Z_n \setminus [x_n - \tau_n, x_n + \tau_n]$, we note that since $R_n \leq T_n \leq \sigma_n$, we have that $|x_n - t| \leq \sigma_n$, and so we may use inductive hypothesis~\ref{wn''bd} to see that
	\begin{align}
	|w'_n (t) - w_{n-1}'(t)| = |m - w'_{n-1}(t)|  = |w'_{n-1}(x_n) - w'_{n-1}(t)| 
	& \leq \|w_{n-1}''\|_{L^{\infty}(Y_n)}|t - x_n| \nonumber\\
	& \leq K_n R_n. \label{Znwn''}
	\end{align}
	Since~\ref{r3} implies that $K_nR_n \leq \frac{13K_n g(R_n)}{T_n} \leq 2^{-n}$, in particular we gain condition~\ref{wn'-wn-1'}.
	
	The pointwise comparison we have just established between $w_n'$ and $w_{n-1}'$ fails in general on $[x_n - \tau_n, x_n + \tau_n]$.  On this set, in fact, the whole point of the construction is that $w_n'$ now equals $\tilde{w}_n' = \tilde{w}'( \cdot - x_n)$, which oscillates between arbitrarily large positive and negative values.  On the other hand, $w_{n-1}'$, since it is locally Lipschitz away from the points $\{ x_i \}_{i = 1}^{n-1}$, may be regarded as basically constant on $[x_n - \tau_n, x_n + \tau_n]$, at least compared to the behaviour of $w_n'$.  We chose $R_n$ to be so small that despite this (large!) discrepancy on $[x_n - \tau_n, x_n + \tau_n] \subseteq Z_n$, the two derivatives are close in $L^2(-T, T)$, as stated in~\ref{intwn'}, which we now check. 
	First note that, using the definition of $w_n$ and~\ref{wn=wn-1} (which we have checked for $n$),
	\begin{align*}
		\int_{-T}^{T} |w_n' (t) - w_{n-1}'(t)|^2 \, dt
	& = \int_{Y_n} |w_n ' (t)- w_{n-1}'(t)|^2  \, dt \\
	& = \int_{x_n - T_n}^{x_n - R_n} |q_{-} (t) |^2  \, dt + \int_{Z_n \setminus [x_n - \tau_n, x_n + \tau_n]} |m - w_{n-1}'(t)|^2  \, dt \\
	& \phantom{=} {}+ \int_{x_n - \tau_n}^{x_n + \tau_n} |\tilde{w}_n' (t) - w_{n-1}'(t)|^2  \, dt + \int_{x_n + R_n}^{x_n + T_n} |q_{+} (t)|^2  \, dt.
	\end{align*}
	Now, by~\eqref{epsilonlb},
	\begin{align*}
	\lefteqn{
			\int_{x_n - T_n}^{x_n - R_n} |q_{-} (t) |^2  \, dt  + \int_{x_n + R_n}^{x_n + T_n} |q_{+}(t)|^2 \, dt 
			}\\
	& \leq \int_{x_n - T_n}^{x_n - R_n} \frac{(13K_n g(R_n))^2}{T_n^2}\, dt + \int_{x_n + R_n}^{x_n + T_n}  \frac{(13K_n g(R_n))^2}{T_n^2}\, dt \\
	& \leq \frac{2T_n 169g(R_n)^2 K_n^2}{T_n^2}\\
	& \leq \frac{338g(R_n) K_n^2}{T_n},
	\end{align*}
	using also that $g(R_n) \leq 1$. 
	Further, by~\eqref{Znwn''}, we have that
	\begin{align*}
	\int_{Z_n \setminus [x_n- \tau_n, x_n + \tau_n]}  |m - w_{n-1}'(t)|^2  \, dt  \leq \int_{Z_n \setminus [x_n- \tau_n, x_n + \tau_n]} (K_n R_n)^2 
	& \leq  2 R_n (K_n R_n)^2 \\
	&\leq 2 R_n K_n^2.
	\end{align*}
	Finally, by inductive hypothesis~\ref{wn'bd}, we have that 
	\begin{align*}
	\int_{x_n - \tau_n}^{x_n + \tau_n} |\tilde{w}_n' (t) - w_{n-1}'(t)|^2  \, dt 
	&\leq \int_{x_n - \tau_n}^{x_n + \tau_n} 2\left(|\tilde{w}_n' (t)|^2 + |w_{n-1}'(t)|^2\right)  \, dt \\
	& \leq 2 \! \left(\int_{- \tau_n}^{ \tau_n} |\tilde{w}' (t)|^2 \, dt + \int_{-\tau_n}^{\tau_n} K_n^2\, dt\right)\\
	& \leq 2 \! \int_{-R_n}^{R_n} |\tilde{w}' (t)|^2  \, dt  + 4 K_n^2 \tau_n \\
	& \leq 2 \! \int_{-R_n}^{R_n} |\tilde{w}'(t)|^2  \, dt + 4 K_n^2 R_n.
	\end{align*}
	Combining these estimates, using~\ref{r1},~\ref{r3}, and that $g(R_n) \geq R_n$, we see that
	\begin{align*}
	\lefteqn{\int_{-T}^{T} |w_n' (t) - w_{n-1}'(t)|^2 \, dt}\\ 
	& \leq \frac{338g(R_n) K_n^2}{T_n} + 2 R_n K_n^2 + 2 \!\int_{-R_n}^{R_n} |\tilde{w}'(t)|^2  \, dt + 4 K_n^2 R_n\\
	& \leq \frac{344 g(R_n)K_n^2}{T_n} + \frac{T_n^4}{512\left( 1 + \|\tilde{w}'\|_{L^2(-T, T)}\right)^2}\\
	& \leq \frac{T_n^4}{512\left( 1 + \|\tilde{w}'\|_{L^2(-T, T)}\right)^2} + \frac{T_n^4}{512\left( 1 + \|\tilde{w}'\|_{L^2(-T, T)}\right)^2}\\
	& = \frac{T_n^4}{256\left( 1 + \|\tilde{w}'\|_{L^2(-T, T)}\right)^2}.
	\end{align*} 
	Taking square roots gives condition~\ref{intwn'}.
	Now, $w_n''$ exists almost everywhere by inductive hypothesis~\ref{wn''bd} and by construction, and, where it does, is given by
	\[
	w_n''(t) = 
	\begin{cases}
	w_{n-1}''(t) + q_{-}'(t) & t < x_n - R_n, \\
	0 & x_n - R_n < t < x_n - \tau_n, \\
	\tilde{w}_n''(t) & x_n - \tau_n < t < x_n, \ x_n < t <x_n + \tau_n, \\
	0 & x_n + \tau_n < t < x_n + R_n,\\
	w_{n-1}''(t) + q_{+}'(t) & x_n + R_n < t.
	\end{cases}
	\]
	Thus by~\eqref{zetal}, for almost every $ t \in (-T, T) \setminus Z_n$, we have that
	\[
	|w_n''(t)| \leq |w_{n-1}''(t)| + |q_{\pm}'(t)| \leq |w_{n-1}''(t)| + 2^{-n}.
	\]
	Condition~\ref{wn''-wn-1''} follows, since $w_n'' = 0$ on $Z_n \setminus [x_n - \tau_n, x_n + \tau_n]$.
	
	We now check~\ref{wn'bd} and~\ref{wn''bd}.  Suppose $|t - x_{n+1}| \leq \sigma_{n+1}$.  Then by~\eqref{sigmause} the inequality in~\eqref{Kn} holds, in particular
	\[
	\sum_{i=0}^n (|\tilde{w}_i''(t)| + |\tilde{w}_i'(t)|)\leq K_{n+1},
	\]
	for all $0 \leq i \leq n$, precisely by choice of $K_{n+1}$.
	
	Let $0 \leq k \leq n$ be such that $t \in Y_k \setminus \bigcup_{i=k+1}^n Y_i$.  Then by inductive hypothesis~\ref{wn=wn-1} for $k+1, \dots, n$ (note we have checked this for $n$),
	we have that $w_n = w_k$ on a neighbourhood of $t$, so $w_n'(t) = w_k'(t)$ and $w_n''(t) = w_k ''(t)$ where both sides exist, i.e.\ almost everywhere.  We have to distinguish the cases of when $t$ lies in $[x_k - \tau_k, x_k + \tau_k]$ and when it does not.
	
	If $t \notin [x_k - \tau_k , x_k + \tau_k]$, then by inductive hypotheses~\ref{wn'-wn-1'} (we have checked this for $k = n$) and~\ref{wn'bd} (since $t \in Y_k$), and by~\eqref{Kn2}, we have that
	\[
	|w_n'(t) | = |w_k'(t)| \leq |w_{k-1}'(t)| + 2^{-k} \leq K_k + 1 \leq K_{n+1},
	\]
	as required for~\ref{wn'bd}. Similarly by inductive hypotheses~\ref{wn''-wn-1''} (we have checked this for $k = n$) and~\ref{wn''bd}, and by~\eqref{Kn2}, we have almost everywhere that
	\[
	|w_n''(t)| = |w_k''(t) | \leq |w_{k-1}''(t)| + 2^{-k} \leq K_{k} + 1  \leq K_{n+1},
	\]
	as required for~\ref{wn''bd}.
	
	If $ t \in [x_k - \tau_k , x_k + \tau_k]$, observe first that $t \neq x_k$ since $|t - x_{n+1}| \leq \sigma_{n+1}$.  So by inductive hypothesis~\ref{wn=twn} (we have checked this for $k = n$),  we have, using~\eqref{Kn}, that 
	\begin{gather*}
	|w_n'(t)| = |w_k'(t)| = | \tilde{w}_k '(t)|  \leq \sum_{i=0}^k |\tilde{w}_i '(t)| \leq \sum_{i=0}^n |\tilde{w}_i'(t)| \leq  K_{n+1},
	\shortintertext{and, if $t \in (x_k - \tau_k, x_k + \tau_k)$ (we claim nothing about $w_n''$ at the endpoints $x_k \pm \tau_k$), almost everywhere we have that}
	|w_n''(t)| = |w_k''(t)| = | \tilde{w}_k ''(t)|  \leq \sum_{i=0}^k |\tilde{w}_i ''(t)| \leq \sum_{i=0}^n |\tilde{w}_i''(t)| \leq  K_{n+1},
	\end{gather*}
	as required.  So~\ref{wn'bd} and~\ref{wn''bd} hold in all cases.
	
	We now move towards checking~\ref{cvg}.  Now observe that~\eqref{iia} and~\eqref{moddelta} imply that on $[-T, x_n - R_n]$, we have, by definition, that 
	\begin{align*}
	|\chi_{-}(t)| \leq \int_{-T}^{x_n - R_n} | q_{-}(s)|\, ds 
	& \leq  \frac{1}{2} \left(\frac{T_n}{2} |d_{-}| + (T_n/2-R_n)|\delta_{-}|\right)\\
	& \leq \frac{T_n}{4}\frac{13 K_n g(R_n)}{T_n} + \frac{T_n K_nR_n}{2} \\
	& \leq \frac{13 K_n g(R_n)}{4} + \frac{K_n g(R_n)}{2}\\
	& \leq 4K_n g(R_n).
	\end{align*}
	The same estimate holds for $\chi_{+}$ on $[ x_n + R_n,T]$: we note first by~\eqref{intfr} that
	\begin{align*} 
	\chi_{+}(t) & =  c_{+} - \delta_{+}((x_n + R_n) + T) + \int_{-T}^t q_{+} (s)\, ds\\
	& =  c_{+} + \int_{x_n + R_n}^t q_{+}(s) \, ds \\
	& =  - \int_{x_n + R_n}^{T} q_{+}(s)\, ds + \int_{x_n + R_n}^t q_{+} (s)\, ds\\
	& = - \int_t^{ T} q_{+}(s)\, ds,
	\end{align*}
	and hence, since $|\chi_{+}| \leq \int_{x_n + R_n}^{T}|q_{+}|$ on $[x_n + R_n, T]$,  we can estimate $|\chi_{+}|$ as we estimated $|\chi_{-}|$ on $[-T, x_n - R_n]$ above. So, we have for all $t \in [-T, T] \setminus Z_n$ that
	\[
	|w_n (t) - w_{n-1}(t)| = |\chi_{\pm}(t)| \leq 4K_n g(R_n).
	\]
	By inductive hypothesis~\ref{wn'bd} and~\eqref{gamma}, we have for $t \in Z_n \setminus [x_n - \tau_n, x_n + \tau_n]$ that  
	\begin{align*}
	|w_n (t) - w_{n-1}(t)| 
	&\leq |m(t - (x_n \pm R_n))| + |w_{n-1}(x_n \pm R_n) - w_{n-1}(t)| + |c_{\pm}| \\
	& \leq |m| |(t - ( x_n \pm R_n))| + \|w_{n-1}'\|_{L^{\infty}(Y_n)}| t- ( x_n \pm R_n)| + | c_{\pm}|\\
	&\leq  K_n R_n + K_n R_n + 3K_n g(R_n) \\
	&\leq 5K_n g(R_n).
	\end{align*}
	Finally for $x_n - \tau_n \leq t \leq x_n + \tau_n $, by inductive hypothesis~\ref{wn'bd}, the definition of $\tilde{w}$, and the monotonicity of $g$, we have that
	\begin{align*}
	|w_n (t) - w_{n-1}(t)|  = |(\tilde{w_n}(t)+ w_{n-1}(x_n)) - w_{n-1}(t)| 
	& \leq | \tilde{w}_n (t)| + |w_{n-1}(x_n) - w_{n-1}(t)| \\
	& \leq g(\tau_n) + \|w_{n-1}'\|_{L^{\infty}(Y_n)}|x_n - t| \\
	& \leq g(R_n) + K_n R_n\\
	& \leq 2K_n g(R_n).
	\end{align*}
	Hence we have that $\|w_n - w_{n-1}\|_{\infty} \leq 5 K_n g(R_n)$, as required for~\ref{cvg}.
	
	We can now check~\ref{lipwn}.  First consider $0 \leq i \leq n-1$.  The result is immediate by inductive hypothesis if $t \notin Y_n$, by~\ref{wnfixxi} and~\ref{wn=wn-1}, both of which we have checked for $n$.    So suppose $ t \in Y_n $.  Then $|g_i (t)| \geq \eta_n$ by~\eqref{Ynfarfromxi} and~\eqref{etan}.  Condition~\ref{cvg}, which we have checked for $n$, and~\ref{r3} imply that 
	\[
	|w_n (t) - w_{n-1}(t) | \leq 5 K_n g(R_n)  \leq 2^{-n} \eta_n  \leq 2^{-n}|g_i(t)| \leq 2^{-n} \theta_i |g_i(t)|,
	\]
	since $\theta_i \geq 1$.  Condition~\ref{wnfixxi} and inductive hypothesis~\ref{lipwn} imply that
	\[
	|w_{n-1} (t) - w_n (x_i)|  = |w_{n-1}(t) - w_{n-1}(x_i)|  \leq \left( 2- 2^{-(n-1)}\right) \theta_i |g_i (t)|.
	\]
	So
	\begin{align*} 
	|w_n (t) - w_n (x_i)| 
	& \leq |w_n (t) - w_{n-1}(t)| + |w_{n-1}(t) - w_{n-1}(x_i)| \\
	& \leq 2^{-n} \theta_i |g_i (t)| + \left( 2- 2^{-(n-1)}\right)\theta_i |g_i(t)| \\
	& = (2- 2^{-n})\theta_i |g_i (t)|.
	\end{align*}
	It just remains to check~\ref{lipwn} in the case $i =n$. We first show that for all $t \in [-T, T]$  we have chosen $\theta_n$ such that 
	\begin{equation}
	\label{wn-1 under gn}
	|w_{n-1}(t) - w_{n-1}(x_n)| \leq \theta_n |g_n (t)| / 2.
	\end{equation}
	This is the motivating factor behind the choice of $\theta_n$: blowing up the graph of $w_{n-1}(x_n) \pm |g_n| = w_n (x_n) \pm |g_n|$ so that it encloses that of $w_{n-1}$. Now, for $ |t- x_n| \leq \sigma_n$, we have by inductive hypothesis~\ref{wn'bd} and~\eqref{thetan}, since $\two 1/|t- x_n| \geq 2$, that 
	\begin{align*}
	|w_{n-1}(t) - w_{n-1}(x_n)| \leq \|w_{n-1}'\|_{L^{\infty}(x_n - \sigma_n, x_n + \sigma_n)} | t- x_n| \leq K_n |t-x_n| 
	& \leq \theta_n |t- x_n |\\
	& \leq \theta_n |g_n (t)|/2.
	\end{align*}
	If $|t - x_n | \geq \sigma_n$, then by inductive hypothesis~\ref{cvg},~\ref{r3}, and~\eqref{thetan}, and since $T$ was chosen small enough such that $g(T) \leq 1 \leq 2 \leq \two 1/2T$, we have that
	\allowdisplaybreaks{
		\begin{align*} 
	|w_{n-1}(t) - w_{n-1}(x_n)| 
	& \leq |w_{n-1}(t) - w_0 (t)| + |w_0 (t) - w_0 (x_n)| + |w_0 (x_n) - w_{n-1}(x_n)|\\
	& \leq  2\|w_0\|_{\infty} + 2 \|w_{n-1} - w_0\|_{\infty}  \\
	& \leq 2 \left(g(T) +  \sum_{i= 1}^{n-1}\|w_i - w_{i-1}\|_{\infty} \right)\\
	& \leq 2 \left( g(T) + \sum_{i=1}^{n-1} 5K_i g(R_i) \right)\\
	&\leq 2 \left(g(T)+ \sum_{i=1}^{n-1}2^{-i} \right) \\
	& \leq 4 \\
	& \leq \theta_n \sigma_n  \\
	& \leq \frac{\theta_n \sigma_n (\two 1/2T )}{2} \\
	& \leq \frac{\theta_n |t - x_n| (\two 1/ 2 T)}{2}\\
	& \leq \frac{\theta_n | (t-x_n) \two 1/|t-x_n| |}{2} \\
	& =\theta_n |g_n (t)|/2,
	\end{align*}
}
	as claimed. 
	
	To check~\ref{lipwn} in this final case, suppose first that $t \in [x_n - \tau_n, x_n + \tau_n]$.  Then by~\ref{wn=twn},  and the definition of $\tilde{w}$ we have, since $\theta_n \geq 1$, that 
	\[
	|w_n (t) - w_n (x_n)| = |\tilde{w}_n (t) - \tilde{w}_n (x_n)| \leq  |g_n (t)| \leq ( 2- 2^{-n})\theta_n  |g_n (t)|.
	\]  
	To deal with the case $x_n - R_n \leq t< x_n - \tau_n$, we note first that the  condition is satisfied at the endpoints of the interval.  That it holds for $t = x_n - \tau_n$ has just been established. Using the definition of $w_n$,~\ref{wnfixxi}, inductive hypothesis~\ref{wn'bd},~\eqref{gamma}, and~\eqref{thetan} we see that
	\begin{align*} 
	|w_n (x_n - R_n) - w_n (x_n)| 
	&= |w_{n-1}(x_n - R_n) + c_- - w_{n-1} (x_n )| \\
	& \leq \|w_{n-1}'\|_{L^{\infty}(Y_n)}R_n + |c_{-}| \\
	&\leq K_n R_n + 3 K_n g(R_n) \\
	&\leq 4K_n g(R_n) \\
	&\leq \theta_n g(R_n)\\
	& \leq ( 2 - 2^{-n}) \theta_n |g_n ( x_n - R_n)|.
	\end{align*}
	So the condition holds at $x_n - R_n$ and $x_n - \tau_n$.  Since $w_n$ is defined to be affine between these points, and $|g_n|$ is concave on $[-T, x_n]$, the result holds for all $t \in [x_n -R_n, x_n - \tau_n]$.  Similarly the result holds for all $t \in [x_n + \tau_n, x_n + R_n]$.  Finally we have to consider $t \notin [x_n - R_n, x_n + R_n]$.  In this case we  have by monotonicity of $g$ that $g_n (t) \geq g (R_n)$, and so we see using~\ref{wnfixxi},~\ref{cvg},~\eqref{wn-1 under gn}, and~\eqref{thetan} that
	\begin{align*} 
	|w_n (t) - w_n (x_n)| 
	& \leq |w_n (t) - w_{n-1}(t)| + |w_{n-1}(t) - w_{n} (x_n )| \\
	& \leq \|w_n - w_{n-1}\|_{\infty} + |w_{n-1}(t) - w_{n-1}(x_n)|\\
	& \leq 5 K_n g(R_n) + \theta_n |g_n (t)|/2 \\
	& \leq 5K_n |g_n (t)| +  \theta_n|g_n (t)|/2 \\
	& \leq \theta_n |g_n(t)| / 2 + \theta_n |g_n (t)| / 2 \\
	& \leq (2 -2^{-n})\theta_n  |g_n (t)|.
	\end{align*}  
	Thus~\ref{lipwn} holds for all $t \in [-T, T]$ as claimed.
\end{proof}
We now show that this sequence converges to some $w \in W^{1,2}(-T, T)$.  This $w$ will be our minimizer.
\begin{lemma}
	\label{wncvg}  
	The sequence $\{w_n\}_{n=0}^{\infty}$ converges uniformly to some function $w \in W^{1,2}(-T, T)$ such that for all $n \geq 0$,
	\begin{enumerate}[label=(\thelemma.\arabic*)]
		\item $w(x_i) = w_n (x_i)$ for all $0 \leq i \leq n+1$; \label{wfixxi}
		\item $ \|w- w_n\|_{\infty} \leq 10 K_{n+1} g(R_{n+1})$;\label{w-wn}
		\item $\|w' - w_n'\|_{L^2(-T, T)} \leq \frac{T_{n+1}^2}{8\left(1 + \| \tilde{w}'\|_{L^2(-T, T)}\right)}$; \label{wn'tow'in L2}and
		\item $|w(t) - w(x_n)| \leq 2 \theta_n|g_n (t)|$ for all $ t \in [-T, T]$. \label{lipw}
	\end{enumerate}
\end{lemma}
\begin{proof} 
	Let $m \geq n+1 \geq 1$. \ref{r3} in particular implies that $K_m g(R_m) \leq g(R_{m-1})/2$, so combining with~\ref{cvg}, we see that
	\begin{align*}
	\|w_m - w_n\|_{\infty}
	& \leq \|w_m - w_{m - 1}\|_{\infty} +{} \cdots {}+ \|w_{n +1} - w_n\|_{\infty} \\
	& \leq 5 (K_m g(R_m) +{} \cdots {}+ K_{n+1}g(R_{n+1})) \\
	&\leq 5 \left(2^{-(m - (n+1))} +{} \cdots {}+ 1\right)K_{n+1}g(R_{n+1}) \\
	& \leq 10K_{n+1}g(R_{n+1}).
	\end{align*}  
	Hence, since~\ref{r3} certainly implies that this tends to $0$ as $n \to \infty$,  the sequence $\{w_n\}_{n=0}^{\infty}$ is uniformly Cauchy, and so converges uniformly to  some $w \in C(-T, T)$, which satisfies~\ref{w-wn}. Conditions~\ref{wfixxi} and~\ref{lipw} follow directly by taking limits in conditions~\ref{wnfixxi} and~\ref{lipwn} respectively.
	
	Now, by~\ref{intwn'} and~\ref{t1}, we have that
	\begin{align} 
	\|w_m' - w_n'\|_{L^2(-T, T)}
	& \leq	\|w_m' - w_{m-1}' \|_{L^2(-T, T)} + \dots + \|w_{n+1}' - w_n'\|_{L^2(-T, T)} \nonumber \\
	& \leq \frac{T_m^2 + \dots + T_{n+1}^2}{16\left(1 + \|\tilde{w}'\|_{L^2(-T, T)}\right)} \nonumber \\
	& \leq \frac{T_{n+1}^2}{8 \left(1 + \|\tilde{w}'\|_{L^2(-T, T)}\right)},\label{L2cauchy}
	\end{align}
	and, since likewise $T_n \to 0$ as $n \to \infty$, $w_n'$ is Cauchy in $L^2(-T, T)$, and it follows that $w_n \to w$ in $W^{1,2}(-T, T)$, and that~\ref{wn'tow'in L2} holds.
\end{proof}
\subsection{Singularity}
Having defined the function $w \in W^{1,2}(-T,T)$, we now check that it exhibits the required oscillating behaviour around each point $x_n$.  The extra oscillations we added in to $w_n$ are small enough in magnitude and far enough from $x_n$ to preserve the behaviour of $w$ as being like that of $w_n$ and hence $\tilde{w}_n$ around $x_n$.  In particular, the limiting behaviour of the difference quotients of $w$ at $x_n$ is the same as that of the difference quotients of $\tilde{w}$ at $0$, for each $n \geq 0$.
\begin{lemma}  
	Let $n \geq 0$.  
	
	Then $\overline{D}w(x_n) = +\infty$ and $\underline{D}w(x_n) = -\infty$.
\end{lemma}
\begin{proof} 
	Let $m \geq  n +1$, and let $t \in [-T, T]$ be such that $| t - x_n| \leq T_m$.   Note that if $ t \in Y_i$ for $i \geq n+1$, we have, since~\ref{t1} implies that $T_i \leq \sigma_i$, that
	\[
	|x_n - x_i| \leq |x_n - t| + |t-x_i| \leq |x_n - t| + T_i \leq |x_n - t| + |x_n - x_i|/2,
	\]
	and hence, again by condition~\ref{t1},
	\begin{equation}
	\label{x}
	T_i \leq |x_n - x_i|/2 \leq|x_n - t| \leq T_m.
	\end{equation}
	Since the $T_i$ are decreasing, this implies that $i \geq m$. 
	
	If $t \notin Y_i$ for any $i \geq n+1$ then $w(t) = w_n (t)$ by~\ref{wn=wn-1}, and the following argument is trivial.  Otherwise choose the least $i \geq n+1$ such that $ t \in Y_i$, so $w_n (t) = w_{i-1}(t)$.  Then by~\ref{w-wn},~\ref{r3}, and~\eqref{x}, 
	\[
	|w(t) - w_n (t)| =  |w(t) - w_{i-1}(t)|\leq\|w-w_{i-1}\|_{\infty}  \leq 10K_{i} g(R_i) < 2^{-i}\,T_i \leq 2^{-i} |t- x_n|.
	\]
	Hence by~\ref{wfixxi} and since, by the above argument, $i \geq m$,  we have that
	\begin{equation}
	\label{diffquot}
	\left|\frac{w(t) - w(x_n)}{t-x_n} - \frac{w_n (t) - w_n (x_n)}{t-x_n} \right|  =  \left| \frac{w(t) - w_n (t)}{t-x_n} \right| \leq \frac{2^{-i}|t- x_n|}{|t-x_n|}  \leq
	2^{-m}.
	\end{equation}
	As $t \to x_n$, we may choose $m \to \infty$. Hence by~\ref{wn=twn} and the definition of $\tilde{w}_n$, 
	\begin{align*}
	\overline{D}w(x_n) = \overline{D}w_n(x_n) = \overline{D} \tilde{w}_n(x_n) 
	& = +\infty,\\
	\shortintertext{and}
	\underline{D}w(x_n) = \underline{D}w_n(x_n) = \underline{D}\tilde{w}_n (x_n) 
	& = -\infty.
	\qedhere
\end{align*}
\end{proof}
\subsection{Construction of the Lagrangian}
We now construct the Lagrangian which shall define the variational problem of which $w$ will be the unique minimizer. Our basic weight function $\tilde{\phi} \colon [-T, T] \times \mathbb{R} \to [0, \infty)$ will be given by
\[
\tilde{\phi}(t,y) =
\begin{cases}
0 & t = 0, \\
5 \psi(t)|g(t)| & |y| \geq 5 |g(t)|, \\
\psi(t)|y| & |y| \leq 5 |g(t)|.
\end{cases}
\]
We need some bound of the form $|\tilde{\phi}(t,y)| \leq c |g(t)| \psi(t)$ to ensure continuity of $\tilde{\phi}$; it turns out (see lemma~\ref{lemma1}) that sensitive tracking of $|y|$, which shall represent the distance of a putative minimizer from our constructed function $w$, only for $|y| \leq 5|g(t)|$ suffices in the proof of minimality.  Our function $\tilde{w}$ was constructed precisely so that~\eqref{wn''cont} and hence~\eqref{psi2} hold, and hence that this $\tilde{\phi}$ is continuous.

We in fact will find it useful to split $\tilde{\phi}$ into the summands by which we defined~$\psi$.  More precisely, we define for each $n \geq 0$ our translated weight functions $\tilde{\phi}_n^1, \tilde{\phi}_n^2\colon [-T, T] \times \mathbb{R} \to [0, \infty)$ as follows. We recall that we need extra weight only on $Y_n$, so we define for $k = 1,2$ and $(t, y) \in Y_n \times \mathbb{R}$,
\[
\tilde{\phi}_n^k (t,y) = 
\begin{cases}
0 & t= x_n, \\
5 \psi_n^k (t)\theta_n |g_n(t)| & |y| \geq 5 \theta_n |g_n(t)|, \\
\psi_n^k(t) |y| & |y| \leq 5\theta_n |g_n(t)|;
\end{cases}
\] 
and extend to a function on  $[-T, T] \times \mathbb{R}$ by setting it to continue constantly at the value attained at the endpoints of $Y_n$, i.e.\ defining for $(t,y) \in ([-T, T] \setminus Y_n) \times \mathbb{R}$ 
\[
\tilde{\phi}_n^k(t,y) = 
\begin{cases}
5\psi^k(T_n) \theta_n g(T_n) & |y| \geq 5\theta_n g(T_n), \\
\psi^k(T_n) |y|& |y| \leq 5\theta_n g(T_n).
\end{cases}
\]
Define  $\tilde{\phi}_n \colon [-T, T] \times \mathbb{R} \to [0, \infty)$ by $\tilde{\phi}_n (t,y) = \tilde{\phi}_n^1 (t,y) + \tilde{\phi}_n^2 (t,y)$, which is continuous by~\eqref{psi2}. 

We claim that for fixed $t \in [-T, T]$, for all $ n\geq 0$ and $k=1,2$, that 
\begin{gather*} 
\tilde{\phi}_n^k (t,y) \leq \tilde{\phi}_n^k(t, z)\ \textrm{whenever}\ |y| \leq |z|;\label{tphiinc}\\
\mathrm{Lip}(\tilde{\phi}_n^k(t,\cdot)) \leq \max\{\psi_n^k (t), \psi^k (T_n)\};\label{liptphi}\ \textrm{and}\\
\tilde{\phi}_n^k (t, 0) = 0.\label{tp0is0}
\end{gather*}
The last result is obvious, as are the other results for $t = x_n$.  Suppose $ t \in Y_n \setminus \{x_n\}$.  First consider the case in which $|y| \leq |z| \leq 5\theta_n |g_n(t)|$.  Then 
\begin{align*} 
\tilde{\phi}_n^k (t, z) - \tilde{\phi}_n^k(t, y) 
&  =   \psi_n^k (t)|z| - \psi_n^k(t)  |y| \geq 0;\\
\shortintertext{and so}
|\tilde{\phi}_n^k (t,z) - \tilde{\phi}_n^k (t,y)| 
& = \psi_n^k(t)(|z| - |y|) \leq \psi_n^k (t) |z-y|,
\end{align*}  
as required, giving that $\mathrm{Lip}(\tilde{\phi}_n^k(t, \cdot)) \leq \psi_n^k(t)$ for such values.

In the case when $5 \theta_n |g(t)| \leq |y| \leq |z|$, we have that
\[
\tilde{\phi}_n^k (t,y) = 5 \theta_n |g_n(t)|\psi_n^k(t) = \tilde{\phi}_n^k (t,z),
\]
and so both results are immediate.  In the case in which $|y| \leq 5 \theta_n | g_n (t) | \leq |z|$, we have that
\begin{align*}
\tilde{\phi}_n^k (t,z) - \tilde{\phi}_n^k (t,y)
& = 5\theta_n |g_n(t)|\psi_n^k(t) - \psi_n^k(t)|y| \geq 0;\\
\shortintertext{and so}
|\tilde{\phi}_n^k (t, z) - \tilde{\phi}_n^k (t,y)| 
&  = \psi_n^k (t)(5\theta_n |g_n(t)| - |y|) \leq \psi_n^k(t)(|z| - |y|) \leq \psi_n^k(t)| z-y|.
\end{align*} 
Thus in this case again $\mathrm{Lip}(\tilde{\phi}_n^k (t, \cdot)) \leq \psi_n^k (t)$.  Both results follow similarly for $t \notin Y_n$: we obtain instead that $\mathrm{Lip}(\tilde{\phi}_n^k (t, \cdot)) \leq \psi^k (T_n)$, hence the full result claimed.  Hence for all $t \in [-T, T]$, $|y| \mapsto \tilde{\phi}_n (t, y)$ is an increasing function with Lipschitz constant at most $\max\{\psi_n(t), \psi (T_n)\}$, and $\tilde{\phi}(t, 0) =0$.  This Lipschitz constant blows up as we approach $x_n$, since $\psi_n (t) \geq |\tilde{w}_n''(t)| \to \infty$, but we shall not need to use a Lipschitz estimate of this function arbitrarily close to $x_n$.

Defining $\phi_n \colon [-T, T] \times \mathbb{R} \to [0, \infty)$ by $\phi_n(t,y) = \sum_{i=0}^n \tilde{\phi}_i (t,y)$ gives a sequence of continuous functions such that for each $t \in [-T, T]$, 
\begin{gather} 
\phi_n (t, y) \leq \phi_n (t, z) \ \textrm{whenever $ |y| \leq |z|;$}\label{phininc}\\
\mathrm{Lip}(\phi_n(t, \cdot)) \leq \sum_{i=0}^n \left(\max\{\psi_i (t), \psi (T_i)\}\right);\ \label{lipphin}\textrm{and}\\
\phi_n (t, 0 ) = 0.\label{phin0is0}
\end{gather} 
For $n \geq 1$, by~\ref{t4}, we see that $0 \leq \tilde{\phi}_n (t,y) \leq \sup_{t \in Y_n} 5 \psi_n(t)\theta_n |g_n(t)| \leq 2^{-n}$ for all $(t, y) \in [-T, T] \times \mathbb{R}$. Hence the sequence $\{\phi_n\}_{n=0}^{\infty}$ converges uniformly to a continuous function given by $\phi(t,y) = \sum_{i=0}^{\infty} \tilde{\phi}_i (t,y)$, satisfying
\begin{gather}
\| \phi\|_{\infty} \leq \| \tilde{\phi}_0\|_{\infty} + \sum_{i=1}^{\infty}\|\tilde{\phi}_i\|_{\infty} \leq \|\tilde{\phi}_0\|_{\infty} + \sum_{i=1}^{\infty}2^{-i} = \|\tilde{\phi}_0\|_{\infty} + 1 = C;\ \textrm{and}\label{phibd}\\ 
\|\phi- \phi_n\|_{\infty} \leq \sum_{i=n+1}^{\infty} \|\tilde{\phi}_i\|_{\infty} \leq \sum_{i = n+1}^{\infty}2^{-i} 
= 2^{-n}.\label{phi-phin}
\end{gather}
By passing to the limit in~\eqref{phininc} and~\eqref{phin0is0} we see that for each $ t \in [-T, T]$, 
\begin{gather} 
\phi(t, y) \leq \phi(t, z)\ \textrm{whenever $ |y | \leq |z|$};\ \textrm{and} \label{phiinc}\\ 
\phi(t, 0) = 0. \label{phi0is0}
\end{gather} 
We shall let $\phi = \phi^1 + \phi^2$, where $\phi^k = \sum_{i=0}^{\infty} \tilde{\phi}_i^k$ for $k=1,2$.  

We can now define a functional $\mathscr{L}$ on $W^{1,1}(-T, T)$, with a continuous Lagrangian, superlinear and convex in $p$, by
\[
\mathscr{L}(u) = \int_{-T}^{T} \left( \phi(t, u(t) - w(t)) +(u'(t))^2 \right)\, dt,
\]
and consider minimizing $\mathscr{L}$ over $\mathscr{A}_{w(-T,), w(T)}$. Since evidently $\mathscr{L}$ is coercive on $W^{1,2}(-T, T)$, a minimizer over $\mathscr{A}_{w(-T), w(T)}$ exists in $W^{1,2}(-T, T)$, and we can regard the minimization problem as being defined on $W^{1,2}(-T, T)$.
\subsection{Minimality}
We shall find certain approximations to our functional $\mathscr{L}$ useful, and so will define for all $n \geq 0$ the functional $\mathscr{L}_n$ on $W^{1,2}(-T, T)$ by 
\[
\mathscr{L}_n(u) = \int_{-T}^{T} \left(\phi(t, u(t) - w_n(t)) + (u'(t))^2\right)\, dt.
\]
Working with these approximations is much easier, since there is only a finite number of singularities in $w_n$.  So it is important to know what error we incur by moving to these approximations.  This is shown in the next lemma.
\begin{lemma}
	\label{finitejump} 
	Let $u \in W^{1,2}(-T, T)$ and  $n\geq 0$.  Then 
	\[
	\left|\left(\mathscr{L}(u) - \mathscr{L}(w)\right) - \left(\mathscr{L}_n (u) - \mathscr{L}_n (w_n)\right)\right| \leq \frac{ T_{n+1}^2}{2}.
	\]
\end{lemma}
\begin{proof}  
	We first estimate  $|\mathscr{L} (u) - \mathscr{L}_n(u)|$. Recall our definitions of $m_n \geq n$, $M_n \geq 0$,  and $G_n \supseteq \bigcup_{i=0}^{m_n}\{x_i\}$ from the beginning of the construction.  
	Let $t \in [-T, T] \setminus G_n$. We see by~\eqref{lipphin}, and precisely by the choice of $M_n$ in~\eqref{Mn}, that
	\[
	\mathrm{Lip}(\phi_{m_n}(t, \cdot))\leq \sum_{i=0}^{m_n} \left(\max\{\psi_i (t), \psi (T_i)\}\right) \leq M_n.
	\]
	This is the one occasion on which we shall use the (in principle very large) Lipschitz constant of $\phi_{m_n}(t, \cdot)$.  The purpose of the open set $G_n$ was to avoid using this number arbitrarily near $\{x_i\}_{i=0}^{m_n}$, at which points it blows up. 
	
	Using~\ref{w-wn} and~\ref{r3} we see that
	\begin{align*}
	|\phi_{m_n}(t, u-w) - \phi_{m_n}(t, u-w_n)| 
	& \leq \mathrm{Lip}(\phi_{m_n}(t, \cdot))|(u(t) - w(t)) - (u(t) - w_n (t))| \\
	& \leq M_n \|w-w_n\|_{\infty} \\
	& \leq 10 M_n K_{n+1}g(R_{n+1})\\
	& \leq  \frac{T_{n+1}^2}{16}.
	\end{align*}
	The choice of $m_n$ in~\eqref{mn} and~\eqref{phi-phin} imply that 
	\[
	\| \phi - \phi_{m_n}\|_{\infty} \leq 2^{-m_n} \leq\frac{T_{n+1}^2}{32}.
	\]
	Hence
	\begin{align*} 
	|\phi(t, u - w) - \phi(t, u- w_n)| 
	&\leq |\phi(t, u-w) - \phi_{m_n} (t, u-w)| \\
	&\phantom{\leq} {}+ |\phi_{m_n}(t, u-w) - \phi_{m_n}(t, u-w_n)| \\
	&\phantom{\leq} {}+ |\phi_{m_n}(t, u-w_n) - \phi(t , u-w_n)|\\
	&\leq 2 \|\phi - \phi_{m_n}\|_{\infty} + \frac{T_{n+1}^2}{16} \\
	& \leq \frac{2 T_{n+1}^2}{32} + \frac{T_{n+1}^2}{16}\\
	& = \frac{T_{n+1}^2}{8}.
	\end{align*}
	Now,  using~\eqref{phibd} and the choice of the measure of $G_n$ in~\eqref{Fnmeas}, we have that
	\[
	\int_{G_n} |\phi(t, u - w) - \phi(t, u - w_n)|  \leq  2 \int_{G_n} \|\phi\|_{\infty}  \leq  2C \lambda(G_n) \leq 2C \frac{T_{n+1}^2}{16C} = \frac{T_{n+1}^2}{8}. 
	\]
	Combining these estimates, we see that
	\begin{align*}
	\lefteqn{|\mathscr{L}(u) - \mathscr{L}_n (u)| }\\
	& =  \left|\int_{-T}^T \left(\phi(t, u - w)+  (u')^2 \right) - \left( \phi(t, u - w_n) +  (u')^2 \right) \right| \\
	& \leq \int_{-T}^{T} |\phi(t, u-w) - \phi(t, u-w_n)| \\
	& =   \int_{G_n} |\phi(t, u-w) - \phi(t, u-w_n)|+ \int_{[-T, T] \setminus G_n} |\phi(t, u-w) - \phi(t, u-w_n)| \\
	& \leq   \frac{T_{n+1}^2}{8} + \int_{[-T, T] \setminus G_n} \frac{T_{n+1}^2}{8} \\
	& \leq \frac{T_{n+1}^2}{8}  + \frac{T_{n+1}^2}{8}\\
	& = \frac{T_{n+1}^2}{4}.
	\end{align*}
	Now we estimate $|\mathscr{L}(w) - \mathscr{L}_n (w_n)|$.  First we  compare $w'$ and $w_n'$ with $w_0 = \tilde{w}$ in the $L^2$-norm, noting that~\ref{wn'tow'in L2} and~\eqref{L2cauchy} in particular allow the estimates
	\[
	\|w'\|_{L^2(-T, T)} \leq 1+ \|\tilde{w}'\|_{L^2(-T, T)},\ \textrm{and}\  \|w_n'\|_{L^2(-T, T)} \leq 1+ \|\tilde{w}'\|_{L^2(-T, T)}.
	\]
	Hence it follows that
	\[
	\|w' + w_n'\|_{L^2(-T, T)} \leq \|w'\|_{L^2(-T, T)} + \| w_n'\|_{L^2(-T, T)} \leq 2\left(1 +  \|\tilde{w}'\|_{L^2(-T, T)} \right).
	\]
	Thus using~\eqref{phi0is0}, Cauchy-Schwartz,  and~\ref{wn'tow'in L2}, we see that
	\begin{align*}
	|\mathscr{L}(w) - \mathscr{L}_n (w_n)|  \leq \int_{-T}^{T} \left|(w')^2 - (w_n')^2\right|
	& = \int_{-T}^T \left(|w' + w_n'| | w' - w_n'|\right) \\
	& \leq  \|w' + w_n'\|_{L^2(-T, T)} \|w' - w_n'\|_{L^2(-T, T)}\\
	& \leq \frac{2\left(1+ \|\tilde{w}'\|_{L^2(-T, T)}\right) T_{n+1}^2}{8\left(1 + \|\tilde{w}'\|_{L^2(-T, T)}\right)}\\
	& = \frac{T_{n+1}^2}{4}.
	\end{align*}
	Combining these two estimates we see that
	\begin{align*} 
	\left|\left(\mathscr{L}(u) - \mathscr{L}(w)\right) - \left(\mathscr{L}_n (u) - \mathscr{L}_n (w_n)\right)\right|
	& \leq |\mathscr{L}(u) - \mathscr{L}_n (u)| + |\mathscr{L}(w) - \mathscr{L}_n (w_n)| \\
	& \leq \frac{T_{n+1}^2}{4} + \frac{T_{n+1}^2}{4} \\
	& =  \frac{T_{n+1}^2}{2} . 
\qedhere
	\end{align*}
\end{proof}
We now show that  $w$ is the unique solution of our minimization problem.   The basic idea on $\tilde{w}$ which we sketched at the beginning of this section is mimicked locally on $w$ around each $x_n$; more precisely we in fact argue with $w_n$ and then either show that for some $n$ this suffices to give the result for $w$, or pass to the limit.  The techniques of our proof show in fact that $w_n$ is the unique minimizer of the variational problem
\[
W^{1,2}(-T, T) \ni u \mapsto \mathscr{L}_n(u)
\]
over those $u$ such that $u(\pm T) = w_n (\pm T) (= w(\pm T))$.

Let $u \in W^{1,2}(-T, T)$ be a minimizer of $\mathscr{L}$ over $\mathscr{A}_{w(-T), w(T)}$, and suppose for a contradiction that $u \neq w$.  Note that a minimizer certainly exists, since the Lagrangian is continuous, and superlinear and convex in $p$.  We now make a number of estimates, with the eventual aim of showing that 
\[
\mathscr{L}(u) - \mathscr{L}(w)  = \int_{-T}^T \left((u')^2 + \phi(t, u - w) - (w')^2 \right)> 0,
\]
which contradicts the choice of $u$ as a minimizer.  If $u(x_n) = w(x_n)$ for all $ n\geq 0$, then  the proof is in principle an easy application of integration by parts as discussed above on the complement of the closure of the points $\{x_n\}_{n= 0}^{\infty}$. (In the case that $\{x_n\}_{n=0}^{\infty}$ forms a dense set in $[-T, T]$, we should immediately have $u=w$ by continuity, thus concluding the proof of minimality of $w$ without using the assumption that $u \neq w$.)   Should $w(x_n) \neq u(x_n)$ for some $ n \geq 0$, further argument is required.  The next lemma shows us that \emph{since $u$ is a minimizer}, it cannot be too badly behaved around any such $x_n$. 
\begin{lemma}
	\label{lemma1} 
	Let $n \geq 0$ be such that $u(x_n) \neq w(x_n)$.  Let $a_n, b_n >0$ be such that $J_n\df (x_n- a_n, x_n+ b_n)$ is the connected component in $[-T, T]$ containing $x_n$ of those points such that $
	| u(t) - w(x_n)| > 3 \theta_n|g_n(t)|$, 
	so $|u(x_n-a_n) - w(x_n)| = 3\theta_n g(a_n)$ and $|u(x_n + b_n) - w(x_n)| = 3\theta_n g(b_n)$.
	(Note that $J_n \subsetneq [-T, T] $ since $u$ and $w$ agree at $\pm T$ and so by~\ref{lipw} 
	\[
	\left|u(\pm T) - w(x_n)\right| = \left|w(\pm T) - w(x_n)\right| \leq 2 \theta_n \left|g_n (\pm T)\right|.)
	\]
	
	Then 
	\begin{equation}
	\label{vnbig}
	\begin{cases}
	|(u - w_n) (t)| \geq \theta_n g(b_n)\ \textrm{for}\ t \in [x_n, x_n + b_n] & b_n \geq a_n,\\
	|(u - w_n) (t)|\geq \theta_n g(a_n) \ \textrm{for}\ t \in [x_n - a_n, x_n] & a_n \geq b_n,
	\end{cases}
	\end{equation}
	and 
	\begin{equation*}
	|u(t) - w(x_n) | \leq 3 \theta_n |g_n(t)|\ \textrm{for $t \notin J_n$.}
	\end{equation*}
\end{lemma}
\begin{proof}
	We suppose that $u(x_n) > w(x_n)$.  The argument for the case in which $u(x_n) < w(x_n)$ is very similar. 	We choose $\alpha_n, \beta_n >0$ such that $(x_n - \alpha_n, x_n + \beta_n)$ is the connected component in $[-T, T]$ containing $x_n$ of those points such that  $|u(t) - w(x_n)| > 2 \theta_n |g_n (t)|$.  So $a_n \leq \alpha_n$ and $ b_n \leq \beta_n$, and $[x_n - \alpha_n, x_n + \beta_n] \subseteq [-T, T]$.  We prove that $u$ is convex on $(x_n-\alpha_n, x_n+\beta_n)$.  In the case in which $u(x_n) < w(x_n)$, we would prove that $u$ is concave on $(x_n - \alpha_n, x_n + \beta_n)$.  Suppose for a contradiction that on some non-trivial subinterval $(t_1, t_2)$ of $(x_n - \alpha_n, x_n + \beta_n)$, $u$ lies above its chord between the points $t_1$ and $t_2$, i.e.\ that there exists some $\mu \in [0, 1]$ such that 
	\[
	u(\mu t_1 + (1- \mu )t_2) > \mu u (t_1) + ( 1- \mu) u (t_2).
	\]
	The basic idea is that in this case we can redefine $u$ to be affine on $(t_1, t_2)$ or some subinterval of $(t_1, t_2)$, producing a function which does not increase the weight term $\phi(t, \cdot - w(t))$ of the integrand, since it only moves closer to $w$, and strictly decreases the  gradient term, since it has constant gradient.  Let $z \colon [-T, T] \to \mathbb{R}$ be the affine function with graph passing through $(t_1, u(t_1))$ and $(t_2, u(t_2))$, so
	\[
	z(t) = \frac{u(t_2) - u(t_1)}{t_2 - t_1}\cdot (t-t_1) + u(t_1).
	\]
	So we have by assumption on $t_1, t_2$ that
	\[
	z(\mu t_1 + (1-\mu)t_2) = \mu u(t_1) + (1- \mu) u(t_2) < u(\mu t_1 + (1-\mu)t_2).
	\]
	Passing to connected components if necessary, we can assume that $z < u$ on $(t_1, t_2)$.  We claim that adding a certain constant value onto the function $z$ gives an affine function $\tilde{z}$ such that on some subinterval $(\tilde{t}_1, \tilde{t}_2)$ of $(t_1, t_2)$, we have
	\[
	w(x_n) + 2 \theta_n |g_n| \leq \tilde{z} < u.
	\]
	We then show that this contradicts the choice of $u$ as a minimizer.
	
	Since $z$ is affine and $|g_n|$ is monotonic and concave on $[-T, x_n]$ and $[x_n, T_0]$, the equation $z = w(x_n) + 2 \theta_n|g_n|$ can in principle have no or up to three distinct solutions on $(t_1,t_2)$.  If there is at most one solution, then  since $z(t_i) = u(t_i) \geq w(x_n) + 2\theta_n |g_n(t_i)|$ for $i=1,2$, evidently $z \geq w(x_n) + 2 \theta_n |g_n|$ on $(t_1, t_2)$.  So we need not modify $z$ at all to get our required $\tilde{z}$.
	
	The case of three distinct solutions is in fact impossible.   Suppose we had three such points $s_1, s_2, s_3 \in (t_1, t_2)$.  Again by the elementary properties of $g_n$  and $z$, all three points cannot lie on one side of $x_n$.  So suppose $s_1 < x_n \leq s_2 < s_3$.  The principle here is that $z$ must have a positive gradient if it intersects $w(x_n) + 2 \theta_n |g_n|$ twice on the right of $x_n$.   This then forces $z$ to lie below  $w(x_n) + 2 \theta_n |g_n|$ for all points $t <s_1 < x_n$, which is a contradiction since it agrees with $u$ at $t = t_1 < s_1$, and $u$ lies above $w(x_n) + 2 \theta_n |g_n|$ at this point.  More precisely, for $t < x_n$, we have that
	\[
	z' = \frac{ 2 \theta_n|g_n (s_3)| - 2 \theta_n |g_n (s_2)| }{s_3 - s_2} = \frac{2 \theta_n g_n (s_3) - 2 \theta_n g_n (s_2)}{s_3 - s_2} > 0 > -2 \theta_n g_n' (t).
	\]
	Since $t_1 < s_1 < x_n$, we have $| g_n (s_1)| = - g_n (s_1)$ and $|g_n (t_1)| = - g_n (t_1)$, so
	\begin{align*}
	z(t_1)  = z(s_1) - \int_{t_1}^{s_1} z'(t) \,dt 
	& < w(x_n) -2 \theta_n g_n(s_1) - \int_{t_1}^{s_1}(-2 \theta_n g_n'(t)) \,dt \\
	&= w(x_n) -2 \theta_n g_n (t_1) \\
	& = w(x_n) + 2 \theta_n | g_n(t_1)|	.
	\end{align*}
	This  is a contradiction since $z(t_1) = u(t_1) \geq w(x_n) +  2\theta_n |g_n(t_1)|$.  We can deal similarly with the case $s_1 < s_2 \leq x_n < s_3$.
	
	So it remains to deal with the case in which  we have precisely two distinct solutions $(s_1, s_2)$---this is the case in which we have in general to add a constant to $z$, since it is possible that $w(x_n) + 2 \theta_n |g_n|$ lies above $z$ on some subinterval of $(s_1, s_2)$.  The same considerations as in the preceding paragraph show that we must have both solutions lying on one side of $x_n$.  Suppose $x_n \leq s_1 < s_2$.  Then $2|g_n| = 2g_n$ is $C^{\infty}$ on $(s_1, s_2)$, so applying the mean value theorem we see that there is a point $s_0 \in (s_1, s_2)$ such that
	\[
	2 \theta_n g_n' (s_0) = \frac{2 \theta_n g_n (s_2) - 2 \theta_n g_n (s_1)}{s_2 - s_2} = \frac{z(s_2) - z(s_1)}{s_2 - s_1} = z'.
	\]
	Define $\tilde{z}$ by
	\[
	\tilde{z}(t) = z'\cdot (t-s_0) + w(x_n) + 2 \theta_n g_n(s_0),
	\]
	the tangent to $w(x_n) + 2\theta_n g_n$ at $s_0$, so, since $s_0 \in (s_1, s_2) \subseteq (t_1, t_2) \subseteq (x_n - \alpha_n, x_n + \beta_n)$,
	\[
	\tilde{z}(s_0) = w(x_n) + 2 \theta_n g_n (s_0) = w(x_n) + 2 \theta_n |g_n (s_0)|< u(s_0).
	\]
	Let $(\tilde{t}_1, \tilde{t}_2) $ be the connected component containing $s_0$ of those points at which $u > \tilde{z}$.  Since $s_0 \in (s_1, s_2)$, and $z(s_i) = w(x_n) + 2\theta_n g_n (s_i)$ for $ i = 1,2$, concavity of $g$ implies that $w(x_n) + 2 \theta_n g_n (s_0) \geq z(s_0)$.  Since $\tilde{z}(s_0) = w(x_n) + 2 \theta_n g_n(s_0)$ by definition, and $z' = \tilde{z}'$, we have that $\tilde{z} \geq z$ everywhere.  So $u > \tilde{z}$ implies that $u > z$, thus $(\tilde{t}_1, \tilde{t}_2) \subseteq (t_1, t_2)$.
	
	We claim that $\tilde{z} \geq w(x_n) + 2 \theta_n |g_n|$ on $(\tilde{t}_1, \tilde{t}_2)$.  Since $s_0 > s_1 \geq x_n$ and $\tilde{z}(s_0) = w(x_n) + 2 \theta_n |g_n (s_0)|$, with $\tilde{z}' = z' = 2\theta_n g_n'(s_0)$, by concavity of $g$ we have that $\tilde{z} \geq w(x_n) + 2 \theta_n|g_n| $ on $(x_n, T)$.  Suppose there existed $s \in (\tilde{t}_1, x_n]$ such that $\tilde{z}(s) < w(x_n) + 2 \theta_n|g_n (s)| = w(x_n) - 2 \theta_n g_n (s)$.  Then we see as before, since $\tilde{z}' > 0 > - 2\theta_n g_n'(t)$ for $t < x_n$, that
	\begin{align*}
	\tilde{z}(\tilde{t}_1)  = \tilde{z}(s) - \int_{\tilde{t}_1}^{s} \tilde{z}'(t) \,dt 
	& < w(x_n) -2 \theta_n g_n (s) - \int_{\tilde{t}_1}^{s} (-2 \theta_n g_n'(t)) \,dt \\
	& = w(x_n) -2 \theta_n g_n (\tilde{t}_1) \\
	& = w(x_n) + 2 \theta_n |g_n(\tilde{t}_1)|,
	\end{align*}
	which contradicts $\tilde{z} (\tilde{t}_1) = u(\tilde{t}_1) \geq w(x_n) + 2 \theta_n |g_n (\tilde{t}_1)|$.  So $\tilde{z} \geq w(x_n) + 2 \theta_n |g_n|$ on $(\tilde{t}_1, \tilde{t}_2)$ indeed.  The case in which $s_1 < s_2 \leq x_n$ is similar.  So we have constructed an affine $\tilde{z}$ as claimed.
	
	Thus, since $w \leq w(x_n) + 2 \theta_n |g_n|$ by~\ref{lipw},  we have on $(\tilde{t}_1, \tilde{t}_2)$ that
	\begin{equation}
	\label{affine}
	|u-w| = u-w \geq \tilde{z} - w = |\tilde{z} - w|.
	\end{equation}
	Since $u > \tilde{z}$ on $(\tilde{t}_1, \tilde{t}_2)$, where $\tilde{z}$ is affine, but $u=\tilde{z}$ at the endpoints, we know that $u$ is not affine on $(\tilde{t}_1, \tilde{t}_2)$, so we have strict inequality in the Cauchy-Schwartz inequality, thus
	\begin{align}
	\int_{\tilde{t}_1}^{\tilde{t}_2} (u')^2  = \frac{1}{\tilde{t}_2 - \tilde{t}_1} \left(\int_{\tilde{t}_1}^{\tilde{t}_2} 1^2\right)\left( \int_{\tilde{t}_1}^{\tilde{t}_2} (u')^2 \right) 
	& >  \frac{1}{\tilde{t}_2 - \tilde{t}_1} \left(\int_{\tilde{t}_1}^{\tilde{t}_2} u'\right)^{\! 2} \nonumber\\
	& =  \frac{(u(\tilde{t}_2) - u(\tilde{t}_1))^2}{\tilde{t}_2 - \tilde{t}_1} \nonumber\\
	& =  (\tilde{t}_2 - \tilde{t}_1) \left(\frac{z(\tilde{t}_2) - z(\tilde{t}_1)}{\tilde{t}_2 - \tilde{t}_1}\right)^{\! 2}\nonumber \\
	& =  (\tilde{t}_2 - \tilde{t}_1) (\tilde{z}')^2 \nonumber\\
	& = \int_{\tilde{t}_1}^{\tilde{t}_2} (\tilde{z}')^2.\label{holder}
	\end{align}
	Hence defining $\tilde{u}\colon [-T, T] \to \mathbb{R}$ by
	\[
	\tilde{u}(t) =
	\begin{cases}
	u(t) & t \notin (\tilde{t}_1, \tilde{t}_2), \\
	\tilde{z}(t) & t \in (\tilde{t}_1, \tilde{t}_2);
	\end{cases}
	\]
	we obtain a function $\tilde{u} \in W^{1, 2}(-T, T)$ with $\tilde{u}(\pm T) = w(\pm T)$ and such that, using~\eqref{holder},~\eqref{affine}, and~\eqref{phiinc},
	\begin{align*}
	\mathscr{L}(\tilde{u}) 
	& =  \int_{-T}^{T} \left((\tilde{u}')^2 + \phi(t, \tilde{u} - w) \right)\\
	& =  \int_{[-T, T] \setminus (\tilde{t}_1, \tilde{t}_2)} \left((u')^2 + \phi(t, u-w)\right) + \int_{\tilde{t}_1}^{\tilde{t}_2}\left( (\tilde{z}')^2 + \phi(t, \tilde{z}-w)\right) \\
	& <  \int_{[-T, T] \setminus (\tilde{t}_1, \tilde{t}_2)} \left((u')^2 + \phi(t, u-w)\right) + \int_{\tilde{t}_1}^{\tilde{t}_2}\left( (u')^2 + \phi(t, u-w)\right) \\
	& =  \mathscr{L}(u),         
	\end{align*}
	which contradicts the choice of  $u$ as a minimizer.  Hence $u$ is indeed convex on $(x_n-\alpha_n, x_n+\beta_n)$.
	
	It now follows that the graph of $u$ on $(x_n - \alpha_n, x_n +\beta_n)$ lies above the tangents to $w(x_n) + 2 \theta_n|g_n|$ at $(x_n-\alpha_n)$ and $(x_n+\beta_n)$:
	\begin{align*}
	u(t) & \geq w(x_n) + 2 \theta_n g(\beta_n) +  2\theta_n g'(\beta_n) (t - (x_n + \beta_n)),  
	\shortintertext{and}
	u(t) & \geq  w(x_n)+  2 \theta_n |g(-\alpha_n)| - 2\theta_n g'(-\alpha_n)(t - (x_n -\alpha_n)),  
	\end{align*}
	for $t \in (x_n - \alpha_n, x_n + \beta_n)$.
	For suppose the first fails, i.e.\ that for some $t_0 \in (x_n - \alpha_n, x_n +\beta_n)$ we have that
	\[
	u(t_0) <  w(x_n) + 2 \theta_n g(\beta_n) + 2 \theta_n g'(\beta_n)(t_0 - (x_n +\beta_n)).
	\]
	Then by convexity the graph of $u$ lies below the chord between the points $(t_0, u(t_0))$ and $(x_n + \beta_n, u(x_n + \beta_n))=(x_n +\beta_n, w(x_n) + 2 \theta_n g(\beta_n) )$, which has slope
	\[
	\frac{w(x_n) + 2 \theta_n g(\beta_n) - u(t_0)}{x_n + \beta_n - t_0}.
	\]
	By assumption 
	\[
	\frac{w(x_n) + 2 \theta_n g(\beta_n) - u(t_0)}{x_n + \beta_n - t_0} > 2 \theta_n g'(\beta_n),
	\]
	and so since $g'$ is continuous we have that
	\[
	2 \theta_n g_n'(t) < \frac{w(x_n) + 2 \theta_n g(\beta_n) - u(t_0)}{x_n + \beta_n - t_0}
	\]
	on some left neighbourhood of $x_n + \beta_n$.  So for $t$ in this neighbourhood, we have that
	\begin{align*}
	w(x_n) +  2 \theta_n g_n (t) 
	&= w(x_n) + 2  \theta_n g_n (x_n + \beta_n) - \int_t^{x_n + \beta_n} 2 \theta_n g_n'(s)\,ds  \\
	&> w(x_n) + 2 \theta_n g(\beta_n) - \int_t^{x_n + \beta_n} \frac{w(x_n) + 2 \theta_n g(\beta_n) - u(t_0)}{x_n + \beta_n - t_0} \, ds\\
	&= w(x_n) + 2  \theta_n g (\beta_n) - \frac{w(x_n) + 2 \theta_n g(\beta_n) - u(t_0)}{x_n + \beta_n - t_0}(x_n + \beta_n - t)\\
	& = u (x_n + \beta_n) - \frac{w(x_n) + 2 \theta_n g(\beta_n) - u(t_0)}{x_n + \beta_n - t_0}(x_n + \beta_n - t)\\
	&\geq u(t), 
	\end{align*}
	which is a contradiction for $t \in (x_n - \alpha_n, x_n + \beta_n)$.
	Similarly we can prove that $u$ lies above the other tangent.
	
	We can now prove certain bounds on $u'$.  Suppose that there exists $t_0 \in (x_n - \alpha_n, x_n + \beta_n )$ such that $u'(t_0) > 2 \theta_n g'(\beta_n)$.  Then we have that $u'(t) > 2 \theta_n g'(\beta_n)$ for all $t \in (t_0, x_n +\beta_n)$ by convexity.  Then we see, using the inequality proved in the previous paragraph,  that
	\begin{align*}
	u(x_n + \beta_n) 
	& = u(t_0) + \int_{t_0}^{x_n + \beta_n}u'(t) \,dt\\
	& >  w(x_n) + 2\theta_n g(\beta_n) + 2 \theta_n g'(\beta_n) (t_0 - (x_n + \beta_n)) + \int_{t_0}^{x_n +\beta_n} 2\theta_n g'(\beta_n)\, ds\\
	& =  w(x_n) + 2\theta_n g(\beta_n) + 2 \theta_n g'(\beta_n) (t_0 - (x_n + \beta_n))  \\
	& \phantom{=}  + ((x_n +\beta_n) - t_0)2\theta_n g'(\beta_n)\\
	& = w(x_n) + 2 \theta_n g(\beta_n),     
	\end{align*}
	which is a contradiction since $u(x_n + \beta_n) =  w(x_n) + 2 \theta_n g(\beta_n)$ by the choice of $\beta_n$.  So $u'(t) \leq 2 \theta_n g'(\beta_n)$ for almost every $t \in (x_n -\alpha_n, x_n + \beta_n)$.  Similarly we can prove that $u'(t) \geq -2 \theta_n g'(-\alpha_n)$ for almost every $t \in (x_n - \alpha_n, x_n + \beta_n)$. In the case in which $u(x_n) < w(x_n)$ we would prove that $	-2\theta_n g'(\beta_n) \leq u'(t) \leq 2 \theta_n g'(-\alpha_n)$ for almost every $t \in (x_n - \alpha_n, x_n + \beta_n)$.
	
	We now prove the important consequence~\eqref{vnbig} of these estimates.  Suppose that $b_n \geq a_n$.  
	Then using convexity of $u$, and, by monotonicity of $g$, the fact that $g(b_n) \geq  g(-a_n) = -g(a_n)$, we see that for $t \in J_n$,
	\begin{align*}
	u(t) 
	&\leq \frac{u(x_n + b_n) - u(x_n - a_n)}{b_n + a_n} (t- (x_n + b_n)) + u(x_n + b_n)\\
	& = \frac{3 \theta_n g (b_n) - 3 \theta_n g (a_n)}{b_n + a_n}(t- (x_n + b_n)) + w(x_n) + 3\theta_n g (b_n)\\
	& \leq w(x_n) + 3\theta_n g (b_n).
	\end{align*}
	Fix $t \in [x_n, x_n + b_n]$. We then have by the estimates we have just proved that
	\begin{align*}
	u(t) 
	& = u(x_n + b_n) - \int_t^{x_n + b_n} u'(s) \, ds \\
	& \geq w(x_n) + 3 \theta_n g (b_n)  - \int_t^{x_n + b_n} 2\theta_n g' (\beta_n)\, ds\\
	& = w(x_n) + 3\theta_n g (b_n) - 2((x_n + b_n) - t)\theta_n g'(\beta_n).
	\end{align*}
	Also, since $t \leq x_n + b_n$, we have, using~\ref{lipwn},~\ref{wfixxi}, and concavity of $g$,
	\begin{align*}
	w_n (t) 
	& \leq w(x_n) + 2 \theta_n g_n (t)\\
	& \leq  w(x_n) + 2 \theta_n g_n'(x_n + b_n) (t- (x_n + b_n)) + 2\theta_n g_n (x_n + b_n) \\
	& \leq  w(x_n) + 2 \theta_n g' (\beta_n)(t- (x_n + b_n)) + 2\theta_n g (b_n) .
	\end{align*}
	So  we have that 
	\begin{align*}
	u(t) - w_n (t) 
	& \geq  \left( w(x_n) + 3\theta_n  g (b_n) - 2((x_n + b_n) - t)\theta_n g' (\beta_n)\right)\\
	& \phantom{=} {}- \left( w(x_n) +   2 \theta_n g'(\beta_n) (t- (x_n + b_n)) + 2 \theta_n g (b_n)\right)\\
	& = \theta_n g(b_n).
	\end{align*}
	Similarly we can prove that $u(t) - w_n (t) \geq \theta_n g(a_n)$  for $t \in [x_n - a_n, x_n]$ if $a_n \geq b_n$.  In the case that $u(x_n) < w(x_n)$ we can prove in the same way that $u(t) - w_n (t) \leq - \theta_n g(b_n)$ on $[x_n, x_n + b_n]$ if $b_n \geq a_n$, or $u(t) - w_n(t) \leq - \theta_n g(a_n)$ on $[x_n - a_n, x_n]$ if $ a_n \geq b_n$,  hence the full result.
	
	The final statement of the lemma is proved using the techniques we  used above to prove convexity of $u$ on $(x_n - \alpha_n, x_n + \beta_n)$.  Suppose that there is a $t_0 \in (x_n + b_n, T)$ such that $u(t_0) >  w(x_n) + 3 \theta_n g_n (t_0)$. The argument on the left of $x_n$ is the same.  Defining affine $z \colon [-T, T]\to \mathbb{R}$ by
	\[
	z(t) = w(x_n) + 3 \theta_n g_n (t_0)+  3 \theta_n g_n'(t_0) (t-t_0),
	\]
	we see that $z(t_0 ) =  w(x_n) + 3 \theta_n g_n (t_0) < u(t_0)$, and, using the concavity of $g_n$, that $z \geq  w(x_n) + 3 \theta_n g_n$ on $(x_n, T)$.  The connected component of $[-T, T]$ containing $t_0$ of those points for which $z < u$ is a subinterval of $(x_n + b_n, T)$, since
	\[
	u(x_n + b_n) = w(x_n) + 3 \theta_n g_n (b_n) \leq z(x_n + b_n),
	\]
	and by~\ref{lipw}, 
	\[
	u(T) = w(T) \leq w(x_n) + 2 \theta_n g_n (T)  < z(T).
	\]
	So we have that $u(t) > z(t) \geq w(x_n) + 3 \theta_n g_n (t)$ on some open subinterval of $(x_n + b_n, T)$.  Hence we can perform the same trick as before, constructing a new function $\tilde{u} \in W^{1, 2}(-T, T)$ by replacing $u$ with $z$ on this subinterval, such that $\mathscr{L}(\tilde{u})  < \mathscr{L}(u)$, which again contradicts the choice of $u$ as a minimizer.
	\end{proof}
Thus we see that if for some $n \geq 0$, $u(x_n) \neq w(x_n)$, then $u$ must be Lipschitz on a neighbourhood of $x_n$, and its graph cannot escape the region bounded by the graphs of $t \mapsto w(x_n)\pm 3 \theta_n |g_n(t)|$ off this neighbourhood.  We note that the final statement of the lemma holds by the same argument even when $u(x_n) = w(x_n)$, and thus when the set $J_n$ introduced is empty.

For the remainder of the proof of minimality, we assume that $u(x_n) \neq w(x_n)$ for all $n \geq 0$.  If not one can just perform the argument in the proofs of lemma~\ref{ngood} and corollary~\ref{cor} on the connected components of $[-T, T] \setminus \overline{\{x_n : u(x_n) = w(x_n) \}}$.  We make  remarks in these proofs at those points where an additional argument is required in the general case.

For each $n \geq 0$ we now introduce some definitions and notation.  Let $a_n, b_n > 0$ be such that $J_n \df (x_n - a_n, x_n + b_n)$ 
is the connected component in $[-T,T]$ containing $x_n$ of those points $t$ such that $|u(t) - w(x_n)| > 3 \theta_n |g_n (t)|$, as in lemma~\ref{lemma1}. It will be easier to work on a symmetric interval around $x_n$, so let $c_n \df \max\{a_n, b_n\}$, and $\tilde{J}_n \df [x_n - c_n, x_n + c_n]$.  We note the following immediate corollary of lemma~\ref{lemma1}.  Fix $n \geq 0$.  For $ t \notin J_n$, we have for any $i \geq n$, by~\ref{wfixxi} and~\ref{lipwn}, that 
\begin{align} 
|(u - w_i)(t)| \leq |u(t) - w(x_n)| + |w(x_n) - w_i(t)|
& =  |u(t) - w (x_n)| + |w_i(x_n)- w_i (t)| \nonumber \\
& \leq 3 \theta_n |g_n(t)| + 2 \theta_n |g_n(t) |\nonumber\\
& = 5 \theta_n  |g_n (t)|. \label{vileq5offJn}
\end{align} 
The inequalities~\eqref{vnbig} from lemma~\ref{lemma1} tell us that the graph of a putative minimizer $u$ cannot get too close to that of $w$ around $x_n$.  In the next result, this lower bound on the distance between the two functions is shown to concentrate a certain amount of weight in the Lagrangian around each $x_n$.   The total weight is of course in general even larger---we took an infinite sum of such non-negative terms---but the important term is the $\tilde{\phi}_n$ term which deals precisely with the oscillations introduced by $w_n$ to get singularity of $w$ at $x_n$.
\begin{lemma}
	\label{intphibig} 
	Let $n \geq 0$, and suppose $\tilde{J}_n \subseteq Y_n$.  
	
	Then 
	\[
	\int_{\tilde{J}_n} \tilde{\phi}_n^1(t, u - w_n) \geq \frac{453\theta_n g(c_n)}{ (\one 1/c_n)^{1/3}}.
	\]
\end{lemma}
\begin{proof}
	Choose $t_{c_n} \in (0, c_n)$ such that $g(t_{c_n}) = g(c_n)/5$.  Noting that~\eqref{T0small} in particular implies that $t^{1/2} \two 1/ |t| \leq 1 \leq \two 1/ |t|$, we see that if $0 < t^{1/2} \leq c_n /5$, we have that
	\begin{align*}
	g(t) = t \two 1/|t| = t^{1/2} \left(t^{1/2} \two 1/|t| \right)\leq t^{1/2} \leq c_n / 5 
	& \leq (c_n \two 1/c_n )/5 \\
	& = g(c_n)/5,
	\end{align*}
	hence we have the lower bound $t_{c_n}^{1/2} \geq c_n / 5$, and thus the inequality 
	\[
	\one 1/c_n \geq \one 1/ (5 t_{c_n}^{1/2}) = (\one 1/25t_{c_n})/2.
	\]
	Since~\eqref{T0small} also in particular implies that $t_{c_n}^{1/2} \leq (g(c_n) / 5)^{1/2} \leq (1 / 5 \cdot 125)^{1/2} = 1/25$, we have that $1/(25t_{c_n}) \geq (1/t_{c_n})^{1/2}$ and hence that 
	\[
	\one 1/c_n \geq (\log 1/25t_{c_n})/2 \geq (\one (1/t_{c_n})^{1/2})/2 = (\one 1/t_{c_n})/4,
	\]
	the ultimate point being that 
	\begin{equation*}
	\frac{1}{(\one 1/c_n)^{1/3}} \leq \frac{4^{1/3}}{(\one 1/t_{c_n})^{1/3}} \leq \frac{2}{(\one 1/t_{c_n})^{1/3}}.
	\end{equation*}
	
	Suppose that $b_n \geq a_n$, so by definition $c_n = b_n$.  The case in which $a_n > b_n$ differs only in trivial notation. For $t \in [x_n, x_n + t_{c_n}]$ we have by~\eqref{vnbig}, the choice of $t_{c_n}$, and the monotonicity of $g$,  that $ |(u - w_n)(t)| \geq \theta_n g(c_n) = 5 \theta_n g(t_{c_n}) \geq 5 \theta_n g_n (t)$, hence by the definition of $\tilde{\phi}_n^1$ (noting our one assumption in the statement that $\tilde{J}_n \subseteq Y_n$),  
	$\tilde{\phi}_n^1 (t, u - w_n) = 5 \theta_n g_n (t) \psi_n^1 (t)$.  On the interval $[x_n, x_n + t_{c_n}]$ this function is concave, so the integral admits an easy lower estimate by calculating the area of the triangle under the graph, using the definitions of $t_{c_n}$ and $\psi^1$:
	\begin{align*} 
	\int_{\tilde{J}_n} \tilde{\phi}_n^1 (t, u - w_n) 
	\geq \int_{x_n}^{x_n + t_{c_n}} \tilde{\phi}_n^1(t, u - w_n) 
	&  = 5  \theta_n \int_{x_n}^{x_n + t_{c_n}} g_n(t) \psi_n^1 (t)\\ 
	& \geq \frac{5}{2} \theta_n g (t_{c_n}) \psi^1 (t_{c_n}) t_{c_n} \\
	& = \frac{\theta_n g(c_n)}{2}\frac{1812}{t_{c_n} ( \one 1/ t_{c_n})^{1/3}}t_{c_n}\\
	& \geq \frac{\theta_n g(c_n)}{4}\frac{1812}{ ( \one 1/ c_n)^{1/3}}\\
	& = \frac{ 453 \theta_n g(c_n)}{ (\one 1/c_n)^{1/3}}. 
	\qedhere
	\end{align*} 
\end{proof}
We shall want to give special attention to that part of $\tilde{J}_n$ on which $w_n = \tilde{w}_n$, so for $ n\geq 0$ define $H_n \subseteq [-T, T]$ by 
setting $H_n \df \tilde{J}_n \cap [x_n - \tau_n , x_n + \tau_n] = [x_n - d_n , x_n + d_n]$,  say, so $d_n \leq c_n$.  Note that by construction and~\ref{wn=twn},
\[
w_n(x_n \pm d_n) =  \tilde{w}_n (x_n \pm d_n) + \rho_n,\ \textrm{and}\ w_n'(x_n \pm d_n) =  \tilde{w}_n'(x_n \pm d_n).
\]
We cannot immediately mimic the main principle of the proof and  integrate by parts across $x_n$, since $w'_n$ does not exist at $x_n$.  This singularity is of course the whole point of the example.  The main trick of the proof was in making the oscillations of $w_n$ near $x_n$  slow enough so that we can now replace this function with a straight line on an interval containing $x_n$.  We can then use integration by parts on each side of this interval, and inside the interval exploit the fact that we have now introduced a function with constant derivative.  We incur an error in the boundary terms, of course, as we in general introduce discontinuities of the derivative where the line meets $\tilde{w}_n$, but the function $\tilde{w}_n$ moves slowly enough that this error can be dominated by the weight term in the Lagrangian (the role of $\psi_n^1$).

So let $\tilde{l}_n\colon [-T, T] \to \mathbb{R}$ be the affine function defined by
\[
\tilde{l}_n (t) = \tilde{l}_n' \cdot (t - (x_n - d_n)) + \tilde{w} (- d_n),
\]
where 
\begin{equation}
\label{ix}
\tilde{l}_n' \df \frac{\tilde{w} (d_n) - \tilde{w} (- d_n)}{2d_n} = (\two 1/d_n )(\sin \three 1/d_n),
\end{equation}
and define $l_n \colon [-T, T] \to \mathbb{R}$ by 
\[
l_n (t) = 
\begin{cases}
w_n(t) & t \notin H_n, \\  
\tilde{l}_n (t) + \rho_n & t \in H_n.
\end{cases}
\] 
Clearly $l_n \in W^{1,2}(-T, T)$.

We shall find the following notation useful, representing the boundary terms we get as a result of integrating by parts, firstly inside $H_n$, integrating $l_n' (u - w_n)'$, and secondly outside $H_n$, integrating $w_n' (u - w_n)'$:
\begin{gather*}
I_{n, \pm} =  l_n'\left(u (x_n \pm d_n)- w_n(x_n \pm d_n)\right),\\
E_{n, \pm} = w_n' (x_n \pm d_n) \left( u(x_n \pm d_n) - w_n(x_n \pm d_n)\right).
\end{gather*}
Note that 
\begin{align}
|I_{n,\pm} - E_{n,\pm} | &= \left|(l_n' - w_n'(x_n \pm d_n))\left(u(x_n \pm d_n) - w_n(x_n \pm d_n)\right)\right|.\label{pil}
\end{align}
The next lemma describes the consequence for the derivative terms in the integrand of exchanging $w_n$ with $l_n$ on $H_n$.  Integrating by parts gives us the  boundary terms involving $l_n'$, and the second derivative term vanishes, since $l_n$ is affine.  The $L^2$-norm of the difference between $w_n'$ and $l_n'$ gives us an error which we see, comparing with lemma~\ref{intphibig}, will be absorbed into the weight term of the integrand.  
\begin{lemma}
	\label{u'-wn'overHn}
	Let $n \geq0$.
	
	Then 
	\[
	\int_{H_n} \left((u')^2 - (w_n')^2\right) \geq 2(I_{n,+} - I_{n,-})  -  \frac{432g(d_n)}{(\one 1/d_n)^{1/3}}.
	\]
\end{lemma}
\begin{proof} 
	We want to use the following estimate, replacing $w_n$ with the line $l_n$ and estimating the error:
	\begin{align*}
	\int_{H_n} \left((u')^2 - (w_n')^2 \right)
	& =  \int_{H_n} \left((u')^2 - (l_n')^2 \right)+ \int_{H_n}\left((l_n')^2 - (w_n')^2 \right)	\\
	& \geq  \int_{H_n} \left( (u')^2 - (l_n')^2 \right) - \int_{H_n} |(l_n')^2 - (w_n')^2| .
	\end{align*}
	Since $w_n' = \tilde{w}_n'$ and $l_n'= \tilde{l}_n'$ on $H_n$, we can just estimate this term in the case $n=0$; the case of general $n$ is just a translation of this base case.  We drop the index 0 from the notation.
	
	Observe for $t > 0$ that
	\[
	\frac{d}{dt}\left((\two 1/t )(\sin \three 1/t)\right) = -\frac{\sin \three 1/t + \cos \three 1/t}{t \one 1/t},
	\]
	so 
	\[
	\left| \frac{d}{dt}\left((\two 1/t)( \sin \three 1/t)\right)\right| \leq \frac{2}{t \one 1/t}.
	\]
	Hence, recalling the expressions for the derivatives given in~\eqref{ix} and~\eqref{lamda}, and by applying the mean value theorem, we can see that for $0 < t \leq d$,
	\begin{align}
	\lefteqn{|\tilde{l}' - \tilde{w}'(t)|}\nonumber\\ 
	& = \bigg| (\two 1/d )(\sin \three 1/d)  \nonumber\\
	& \phantom{=}{}- \left( (\two 1/t)( \sin \three 1/t) - \frac{\sin \three 1/t + \cos \three 1/t}{\one 1/t}\right) \bigg| \nonumber\\
	& \leq \left|(\two 1/d)( \sin \three 1/d) - (\two 1/t)( \sin \three 1/t)\right| + \frac{2}{\one 1/t} \nonumber \\
	& \leq \frac{2(d- t )}{t \one 1/t} + \frac{2}{\one1/t} \nonumber \\
	&= \frac{2d}{t \one 1/t}.\label{l'-w'}
	\end{align}
	Now, let 
	\begin{equation}
	\label{gammac}
	\gamma (d) \df \frac{d}{(\one 1/d)^{2/3}} \leq d,
	\end{equation}
	so $\one 1/\gamma(d) = \one \left(\frac{(\one 1/d)^{2/3}}{d}\right) = \frac{2}{3} \two 1/d + \one 1/d \leq 2 \one 1/d$, and so we have that
	\begin{equation}
	\label{lglggamma}
	\two 1/\gamma(d) \leq \one (2 \one 1/d) \leq \one (\one1/d)^2  = 2 \two 1/d.
	\end{equation}
	For $ t \in [ \gamma (d) , d]$, we have by~\eqref{l'-w'} and the definition of $\gamma(d)$ that
	\begin{equation}
	\label{|l'-w'|}
	|\tilde{l}' - \tilde{w}'(t)| \leq \frac{2d}{t \one 1/t} \leq \frac{2d}{\gamma(d) \one 1/d} 
	= \frac{ 2 (\one 1/d)^{2/3}}{\one 1/d} 
	= \frac{2}{(\one 1/d)^{1/3}}.
	\end{equation}
	This is one sense in which $\tilde{w}$ oscillates slowly enough: a good estimate for the discrepancy between the derivatives holds on an interval in the domain of integration large enough in measure. Noting that 
	\begin{align*}
	\lefteqn{|\tilde{l}' \pm \tilde{w}'(t)|}\\
	& = \Bigg| (\two 1/d)( \sin \three 1/d) \nonumber\\
	& \phantom{=} {}\pm \left( (\two 1/t )(\sin \three 1/t) - \frac{\sin \three 1/t + \cos \three 1/t}{\one 1/t}\right)\Bigg| \nonumber\\
	& \leq | (\two 1/d)( \sin \three 1/d )| + | (\two 1/ t)( \sin \three 1/ t)| \nonumber \\
	& \phantom{=} {}+ \left| \frac{\sin \three 1/t + \cos \three 1/t}{\one 1/t} \right| \nonumber \\
	& \leq \two 1/d + \two 1/t + \frac{2}{\one1/t}\nonumber\\
	& \leq 4 \two 1/t,
	\end{align*}
	we see, since the integrand is an even function, that
	\begin{align}
	\int_H \left|(\tilde{l}')^2 - (\tilde{w}')^2\right| 
	& = \int_{-d}^d |\tilde{l}' - \tilde{w}'||\tilde{l}' + \tilde{w}'| \nonumber\\
	& = 2 \int_0^d  |\tilde{l}' - \tilde{w}'||\tilde{l}' + \tilde{w}'| \nonumber \\
	& = 2 \left( \int_0^{\gamma(d)} |\tilde{l}' - \tilde{w}'||\tilde{l}' + \tilde{w}'| + \int_{\gamma(d)}^d |\tilde{l}' - \tilde{w}'||\tilde{l}' + \tilde{w}'|\right)\nonumber \\
	& \leq 2 \left( \int_0^{\gamma(d)} (4 \two 1/t)^2  + \int_{\gamma(d)}^d |\tilde{l}' - \tilde{w}'||\tilde{l}' + \tilde{w}'|\right).\label{extra}
	\end{align}
	We then use Cauchy-Schwartz and~\eqref{|l'-w'|} to see that 
	\begin{align*} 
	\int_{\gamma(d)}^d |\tilde{l}' - \tilde{w}'||\tilde{l}' + \tilde{w}'|
	& \leq  \left( \int_{\gamma(d)}^d |\tilde{l}' - \tilde{w}'|^2 \right)^{1/2} \left( \int_{\gamma(d)}^d |\tilde{l}' + \tilde{w}'|^2\right)^{1/2} 	\\
	& \leq \left( \int_{\gamma(d)}^d \left(\frac{2}{(\one1/d)^{1/3}}\right)^2 \right)^{1/2} \left( \int_{\gamma(d)}^d (4 \two 1/t)^2 \right)^{1/2} \\
	& \leq \frac{8 d^{1/2}}{(\one 1/d)^{1/3}} \left( \int_0^d (\two 1/t)^2 \right)^{1/2}.
	\end{align*}
	We now use repeated applications of integration by parts to derive the inequality
	\begin{equation*}
	\int_0^d (\two 1/t)^2 \, dt \leq 3d ( \two 1/d)^2.
	\end{equation*}
	First note, using the substitution $y = \one 1/t$, and integrating by parts, that
	\begin{align*}
	\int_0^d (\two 1/ t)^2 \, dt 
	& =  \int_{\one 1/d}^{\infty} (\one y)^2 e^{-y}\, dy \\
	& =  \left( [-e^{-y} (\one y)^2 ]_{\one 1/d}^{\infty} + \int_{\one 1/d}^{\infty} \frac{2 (\one y ) e^{-y}}{y} \, dy \right)\\
	& = d ( \two 1/d)^2 + \int_{\one 1/d}^{\infty}\frac{ 2 (\one y ) e^{-y}}{y} \, dy.
	\end{align*}
	Examining the second summand, we use Cauchy-Schwartz, and integration by parts twice more to see, using the simplifications that $\one 1/d \geq \two  1/d \geq 2 \geq 1$, that
	\begin{align*}
	\lefteqn{\int_{\one 1/d}^{\infty} \frac{ (\one y ) e^{-y}}{y} \, dy}\\ 
	& \leq \left( \int_{\one 1/d}^{\infty} e^{-2y}\, dy\right)^{ \! 1/2} \left( \int_{\one 1/d}^{\infty} \frac{( \one y)^2 }{y^2} \,dy\right)^{\! 1/2} \\
	& =  \left( \left[\frac{- e^{-2y}}{2}\right]_{\one 1/ d}^{\infty}\right)^{\! 1/2} \left( \left[\frac{-(\one y)^2}{y}\right]_{\one 1/d}^{\infty} - \int_{\one 1/d}^{\infty} \frac{-2 \one y}{y^2}\, dy \right)^{\! 1/2}\\
	& \leq 2^{-1/2}d \left( \frac{(\two 1/d)^2}{\one 1/d} - \left( \left[\frac{2 \one y}{y}\right]_{\one 1/d}^{\infty} - \int_{\one 1/d}^{\infty} \frac{2}{y^2}\, dy\right)\right)^{\! 1/2}\\
	& = 2^{-1/2}d \left( \frac{(\two 1/d)^2}{\one 1/d} - \left( \frac{-2 \two 1/d}{\one 1/d} - \left[\frac{-2}{y}\right]_{\one 1/d}^{\infty}\right)\right)^{\! 1/2}\\
	& = 2^{-1/2}d \left( \frac{(\two 1/d)^2}{\one 1/d} + \frac{2 \two 1/d}{\one 1/d} + \frac{2}{\one 1/ d}\right)^{\! 1/2}\\
	& = \frac{2^{-1/2}d}{(\one 1/ d)^{1/2}}(( \two 1/d + 1)^2 + 1)^{1/2} \\
	& \leq \frac{2^{-1/2}d}{(\one 1/ d)^{1/2}}2^{1/2} (\two 1/ d + 1)\\
	& \leq \frac{2^{-1/2}d}{(\one 1/ d)^{1/2}} 2^{3/2} \two 1/d\\
	& \leq 2 \two 1/d.
	\end{align*}
	Combining with the original expression, we have, using again that $\two 1/d \geq 2$, that
	\begin{align*}
	\int_0^d (\two 1/ t)^2 \, dt 
	& \leq d( \two 1/d) \left( (\two 1/d) + 4\right) \\
	& \leq 3 d (\two 1/d)^2,
	\end{align*}
	as claimed.
	
	So we can conclude our estimates. Since~\eqref{T0small} implies that $\two 1/d \leq ( \one 1 /d)^{1/3}$, applying this inequality to~\eqref{extra}, using~\eqref{gammac}, and~\eqref{lglggamma}, we see that
	\begin{align} 
	\int_H |(\tilde{l}')^2 - (\tilde{w}')^2| & \leq 2 \left( 48 \gamma(d) (\two 1/\gamma(d))^2 + \frac{24d \two 1/d}{(\one 1/d)^{1/3}}\right)\nonumber \\
	& \leq  \frac{384 d (\two 1/d )^2}{(\one 1/d)^{2/3}} + \frac{48d \two 1/d}{(\one 1/d)^{1/3}}\nonumber \\
	& = \frac{g(d)}{(\one 1/d)^{1/3}} \left( \frac{384 \two 1/d}{(\one 1/d)^{1/3} } + 48 \right)\nonumber \\
	& \leq \frac{432 g(d)}{(\one1/d)^{1/3}}. 
	\label{ln'2-wn'2}
	\end{align}
	By~\eqref{triv} we have, since $\tilde{l}(\pm d) = \tilde{w}(\pm d)$,  that 
	\begin{align*}
	\int_{H} \left((u')^2 - (\tilde{l}')^2\right)  \geq  \int_{H} 2\tilde{l}'(u' -\tilde{l}')  =  2\tilde{l}' \int_{H}( u' - \tilde{l}')  =  2\tilde{l}' [ u - \tilde{l}]_{-d}^d 
	& =  2   \tilde{l}' [u -\tilde{w}]_{ - d}^{ d} \\
	& =  2(I_{+} - I_{-}).
	\end{align*}
	Since  
	\begin{align*}
	\int_{H} \left((u')^2 - (\tilde{w}')^2  \right)
	& = \int_{H} \big((u')^2 - (\tilde{l}')^2 \big)+ \int_{H}\big((\tilde{l}')^2 - (\tilde{w}')^2 \big)\\
	& \geq  \int_{H} \big( (u')^2 - (\tilde{l}')^2 \big) - \int_{H} |(\tilde{l}')^2 - (\tilde{w}')^2|, 
	\end{align*}
	the result follows from~\eqref{ln'2-wn'2}.
\end{proof}
An estimate established in the preceding proof gives easily the following important result.  The errors we incur in our boundary terms by introducing a jump discontinuity in the derivative of our new function $l_n$ are sufficiently small; they can be controlled by the integral over $H_n = [x_n - d_n, x_n + d_n]$ of a continuous function in $c_n \geq d_n$ taking value $0$ at $x_n$.
\begin{lemma}
	\label{partserror} 
	Let $n \geq 0$.
	
	Then 
	\[
	|I_{n,+} - E_{n,+}| + |I_{n, -} - E_{n, -}| \leq \frac{20\theta_n g(c_n)}{\one 1/ c_n}.
	\]
\end{lemma}
\begin{proof} 
	We just have to estimate $|(u - w_n) (x_n \pm d_n)|$.  Suppose that  $u(x_n) > w(x_n)$; the argument for the case in which $u(x_n)<w(x_n)$ is similar.  Suppose also that $b_n \geq a_n$, so by definition $c_n = b_n$.  The case in which $a_n > b_n$ is similar.  Then $u(t) \leq u(x_n + b_n)$ by the convexity of $u$ established in lemma~\ref{lemma1}, for all $t \in J_n$.  
	
	If $x_n - d_n \notin J_n$, then~\eqref{vileq5offJn} implies that $| ( u -w_n) ( x_n - d_n)| \leq 5 \theta_ng( d_n) \leq 5 \theta_n g(b_n)$, by monotonicity of $g$, since $d_n \leq b_n$.   
	
	Certainly $x_n + d_n \in J_n$, since $d_n \leq b_n$ by definition of $d_n$, so if also $x_n - d_n \in J_n$, by definition of $J_n$ we see that
	\[
	w(x_n) \leq w(x_n) + 3 \theta_n |g (d_n)| \leq u ( x_n \pm d_n )  \leq u ( x_n + b_n)  = w(x_n) + 3 \theta_n g ( b_n),
	\]
	so $0 < u ( x_n \pm d_n ) - w(x_n) \leq 3 \theta_n g (b_n)$. Hence, using~\ref{wfixxi} and~\ref{lipwn}, we have that, since $d_n \leq b_n$,
	\begin{align*}
	|(u-w_n)(x_n \pm d_n)| 
	& \leq |u (x_n \pm d_n) - w(x_n)| + |w_n (x_n) - w_n ( x_n \pm d_n)| \\
	&  \leq 3 \theta_n g(b_n) + 2 \theta_n g(d_n) \\
	& \leq 5 \theta_n g (b_n).
	\end{align*}
	
	Hence in both cases $|(u - w_n)(x_n \pm d_n)| \leq 5 \theta_n g(b_n)$. The result then follows by using~\eqref{l'-w'} with $t = d$ in~\eqref{pil}, and since $d_n \leq b_n$.
\end{proof}
The following is the key lemma, providing a positive lower bound for $\mathscr{L}_n (u) - \mathscr{L}_n(w_n)$.   We combine our estimates for $\mathscr{L}_n$ across the whole domain $[-T, T]$, integrating by parts off $\bigcup_{i = 1}^{n} H_i$, and using the estimates from lemmas~\ref{u'-wn'overHn} and~\ref{partserror} on each $H_i$. The argument is made more straightforward by assuming that the intervals $\tilde{J}_i$ are small in a certain sense, which implies that the intervals on which we work do not overlap.  Should this assumption fail for some $n$, then, as later lemmas will show, this means that the discrepancy $u-w$ around $x_n$ is sufficiently large that we may ignore the fine detail of our construction at and beyond the stage $n$, and we can conclude the proof using just $\mathscr{L}_{n-1}$.
\begin{lemma}
	\label{ngood} 
	Suppose  $n \geq 0$ is such that for all $0 \leq j \leq n$, 
	\begin{gather}
	\tilde{J}_k \cap Y_j = \emptyset\ \textrm{for all $ 0 \leq k \leq  j-1 $; and} \label{mu}\\
	\tilde{J}_j \subseteq Y_j. \label{xi}
	\end{gather}
	Then 
	\[
	\mathscr{L}_n (u) - \mathscr{L}_n (w_n) \geq  \sum_{i=0}^n \left(\frac{\theta_i g(c_i)}{(\one1/c_i)^{1/3}}\right)  + \int_{[-T, T] \setminus \bigcup_{i=0}^n H_i} |u - w_n|.
	\]
\end{lemma}
\begin{proof} 
	By~\ref{wn=wn-1} and assumption~\eqref{mu} we have for all $ 0 \leq k \leq  j -1 \leq j \leq n$ that $w_j = w_k$ on $\tilde{J}_k$, in particular that
	\begin{equation}
	\label{wn=wk}
	w_n = w_k,\ w_n' = w_k',\ \textrm{and}\ w_n'' = w_k''\ \textrm{on $\tilde{J}_k$, whenever both sides exist}.
	\end{equation}
	Also, assumptions~\eqref{xi} and~\eqref{mu} together imply that $\{\tilde{J}_i\}_{i=0}^n $ is pairwise disjoint.
	
	Now, let $0 \leq i \leq n$.    We split up the integral into summands which we shall tackle separately:
	\begin{align*}
	\lefteqn{
		\int_{\tilde{J}_i} \left((u')^2 + \phi(t, u-w_i) - (w_i')^2\right)
	}\\
	& =  \int_{\tilde{J}_i} \left(\phi^1 (t, u-w_i) + \phi^2 (t, u-w_i) \right) +  \int_{H_i} \big((u')^2  - (w_i')^2\big)+ \int_{\tilde{J}_i \backslash H_i} \big((u')^2  - (w_i')^2\big)  \\
	& \geq  \int_{\tilde{J}_i} \phi^1(t, u-w_i) + \int_{H_i} \big((u')^2  - (w_i')^2\big) + \int_{\tilde{J}_i \backslash H_i} \left(\phi^2(t, u-w_i) +(u')^2  - (w_i')^2\right).
	\end{align*}
	Now, by lemma~\ref{intphibig} (note that this applies by assumption~\eqref{xi}) and lemma~\ref{u'-wn'overHn}, and since $c_i \geq d_i$ and $\theta_i \geq 1$, 
	\begin{align*} 
	\int_{\tilde{J}_i} \phi^1 (t, u-w_i) + \int_{H_i} \left((u')^2 - (w_i')^2\right)
	& \geq  \int_{\tilde{J}_i} \tilde{\phi}_i^1(t, u-w_i) + \int_{H_i} \left((u')^2 - (w_i')^2\right) \\
	& \geq \frac{ 453 \theta_i g(c_i)}{( \one 1/c_i)^{1/3}} + 2(I_{i,+} - I_{i, -}) - \frac{432 g(d_i)}{(\one 1/ d_i)^{1/3}}\\
	& \geq  \frac{21 \theta_i g(c_i)}{(\one 1/c_i)^{1/3}} + 2(I_{i,+} - I_{i, -}).
	\end{align*}
	So, combining, we have that
	\begin{align}
	\int_{\tilde{J}_i}\left( (u')^2 + \phi(t, u-w_i) - (w_i')^2 \right)
	& \geq   \frac{21 \theta g(c_i)}{(\one 1/c_i)^{1/3}} + 2(I_{i,+} - I_{i, -}) \nonumber\\
	& \phantom{=} {}+ \int_{\tilde{J}_i \backslash H_i} \left(\phi^2(t, u-w_i) + (u')^2  - (w_i')^2\right).\label{Jicomplete2}
	\end{align}
	
	Now, for any $t \in [-T, T]$, let $\mathcal{I}_n (t) \df \{ j = 0, \ldots, n : t \in Y_j\}$.
	We now show by an easy induction on $n$ that 
	\begin{equation}
	\label{omicron}
	\sum_{j \in \mathcal{I}_n(t)}\psi_j^2(t) \geq 2|w_n''(t)| + 1 + 2^{-(n-1)},
	\end{equation}
	for almost every $ t \in [-T, T]$.  For $n = 0$, we have by definition of $\psi^2$ that for all $ t \neq x_0$, $\psi_0^2 (t) =  3 +  2|w_0''(t)|$, as required. Suppose the result holds for all $0 \leq i \leq n-1$, where $ n \geq 1$.  
	Let $i_n(t) \leq n$ denote the greatest index in $\mathcal{I}_n (t)$, i.e.\ the greatest index $j \leq n$ such that $t \in Y_j$.
	By~\ref{wn=wn-1} we have that $w_n''(t) = w_{i_n(t)}''(t)$ almost everywhere.  If $t \in (x_{i_n(t)} - \tau_{i_n(t)}, x_{i_n(t)} + \tau_{i_n(t)})$, then $w_{i_n(t)}''(t) =  \tilde{w}_{i_n(t)} ''(t)$ by~\ref{wn=twn}, and by definition of $\psi^2$, for $t \neq x_{i_n(t)}$,  
	\[
	\sum_{j \in \mathcal{I}_n (t)} \psi_j^2(t) 
	\geq \psi_{i_n(t)}^2 (t) 
	= 3 + 2|\tilde{w}_{i_n(t)}''(t)| 
	\geq 
	1 + 2^{-(n-1)} + 2|w_{i_n(t)}''(t)|,
	\]
	as required.  If $t \notin [x_{i_n(t)} - \tau_{i_n(t)}, x_{i_n(t)} + \tau_{i_n(t)}]$ (note then necessarily $i_n(t) \geq 1$ since $\tau_0 = T$), then $|w_{i_n(t)} ''(t)| \leq |w_{i_n(t)-1}''(t)| + 2^{-i_n(t)}$ almost everywhere by~\ref{wn''-wn-1''}. So by inductive hypothesis 
	\begin{align*} 
	\sum_{j \in \mathcal{I}_n(t)} \psi_j^2 (t) 
	& \geq  \sum_{j \in \mathcal{I}_{i_n(t)-1}(t)} \psi_j^2 (t) \\
	& \geq  2 |w_{i_n(t)-1}''(t)| + 1 + 2^{-((i_n(t)-1)-1)} \\
	& \geq  2|w_{i_n(t)}''(t)| - 2 \cdot 2^{-i_n(t)} + 1 + 2^{-((i_n(t)-1)-1)} \\
	& =  2|w_{i_n(t)}''(t)| +1 + 2^{-(i_n(t)-1)} \\
	& \geq  2|w_n''(t)| + 1 + 2^{-(n-1)},
	\end{align*} 
	as required for~\eqref{omicron}.
	
	Given this, now consider $ t \notin \bigcup_{i=0}^n \tilde{J}_i$. Then since by definition $\tilde{J} _j \supseteq J_j$ for all $j \geq 0$,~\eqref{vileq5offJn}  implies that $|(u-w_n) (t)| \leq 5 \theta_j |g_j (t)|$ for all $0 \leq j \leq n$.  Therefore  $\tilde{\phi}_j^2 (t, u - w_n) =  \psi_j^2 (t)|u -w_n|$ by definition of $\tilde{\phi}^2$, for $j \in \mathcal{I}_n (t)$.  Thus almost everywhere, we have by~\eqref{omicron} that 
	\begin{align*}
	\phi^2(t, u -w_n) - 2(u -w_n) w_n'' 
	& \geq  \sum_{j \in \mathcal{I}_n (t)} \left(\tilde{\phi}_j^2 (t, u-w_n)\right) - 2|u-w_n| |w_n''| \\
	& =  \sum_{j \in \mathcal{I}_n (t)} \left(\psi_j^2 (t)|u-w_n|\right) - 2|u-w_n| |w_n''| \\
	& =  |u-w_n| \left(\sum_{j \in \mathcal{I}_n (t)} (\psi_j^2 (t)) - 2|w_n''(t)|\right) \\
	& \geq | u -w_n|.
	\end{align*}
	
	Now, let $ t \in \tilde{J}_i \setminus H_i$ for some $0 \leq i \leq n$.  Then note that we must have $i \geq 1$, since $\tau_0 = T$, so $H_0 = \tilde{J}_0$. Since $\{\tilde{J}_j\}_{j=0}^n $ is pairwise disjoint, we have that $ t \notin \tilde{J}_j$ for $j \leq  i-1$.  Hence, again by~\eqref{vileq5offJn}, $|(u - w_i)| \leq 5 \theta_j |g_j (t)|$, so by definition of $\tilde{\phi}^2$, $\tilde{\phi}_j^2(t,  u -w_i) = \psi_j^2 (t) |u -w_i|$ for all $j \leq  i -1$, recalling assumption~\eqref{xi}.   Since $t \notin H_i$, we have $t \notin [x_i - \tau_i, x_i +,\tau_i]$, and hence that $|w_i''(t)| \leq |w_{i-1}''(t)| + 2^{-i}$ almost everywhere by~\ref{wn''-wn-1''}.  Hence by~\eqref{omicron} we have almost everywhere that
	\begin{align*}
	\sum_{j \in \mathcal{I}_{i-1} (t)}\psi_j^2(t)  \geq  1 + 2|w_{i-1}''(t)| + 2^{-(i-2)} 
	& \geq  1 + 2|w_i''(t)| - 2^{-(i-1)} + 2^{-(i-2)} \\
	& \geq  1 + 2|w_i ''(t)|,
	\end{align*}
	and so 
	\begin{align*}
	\phi^2(t,  u -w_i) - 2(u -w_i)w_i'' 
	& \geq \sum_{j \in \mathcal{I}_{i-1} (t)}\left(\tilde{\phi}_j^2 (t, u-w_i)\right) - 2|u-w_i ||w_i''| \\
	& = \sum_{j \in \mathcal{I}_{i-1} (t)}\left(\psi_j^2 (t)|u-w_i|\right) - 2|u-w_i||w_i''| \\
	& \geq |u -w_i|.
	\end{align*}
	Thus we have for almost every $t \notin \bigcup_{i=0}^n H_i$, noting the argument on $\tilde{J}_i \backslash H_i$ above applies by~\eqref{wn=wk},  that 
	\[
	\phi^2 (t, u-w_n) - 2(u-w_n) w_n'' \geq |u-w_n|,
	\]
	and hence that
	\begin{equation}
	\label{ineq}
	\int_{[-T, T] \setminus \bigcup_{i=0}^n H_i} \left(\phi^2 (t, u -w_n) - 2(u -w_n) w_n''\right) \geq \int_{[-T, T] \setminus \bigcup_{i=0}^n H_i} |u -w_n|.
	\end{equation}
	The reason for making this estimate is that we want to integrate $(u' - w_n') w_n'$ by parts on $[-T,T] \setminus \bigcup_{i=0}^n H_i$.  Under our standing assumption that $u(x_i) \neq w(x_i)$ for all $i \geq 0$, we see immediately that this is possible, since $(u - w_n)$ and $w_n'$ are bounded and absolutely continuous on $[-T,T] \setminus \bigcup_{i=0}^n H_i$ by~\ref{wn'lip}, and thus $(u - w_n) w_n'$ is absolutely continuous on $[-T,T] \setminus \bigcup_{i=0}^n H_i$.  However,  in the general case in which $w(x_j) = u(x_j)$ for some $0 \leq j \leq n$, and thus that $w_n (x_j) = u(x_j)$,  we have to argue a little more carefully.
	
	We claim that even in this general case the parts formula is still valid on $[-T,T] \setminus \bigcup_{i=0}^n H_i$; this is the assertion that $(u-w_n) w_n'$ can be written as an indefinite integral on  $[-T,T] \setminus \bigcup_{i=0}^n H_i$.   The argument of the preceding paragraph gives us that $(u-w_n) w_n'$ is absolutely continuous on subintervals bounded away from all $x_j$ with $u(x_j) = w(x_j)$.  Fix such an index $0 \leq j \leq n$.
	
	Let $t_j = t_{j, n} = \min \{ \sigma_n, \tau_j\}$.  \label{twnhoodofxn}By~\eqref{Ynfarfromxi}, and since $\{\sigma_n\}_{n=1}^{\infty}$ is decreasing, we know that $[x_j - \sigma_n, x_j + \sigma_n] \cap Y_m = \emptyset$ for all $j + 1 \leq  m \leq n$.  So by~\ref{wn=wn-1} and~\ref{wn=twn}, $w_n =  \tilde{w}_j + \rho_j$ on $[x_j - t_j, x_j + t_j]$.
	It suffices to check that $(u-w_n) w_n'$ can be written as an indefinite integral on $(x_j - t_j, x_j + t_j)$.  We check that
	\[
	\int_{x_j - t_j}^{x_j} \left((u-w_n) w_n'\right)'(s) \, ds =  - ((u-w_n)(x_j - t_j)) w_n'(x_j - t_j);
	\]
	the corresponding equality on the right of $x_j$ follows similarly (recall that $u(x_j) - w_n (x_j) = 0$).
	
	We know that on those subintervals of $(x_j - t_j, x_j + t_j)$ bounded away from $x_j$, $(u-w_n) w_n'$ is absolutely continuous.  We claim that  $((u-w_n) w_n')' \in L^1(x_j - t_j, x_j + t_j)$.  Given this, we can use the dominated convergence theorem to get the required result as follows.
	
	Since $J_j = \emptyset$, we see by~\eqref{vileq5offJn} that $|(u-w_n)(t)| \leq 5 \theta_j | g_j (t)|$ on $(x_j - t_j, x_j + t_j)$.  Thus we see by~\eqref{|tw'|} that
	\begin{align*}
	| (u-w_n) (t)w_n'(t)| = |(u - w_n)(t) \tilde{w}_j'(t)|
	& \leq 5\theta_j  |g_j (t)| 3 \two 1/ |t - x_j| \\
	& \leq 15 \theta_j |t - x_j| ( \two 1/ |t - x_j|)^2\\
	& \to 0\ \textrm{as $t \to x_j$}.
	\end{align*}
	So now, assuming that the dominated convergence theorem can be applied, we see that
	\begin{align*}
	- ((u-w_n)(x_j - t_j)) w_n'(x_j - t_j)
	& = \lim_{t \uparrow x_j }\left( ((u-w_n)(t)) w_n'(t)\right) \\
	& \phantom{=} {}- ((u-w_n)(x_j - t_j)) w_n'(x_j - t_j) \\
	&=\lim_{t \uparrow x_j } \int_{x_j - t_j}^t \left((u-w_n) w_n'\right)'(s) \, ds\\
	& = \int_{x_j - t_j}^{x_j}\left((u-w_n) w_n'\right)'(s) \, ds,
	\end{align*}
	as required. It just remains to justify our use of the dominated convergence theorem, i.e.\ to show that $((u - w_n) w_n')' \in L^1 (x_j - t_j, x_j + t_j)$. 
	Again, noting that~\eqref{vileq5offJn} still holds, we have, using~\ref{wn=twn} and Cauchy-Schwartz, that
	\begin{align*}
	\lefteqn{\int_{x_j - t_j}^{x_j + t_j} |((u-w_n) w_n')'| }\\
	& = \int_{x_j - t_j}^{x_j + t_j} |((u-\tilde{w}_j)\tilde{w}_j')'| \\
	& \leq  \int_{x_j - t_j}^{x_j + t_j} |(u-\tilde{w}_j) \tilde{w}_j''| + \int_{x_j - t_j}^{x_j + t_j} |(u' - \tilde{w}_j') \tilde{w}_j'| \\
	& \leq \int_{x_j- t_j}^{x_j+ t_j} |5 \theta_j g_j \tilde{w}_j''|  +  \int_{x_j - t_j}^{x_j + t_j} |u'\tilde{w}_j'| +   \int_{x_j - t_j}^{x_j + t_j} |\tilde{w}_j'|^2 \\
	& \leq   5 \theta_j \int_{ - t_j}^{t_j} |g \tilde{w}''| +  \left(\int_{x_j - t_j}^{x_j + t_j} |u'|^2\right)^{\! 1/2}\left(\int_{- t_j}^{t_j} |\tilde{w}'|^2\right)^{\! 1/2}  + \int_{- t_j}^{t_j} |\tilde{w}'|^2.
	\end{align*}
	This right hand side is finite by~\eqref{wn''cont}, and since $u, \tilde{w} \in W^{1,2}(-T, T)$.
	
	So, using~\eqref{triv}, and recalling that $u(\pm T) = w(\pm T)$, and using~\eqref{ineq} (recalling that $H_i \subseteq \tilde{J}_i$), we have, integrating by parts as we now know we can do, that 
	\begin{align}
	\lefteqn{ \int_{[-T, T] \setminus \bigcup_{i=0}^n H_i}\left( \phi^2 (t, (u - w_n)) + (u')^2 - (w_n')^2\right)}\nonumber\\
	& \geq \int_{[-T, T] \setminus \bigcup_{i=0}^n H_i}\left( \phi^2 (t, u -w_n) + 2(u' -w_n') w_n'\right)\nonumber\\
	& =  2[(u - w_n) w_n']_{[-T, T] \setminus \bigcup_{i=0}^n H_i} + \int_{[-T, T]\setminus \bigcup_{i=0}^n H_i}\left( \phi^2(t, u -w_n) - 2 (u-w_n) w_n''\right) \nonumber\\
	& = -2 \sum_{i=0}^n [(u -w_i) w_i']_{x_i - d_i}^{x_i + d_i} + \int_{[-T, T] \setminus \bigcup_{i=0}^n H_i}\left( \phi^2 (t, u -w_n) - 2(u -w_n) w_n''\right) \nonumber\\
	& \geq  -2 \sum_{i=0}^n (E_{i, +} - E_{i, -}) + \int_{[-T, T]\setminus \bigcup_{i=0}^n H_i} |u-w_n|. \label{offHi}
	\end{align}
	So, since $\{\tilde{J}_i\}_{i=0}^n$ is pairwise disjoint, we can argue as follows, using~\eqref{wn=wk}, ~\eqref{Jicomplete2},~\eqref{offHi}, and lemma~\ref{partserror} to see that
	\allowdisplaybreaks{
	\begin{align*} 
	\lefteqn{\mathscr{L}_n (u) - \mathscr{L}_n (w_n)}\\
	& =  \int_{-T}^T \left((u')^2 + \phi(t, u-w_n) - (w_n')^2\right) \\
	& =  \int_{\bigcup_{i=0}^n \tilde{J}_i}\left((u')^2 + \phi(t, u -w_n) - (w_n')^2\right) + \int_{[-T, T] \setminus \bigcup_{i=0}^n \tilde{J}_i}\left((u')^2 + \phi(t, u - w_n) - (w_n')^2\right) \\
	& =  \sum_{i=0}^n \int_{\tilde{J}_i} \left((u')^2 + \phi(t, u -w_i) - (w_i')^2\right) + \int_{[-T,T]\setminus \bigcup_{i=0}^n \tilde{J}_i}\left( (u')^2 + \phi(t, u-w_n) - (w_n')^2\right)\\
	& \geq  \sum_{i=0}^n\left(\frac{21\theta_i g(c_i)}{(\one1/c_i)^{1/3}} + 2(I_{i,+} - I_{i,-}) +  \int_{\tilde{J}_i \backslash H_i} \left(\phi^2(t, u-w_i) + (u')^2 - (w_i')^2\right)  \right) \\
	&\phantom{\geq} {}+ \int_{[-T, T] \setminus \bigcup_{i=0}^n \tilde{J}_i} \left(\phi^2 (t, u -w_n) + (u')^2 - (w_n')^2\right) \\
	& \geq  \sum_{i=0}^n \left(\frac{21\theta_i g(c_i)}{(\one1/c_i)^{1/3}} + 2(I_{i,+} - I_{i,-})\right) + \int_{[-T, T] \setminus \bigcup_{i=0}^n H_i} \left(\phi^2 (t, u-w_n) + (u')^2 - (w_n')^2\right)  \\
	& \geq   \sum_{i=0}^n \left(-2(E_{i, +} - E_{i, -})  + \frac{21\theta_i g(c_i)}{(\one1/c_i)^{1/3}} + 2(I_{i,+} - I_{i,-}) \right)+ \int_{[-T,T] \setminus \bigcup_{i=0}^n H_i}|u-w_n|\\
	& =  \sum_{i=0}^n \left( 2 \left((I_{i, +} - E_{i, +}) - (I_{i, -} - E_{i, -})\right) + \frac{21\theta_i g(c_i)}{(\one1/c_i)^{1/3}}\right) + \int_{[-T, T]\setminus \bigcup_{i=0}^n H_i} |u - w_n| \\
	& \geq  \sum_{i=0}^n \left( \frac{21\theta_i g(c_i)}{(\one1/c_i)^{1/3}} - 2\left(|I_{i, +}- E_{i, +}| + |I_{i, -} - E_{i, -}|\right) \right) + \int_{[-T,T] \setminus \bigcup_{i=0}^n H_i} |u-w_n| \\
	& \geq \sum_{i=0}^n \left( \frac{\theta_i g(c_i)}{(\one1/c_i)^{1/3}}\right) + \int_{[-T, T] \setminus \bigcup_{i=0}^n H_i} |u-w_n|.
	\qedhere
\end{align*}
}
\end{proof} 
\begin{corollary}
	\label{cor} 
	Suppose for all $n \geq 0$ that the assumptions~\eqref{mu} and~\eqref{xi} hold.  
	
	Then 
	\[
	\mathscr{L}(u) - \mathscr{L} (w)  \geq \sum_{i=0}^{\infty}\left( \frac{\theta_i g(c_i)}{(\one1/c_i)^{1/3}}\right) + \int_{[-T, T] \setminus \bigcup_{i=0}^{\infty}H_i}  |u-w| > 0.
	\]
\end{corollary}
\begin{proof}  
	This follows by the preceding lemma and the dominated convergence theorem as follows. 
	It is straightforward to see that
	\[
	\lim_{n \to \infty} \left(|u-w_n| \mathds{1}_{[-T, T] \setminus \bigcup_{i=0}^n  H_i}\right)(t) = \left(|u-w|\mathds{1}_{[-T, T] \setminus \bigcup_{i=0}^{\infty}H_i}\right)(t)
	\]
	for all $t \in [-T, T]$: for $ t\in H_k$ for some $k\geq 0$, eventually both sides are  $0$; for $t \notin \bigcup_{i=0}^{\infty} H_i$, we see that
	\begin{align*}
	\big|\mathds{1}_{[-T, T] \setminus \bigcup_{i=0}^{n}H_i}(t)|(u - w_n )(t)| & - \mathds{1}_{[-T, T] \setminus \bigcup_{i=0}^{\infty}H_i}(t) |(u -w)(t)|\big|\\
	&=  ||(u - w_n)(t)| - |(u - w)(t)||\\
	& \leq |(u - w_n)(t) - (u - w)(t)| \\
	& = |w_n (t) - w(t)| \\
	& \to 0 \quad \textrm{as $ n \to \infty$}.
	\end{align*}
	Moreover, since $w_n \to w$ uniformly, we have  that
	\[
	\sup_{n \geq 0}\left\| |u -w_n| \mathds{1}_{[-T, T] \setminus \bigcup_{i=0}^n H_i}\right\|_{\infty} \leq \sup_{n \geq 0} \|u -w_n\|_{\infty} < \infty.
	\]
	So the dominated convergence theorem implies that
	\begin{align*}
	\lim_{n \to \infty}\int_{[-T, T] \setminus \bigcup_{i=0}^n H_i} |u -w_n| 
	& =  \lim_{n \to \infty}\int_{-T}^{T} \left(|u -w_n|\mathds{1}_{[-T, T] \setminus \bigcup_{i=0}^n H_i}\right) \\
	& =  \int_{-T}^{T} \lim_{n \to \infty}\left(|u -w_n| \mathds{1}_{[-T, T] \setminus \bigcup_{i=0}^{n} H_i} \right) \\
	& =  \int_{-T}^{T} \left(|u -w| \mathds{1}_{[-T, T] \setminus \bigcup_{i=0}^{\infty}H_i}\right)\\
	& =  \int_{[-T, T] \setminus \bigcup_{i=0}^{\infty}H_i}|u - w|.
	\end{align*}
	Lemma~\ref{finitejump} and~\eqref{iv} give that
	\[
	\lim_{n \to \infty} (\mathscr{L}_n(u) - \mathscr{L}_n (w_n) ) = \mathscr{L}(u) - \mathscr{L}(w).
	\]
	So since, by assumption, lemma~\ref{ngood} applies for all $n\geq 0$, we can pass to the limit on each side of the inequality in the conclusion of the lemma to get the required result.
	
	We note that in the general case we do indeed have strict inequality, as is necessary for the contradiction proof.  If $u(x_n) \neq w(x_n)$ for some $n\geq 0$, then $c_n > 0$ and so the infinite sum is strictly positive.  If $u(x_n) = w(x_n)$ for all $n\geq 0$, then $[-T, T] \setminus \bigcup_{i=0}^{\infty}H_i = [-T, T]$, so on the assumption that $u \neq w$, where both are continuous functions, the integral term must be strictly positive.
\end{proof}
The arguments of the previous lemma and its corollary relied on the intervals we have to give special attention, the $\tilde{J}_j$, being small enough that they did not escape $Y_j$, or overlap with later $Y_k$ and hence possibly $\tilde{J}_k$.  The trick is now that should one of these assumptions fail, thus apparently making the proof more complicated,  in fact this means that we can ignore the modifications we made at stage $j$ and beyond.  That one of our assumptions fails for $j$ means that $\tilde{J}_j$ is too large, which by the very definition of $\tilde{J}_j$ implies the graph of $u$ is far away from that of $w$ on a set of large measure around $x_j$.  We have chosen our constants so that this large difference between $u$ and $w$ around $x_j$ gives enough weight to our Lagrangian that we can discard all modifications we made to $w_{j-1}$ and hence to $\mathscr{L}_{j-1}$ and work just with these instead; the error so incurred is small enough that it is absorbed into this extra weight.  Very roughly, if $u$ misses $w$ at $x_j$ by an apparently inconveniently large amount, then we don't have to worry about the fine detail of our variational problem at and beyond the scale $j$.
\begin{lemma}
	\label{lemma8} 
	Let $n\geq 1$ be such that assumptions~\eqref{mu} and~\eqref{xi} hold for $n-1$, but for some $0 \leq k \leq n-1$ we have that $\tilde{J}_k \cap Y_n \neq \emptyset$, i.e.~\eqref{mu} fails for $n$.
	
	Then 
	\[
	\mathscr{L}_{n-1}(u) - \mathscr{L}_{n-1}(w_{n-1}) \geq T_n^2.
	\]
\end{lemma}
\begin{proof} 
	That~\eqref{mu} fails for $n$ implies that $c_k \geq T_n$, otherwise choosing $t \in \tilde{J}_k \cap Y_n$ we would have, since~\ref{t1} implies that $T_n \leq |x_n - x_k| /2$, that
	\[
	|x_n - x_k| \leq |x_n - t| + | t- x_k| \leq T_n + c_k < 2T_n \leq |x_n - x_k|,
	\]
	which is a contradiction. So, applying lemma~\ref{ngood} to $n-1$ we see, using this fact, that $\theta_k \geq 1$,  and since~\eqref{T0small} implies that $c_k \leq c_k^{1/3} \leq ( 1/ \one 1/c_k)^{1/3}$, that 
	\begin{align*}
	\mathscr{L}_{n-1}(u) - \mathscr{L}_{n-1}(w_{n-1}) 
	& \geq  \sum_{i=0}^{n-1} \left(\frac{\theta_i g(c_i)}{(\one1/c_i)^{1/3}}\right) + \int_{[-T, T] \setminus \bigcup_{i=0}^{n-1}H_i} |u -w_{n-1}| \\
	& \geq  \frac{\theta_k g(c_k)}{(\one1/c_k)^{1/3}}\\
	& \geq  \frac{c_k \two 1/c_k}{(\one 1/ c_k)^{1/3}}  \\
	& \geq  c_k^2 \two 1/c_k\\
	& \geq  c_k^2 \\
	& \geq  T_n^2.
	\qedhere
	\end{align*}
\end{proof}
\begin{lemma} 
	\label{lemma9}
	Let $n\geq 1$ be such that assumption~\eqref{mu} holds for $n$, assumption~\eqref{xi} holds for $n-1$, but $\tilde{J}_n \nsubseteq Y_n$, i.e.~\eqref{xi} fails for $n$.  
	
	Then 
	\[
	\mathscr{L}_{n-1}(u) - \mathscr{L}_{n-1}(w_{n-1}) \geq T_n^2.
	\]
\end{lemma}
\begin{proof}  
	We suppose that $c_n = b_n$.  The case in which $a_n > b_n$ differs only in trivial notation.  That~\eqref{xi} fails for $n$ implies that $b_n \geq T_n$.  That~\eqref{mu} holds for $n$ implies in particular that $Y_n \cap \bigcup_{i=0}^{n-1}\tilde{J}_i = \emptyset$.  Thus by lemma~\ref{ngood} for $n-1$, since by definition $H_i \subseteq \tilde{J}_i$ for all $0 \leq i \leq n-1$,
	\begin{align*}
	\mathscr{L}_{n-1}(u) - \mathscr{L}_{n-1}(w_{n-1}) 
	& \geq  \sum_{i=0}^{n-1}\left( \frac{\theta_i g(c_i)}{(\one1/c_i)^{1/3}}\right) + \int_{[-T, T] \setminus \bigcup_{i=0}^{n-1}H_i}|u -w_{n-1}| \\
	& \geq  \int_{[-T, T] \setminus \bigcup_{i=0}^{n-1}\tilde{J}_i}|u-w_{n-1}| \\
	& \geq  \int_{Y_n}|u -w_{n-1}|\\
	& \geq  \int_{x_n}^{x_n + T_n}|u-w_{n-1}|. 
	\end{align*}
	But we know by~\eqref{vnbig}, also using~\ref{cvg}, monotonicity of $g$,~\ref{r3}, and that $\theta_n \geq 1$, that for $ t \in [x_n, x_n + b_n]$ we have 
	\begin{align*}
	|(u-w_{n-1})(t)| \geq |(u-w_n) (t)| - |w_n (t) - w_{n-1}(t)| 
	& \geq \theta_n g(b_n) - \|w_n - w_{n-1}\|_{\infty} \\
	& \geq g(T_n) - 5K_n g(R_n) \\
	& \geq g(T_n)/2.
	\end{align*}
	Hence we see, since $[x_n, x_n + T_n] \subseteq [x_n, x_n + b_n]$, and since $\two 1/T_n \geq 1$, that
	\[
	\mathscr{L}_{n-1}(u) - \mathscr{L}_{n-1}(w_{n-1})  \geq  \int_{x_n }^{x_n+T_n} g(T_n)/2  = T_n g(T_n)/2 \geq T_n^2.
	\qedhere
	\]
	\end{proof}
We are now in a position to conclude our argument. If our crucial assumptions~\eqref{mu} and~\eqref{xi} hold for all $n \geq 0$, then we are in the case of corollary \ref{cor} and we are done.  Otherwise, choose the least $n \geq 0$ such that one of~\eqref{mu} or~\eqref{xi} fails.  We observe that then $ n\geq 1$ necessarily, since $\tilde{J}_0 \subseteq [-T, T]$. 

Suppose $n\geq 1$ is such that~\eqref{mu} fails for $n$.  Then we are in the case of lemma~\ref{lemma8} and we see by lemma~\ref{finitejump} that 
\[ 
\mathscr{L}(u) - \mathscr{L}(w)  \geq \mathscr{L}_{n-1}(u) - \mathscr{L}_{n-1}(w_{n-1}) - \frac{T_n^2}{2}  \geq  \frac{T_n^2}{2}  >  0.
\]

Suppose $n \geq 1$ is such that~\eqref{mu} holds for $n$ but~\eqref{xi} fails.  Then we are in the case of lemma~\ref{lemma9} and we see again by lemma~\ref{finitejump} that
\[
\mathscr{L} (u) - \mathscr{L}(w) \geq  \mathscr{L}_{n-1}(u) - \mathscr{L}_{n-1}(w_{n-1}) - \frac{T_n^2}{2} \geq  \frac{T_n^2}{2}
>  0.  
\]
This contradicts the choice of $u$ as a minimizer, so we know that no minimizer $u \neq w$ exists.
Letting $\{x_n\}_{n=0}^{\infty}$ be an enumeration of $\mathbb{Q} \cap(-T, T)$ concludes the proof.
\section{Approximation and variations}
\label{sec:variations}
In this section we investigate different ways of approximating the minimum value of a variational problem.  Throughout we continue to assume only that the Lagrangian $L$ is continuous.
\begin{Question}
	Let $v \in W^{1,1}(a,b)$ be a minimizer of $\mathscr{L}$ over $\scra{v(a)}{v(b)}$.  Does there exist a sequence $u_j \in W^{1, \infty} (a,b)\cap \scra{v(a)}{v(b)}$ such that $\mathscr{L}(u_j) \to \mathscr{L}(v)$? 
	\label{Qu:Lav}
\end{Question}
The answer to this in general is a well-known ``no'', and in situations where the answer is negative, the {\it Lavrentiev phenomenon} is said to occur.  Lavrentiev~\cite{Lavrentiev-1926} gave the first example, and Mani\`a~\cite{Mania-1934} gave an example with a polynomial Lagrangian.  Both examples have Lagrangians which vanish along the minimizing trajectory.  Ball and Mizel~\cite{Ball-Mizel-1985} gave the first superlinear examples, with polynomial $L$ for with $L_{pp} \geq \epsilon > 0$ for some $\epsilon > 0$.

This settles the question of whether in general the minimum value can be approximated by Lipschitz trajectories: no.  A related question is whether the minimum value can be approximated by {\it adding} Lipschitz functions to the minimizing trajectory.  One way of motivating this question is to consider that classically one finds minimizers by taking the first variations in the direction of functions $u \in C_0^{\infty}(a,b)$, i.e.\ computing $\frac{d}{d\gamma} \mathscr{L}(v + \gamma u) \vert_{\gamma=0}$, and looking for functions $v$ for which this value is $0$ for all such $u$.  Under appropriate assumptions on $L$ one can thereby derive the Euler-Lagrange equation, and look for minimizers among solutions to that.  But how does the function $\gamma \mapsto \mathscr{L}(v + \gamma u)$ behave in general?

First we investigate this question forgetting for the moment that $u$ is taken to be Lipschitz.  
\begin{Question}
	Let $v \in W^{1,1}(a,b)$.  Does there exist a sequence $u_j \in \scra{v(a)}{v(b)}$, $u_j \neq v$, such that $\mathscr{L}(u_j) \to \mathscr{L}(v)$? 
	\label{Qu:JB}
\end{Question}
The answer is an easy but apparently unrecorded ``yes'', assuming only continuity of $L$, and holds for vector-valued trajectories $v$ without too much extra work.
\begin{theorem}
	\label{thm:JB}
	Let $v \in W^{1,1}((a,b); \mathbb{R}^n)$ be such that $t \mapsto L(t, v(t), v'(t))$ is integrable, $U \subseteq (a,b)$ be open and non-empty, and   $\epsilon > 0$.
	
	Then there exists $u \in \scra{v(a)}{v(b)}$ such that   $\emptyset \neq \{ t \in [a,b] : u(t) \neq v(t) \} \subseteq U$ and $|\mathscr{L}(u) - \mathscr{L}(v) | \leq \epsilon$.
\end{theorem}
\begin{remark}
	Our method of proof gives the immediate further information that $u$ is locally Lipschitz on $\{t \in [a,b] : u(t) \neq v(t) \}$.
\end{remark}
If the function $v$ is somewhere locally Lipschitz in $U$, then the approximation is obvious and can be done by adding to $v$ a non-zero function of small norm in $W_0^{1,\infty}((a,b); \mathbb{R}^n)$ which is zero where $v$ is not locally Lipschitz.  If $v$ is nowhere locally Lipschitz in $U$---which if $v$ is a minimizer implies that $L$ does not admit a partial regularity theorem---then the approximation is only slightly less obvious, and is done by replacing $v$ with an affine function on appropriately small intervals.  Notice however that the difference between $v$ and the approximating function is non-Lipschitz.

The proof requires an easy lemma.  For $v \in W^{1,1}((a,b); \mathbb{R}^n)$, $m > 0$, and $t \in (a,b)$, define 
\begin{align*}
E_t & \df \{ s \in [a,b] : \| v(s) - v(t)\| > m | s -t | \}; \ \textrm{and}\\
M_t & \df \{ s \in [a,b] : \| v(s) - v(t)\| = m | s -t | \}.
\end{align*}
\begin{lemma}
	\label{levelset}
	Let $v \in W^{1,1}((a, b); \mathbb{R}^n)$ and $m > 0$, and suppose that $t \in (a,b)$ is such that $\meas{M_{t}} = 0$.  
	
	Then $\meas{E_t} \geq \limsup_{s \to t} \meas{E_s}$.
\end{lemma}
\begin{proof}
	Let $t_k \in(a,b)$ be such that $t_k \to t$, and suppose that $s \in \bigcap_{k=1}^{\infty} \bigcup_{l=k}^{\infty} E_{t_l}$.  So for all $k \geq 1$ there exists an $l \geq k$ such that $s \in E_{t_{k_l}}$, which by definition implies that 
	\[
	\| v(s) - v(t_{k_l})\| > m | s - t_{k_l}|.
	\]
	Letting $k \to \infty$, continuity of $v$ implies that then
	\[
	\| v (s) - v(t) \| \geq m |s - t|,
	\]
	showing that $s \in E_t \cup M_t$.  Thus  $\bigcap_{k=1}^{\infty} \bigcup_{l=k}^{\infty} E_{t_l} \subseteq E_t \cup M_t$.  Therefore, since by assumption $\meas{M_t}=0$, we have that 
	\begin{align*}
	\lim_{k \to \infty} \meas{E_{t_{k}}} \leq \lim_{k \to \infty} \meas{ \bigcup_{l=k}^{\infty} E_{t_l}} = \meas{ \bigcap_{k=1}^{\infty} \bigcup_{l=k}^{\infty} E_{t_l}} \leq \meas{E_t \cup M_t} 
	&\leq \meas{E_t} + \meas{M_t} \\
	&= \meas{E_t},
	\end{align*}
	as required.
\end{proof}
\begin{proof}[Proof of theorem~\ref{thm:JB}]
	Choose $t_0 \in U$ such that $v'(t_0)$ exists and $\|v'(t_0)\| < \infty$.  Then there exists $\rho > 0$ such that $|t_0 - t | \leq \rho$ implies that $\|v(t) - v(t_0) \| \leq ( \| v'(t_0)\| + 1) |t - t_0|$.
	For $t \in [a,b]$  such that $|t - t_0| \geq \rho$, we have $\| v(t) - v(t_0)\| \leq 2 \sup_{s \in [a,b]} \|v(s)\| \leq 2 \sup_{s \in [a,b]} \|v(s)\| \rho^{-1}|t - t_0|$.  
	So, for all $m \geq \max \{ \|v'(t_0)\| + 1, 2 \sup_{s \in [a,b]} \|v(s)\| \rho^{-1}\}$, we have that $\|v(t) - v(t_0)\| \leq m |t - t_0|$ for all $t \in [a,b]$.  Choose such an $m$, moreover such that $\meas{\{s \in [a,b] : \| v(s) - v(t_0) \| = m |s-t_0|\}} = 0$; this is possible since this condition fails for at most countably many values of $m$.  Then lemma~\ref{levelset} implies that 
	\begin{equation}
	\label{meascont}
	0 \leq \limsup_{t \to t_0} \meas{E_t} \leq \meas{E_{t_0}} = \meas{\emptyset} = 0.
	\end{equation}
	
	By continuity of $L$ we can choose $\delta \in (0,1) $ such that $|L(t, y, p) - L(s, z, q)| \leq \epsilon/2(b-a)$ whenever $\max\{ |t|, \|y\|, \|p\|\} \leq |a| + |b| + \|v\|_{\infty} + m$ and $\max \{ |s-t|, \|y -z\|, \|p - q\|\} \leq \delta$.  Choose $\tau \in (0, \mathrm{dist}(t_0, [a,b] \setminus U)/2)$ such that
	\begin{enumerate}[label=(\ref{thm:JB}.\arabic*)]
		\item $\tau \leq \delta/2m$; \label{tau1}
		\item $\tau \leq \epsilon / (4 \sup_{t \in [a,b], \|p\|\leq m} |L(t, v(t), p)|)$; \label{tau2}
		\item $ \int_{E} \|v'(t)\|\, dt \leq \delta/2$ whenever $\meas{E} \leq \tau$; and\label{tau3}
		\item $\int_E |L(t, v(t), v'(t))|\, dt \leq \epsilon/4$ whenever $\meas{E} \leq \tau$. \label{tau4}
	\end{enumerate}
	By~\eqref{meascont} we can choose $\eta \in (0, \tau)$ such that $| t - t_0| \leq \eta$ implies that $0 \leq \meas{E_t} \leq \tau$.
	
	Now, if $\|v'(t)\| \leq m$ for almost every $t \in (t_0 - \eta, t_0 + \eta)$, then we can construct a trivial variation in the usual way, by taking some non-zero $\psi \in C^{\infty}((a,b); \mathbb{R}^n)$ with $\mathrm{spt} \psi \subseteq (t_0 - \eta, t_0+ \eta)$, and considering the sequence of functions $(v + j^{-1} \psi)$ as $j \to \infty$.
	
	So suppose otherwise, i.e.\ that there exists $s_0 \in (t_0 - \eta, t_0 + \eta)$ such that $v'(s_0)$ exists and $\|v'(s_0)\| > m$.  Then $s_0$ is an endpoint of some connected  component $(s_0, s_1)$ of the set $E_{s_0}$, by choice of $s_0$.  Notice, since $| s_0 - t_0| < \eta$, by the choice of $\eta$ we have that $0 < (s_1 - s_0) \leq \meas{E_{s_0}}  \leq \tau$. Since $\eta < \tau$, we see that 
	\[
	|s_1 - t_0| \leq |s_1 - s_0| + | s_0 - t_0| \leq \tau + \eta < 2 \tau  < \mathrm{dist}(t_0, [a,b] \setminus U).
	\]
	So $s_1 \in U \subseteq (a,b)$, and we must have that $\|v(s_0) - v(s_1)\| = m |s_0 - s_1|$, since the only other way in which $s_1$ could be an endpoint of a component of $E_{s_0}$ would be for it to be an endpoint of $[a,b]$, which we have now excluded.
	
	So we can define $u \in \scra{v(a)}{v(b)}$ by
	\[
	u(t) \df 
	\begin{cases}
	v(t) & t \notin (s_0, s_1), \\
	\textrm{affine} & t \in (s_0, s_1);
	\end{cases}
	\]
	so $u \neq v$, but $u = v$ off the set $(s_0, s_1) \subseteq U$, where $0 \leq s_1-s_0 \leq \tau$.  Moreover,  on $(s_0 , s_1)$ we have that $\|u'\| = m$, and by~\ref{tau3} and~\ref{tau1} that 
	\begin{align*}
	\|v(t) - u(t)\| \leq \|v(t) - v(s_0)\| + \|u(s_0) - u(t)\|
	& \leq \int_{s_0}^{t} \|v'(s)\|\, ds + m (t - s_0) \\
	& \leq \int_{s_0}^{s_1} \|v'(s)\| \, ds + m|s_1 - s_0| \\
	& \leq \delta / 2 + \delta /2.
	\end{align*}
	So $\|v (t)- u(t)\| \leq \delta$ for all $t \in [a,b]$.  So by the choice of $\delta$ as witnessing the continuity of $L$,~\ref{tau4}, and ~\ref{tau2}, we have that
	\begin{align*}
	\left| \mathscr{L}(u) - \mathscr{L}(v) \right| 
	& \leq \int_{s_0}^{s_1} |L(t, u(t), u') - L(t, v(t), u')| \, dt \\
	& \phantom{=} {}+ \int_{s_0}^{s_1} |L(t, v(t), u') - L(t, v (t), v'(t))| \, dt \\
	& \leq \int_{s_0}^{s_1} \epsilon / 2(b-a) \, dt + \int_{s_0}^{s_1} |L (t, v(t), u')|\, dt + \int_{s_0}^{s_1} |L(t, v(t), v'(t))| \, dt \\
	& \leq \epsilon / 2 + |(s_1 - s_0)| \left(\sup_{t \in [a,b], \|p \| \leq m} |L(t, v(t), p)|\right) + \epsilon / 4 \\
	& \leq 3 \epsilon / 4 + \tau \left(\sup_{t \in [a,b], \|p \|\leq m}|L(t, v(t), p)|\right) \\
	& \leq \epsilon,
	\end{align*}
	as required. 
\end{proof}
\begin{Question}
	Let $v \in W^{1,1}(a,b)$.  Does there exist a sequence of non-zero $u_j \in W_0^{1,\infty}(a,b)$ such that $\mathscr{L}(v + u_j ) \to \mathscr{L}(v)$?
	\label{Qu:var}
\end{Question}
Ball and Mizel~\cite{Ball-Mizel-1985} gave examples exhibiting the Lavrentiev phenomenon for which they made the incidental observation that $\mathscr{L}(v + t u) = \infty$ for all $t \neq 0$, for a large class of $u \in C_0^{\infty}(a,b)$, viz those $u$ which are non-zero at a certain  point in the domain (at which the minimizer $v$ is singular).  The Lagrangians are polynomial, superlinear, and satisfy $L_{pp} \geq \epsilon > 0$.  This would seem to suggest that the same could happen for all $u \in C_0^{\infty}(a,b)$ if a minimizer was singular on a dense set.  Indeed this is the case, as we shall shortly show, so the answer to our question, even if $v$ is a minimizer, is ``no''.  The construction is straightforward if we do not concern ourselves with superlinearity and strict convexity; we have to try rather harder to get $L_{pp} > 0$, since in this case partial regularity statements follow given only the mildest assumptions on the modulus of continuity of the Lagrangian~\cite{Clarke-Vinter-1985-regularity,Sychev-1993,Csornyei-etal-2008,Ferriero-2012}.

The following example is not at all difficult but I am not aware of it being presented elsewhere.
\begin{theorem}
	\label{thm:no var}
	There exists $v \in W^{1,1}(0,1)$ and a continuous Lagrangian $L \colon [0,1] \times \mathbb{R} \times \mathbb{R} \to [0, \infty)$, convex in $p$, such that $v$ is a minimizer of $\mathscr{L}$ over $\scra{v(0)}{v(1)}$, $0 \leq \mathscr{L}(v) < \infty$, but $\mathscr{L}(v + u) = \infty$ for all non-zero $u \in W_0^{1,\infty}(0,1)$.
\end{theorem}
\begin{proof}
	Let $\{x_n\}_{n=0}^{\infty}$ be an enumeration of $\mathbb{Q} \cap [0,1]$.  For $n,k \geq 0$ define $U_{n, k} \df (x_n - 2^{-n - 3k}, x_n + 2^{-n - 3k}) \cap [0,1]$.  For $n \geq 0$ define a non-negative function $\rho_n \in L^1(0,1)$ by
	\[
	\rho_n(t) \df \sum_{k=0}^{\infty}  2^k \mathds{1}_{U_{n, k}}(t).
	\]
	So 
	\[
	\int_0^1 \rho_n(t) \, dt \leq \sum_{k = 0}^{\infty} 2^{k+1} 2^{-3k-n} \leq 2^{-n +2}, 
	\]
	and we can define $\rho \in L^1(0, 1)$ by $\rho \df \sum_{n=0}^{\infty} \rho_n$ and $v \in W^{1,1}(0,1)$ by 
	\[
	v(t) \df \int_0^t (1+ \rho(s)) \, ds.
	\]
	So for all $n, k \geq 0$, for almost every $t \in U_{n, k}$ we have that
	\begin{equation}
	\label{v'big}
	v'(t) = 1 + \rho(t) \geq \rho_n (t) \geq 2^k.
	\end{equation}
	Let $L \colon [0,1] \times \mathbb{R} \times \mathbb{R} \to [0, \infty)$ be given by
	\[
	L(t, y, p) \df (y - v(t))^2 p^8.
	\]
	Then $L$ is continuous, and convex in $p$, and $v$ is clearly a minimizer of $\mathscr{L}$ over $\scra{v(0)}{v(1)}$.
	
	For any non-zero $u \in W_0^{1, \infty}(0,1)$, there exist $n \geq 0$, $K\geq0$, and $\epsilon > 0$ such that $|u| \geq \epsilon$ on $U_{n, k}$ for all $k \geq K$.  Without loss of generality, we may assume that $2^K \geq 2 \|u'\|_{L^{\infty}(0,1)}$. Let $k \geq K$. So we have that $2^k - \|u'\|_{L^{\infty}(0,1)} \geq 2^k - 2^{k-1} = 2^{k-1}$, which, with~\eqref{v'big}, implies that 
	\[
	\int_{U_{n, k}}(v'(t) + u'(t))\, dt \geq \int_{U_{n, k}} \left(2^k - \|u '\|_{L^{\infty}(0,1)}\right) \, dt  \geq 2^{k-1} \meas{U_{n, k}}.
	\]
	So we have, using Jensen's inequality, that
	\begin{align*}
	\mathscr{L}(v+u) & \geq  \int_{U_{n, k}} (u(t))^2(v'(t) + u'(t))^8\, dt \\
	& \geq \epsilon^2 \meas{U_{n, k}} \left(  \frac{1}{\meas{U_{n, k}}}\int_{U_{n, k}}(v'(t) + u'(t)) \, dt\right)^{\! 8}\\
	& \geq \epsilon^2 2^{-n - 3k+1} \cdot 2^{8(k-1)} \\
	& = \epsilon^2 2^{5k - n-7}.
	\end{align*}
	Since this holds for all $k \geq K$, we have that $\mathscr{L}(v + u) = \infty$, as required.
\end{proof}
This example will also serve to demonstrate that the general approximation provided in theorem~\ref{thm:JB} is not of the form $v + u$ for some $u \in W_0^{1,1}(a,b)$ for which $\mathscr{L}(v + \gamma u)$ is finite for the range of values $\gamma \in (0,1]$.  So question~\ref{Qu:JB}, while admitting a positive answer, cannot in general be answered in the manner provided by theorem~\ref{thm:JB} by passing far enough down a sequence of the form $v + j^{-1}u$ for some $u \in W_0^{1,1}(a,b)$. 

Let $L$ and $v$ be as constructed in theorem~\ref{thm:no var}, $U \subseteq (0,1)$ be open and non-empty, and let $u \in \scra{v(0)}{v(1)}$ be as constructed in theorem~\ref{thm:JB} for some $\epsilon > 0$.  Then the key point of the construction of $u$ is that there exists a subinterval $(s_0, s_1)$ of $U$ and a fixed gradient $m$, say, such that $u(t) \neq v(t)$ implies that $t \in (s_0, s_1)$ and $u'(t) = m$, for almost every $t \in (0,1)$.  Define $w \in W_0^{1,1}(0,1)$ by $w (t) \df u(t) - v(t)$, so $v'(t) + w'(t) = m$ for almost every $t \in \{ s \in (0,1) : w(s) \neq 0\} \neq \emptyset$.  Let $\gamma \in (0,1)$.

Then there exist $n, K \geq 0$ and $\delta > 0$ such that $| \gamma w(t)| \geq \delta$ for $t \in U_{n, k} \subseteq (s_0, s_1)$ for all $k \geq K$, where $U_{n, k}$ are as in theorem~\ref{thm:no var}.  Without loss of generality we may choose $K \geq 0$ such that $2^{K} \geq 2|m|/ ( 1-\gamma)$.  Let $k \geq K$.  So $m \leq ( 1- \gamma) 2^{k-1}$ and so
\[
(1- \gamma)2^k - m \geq (1 - \gamma)(2^ k - 2^{k-1}) = (1 - \gamma)2^{k-1}.
\]  
Notice that $v'(t) + \gamma w'(t) = v'(t) + \gamma (m - v'(t)) = ( 1- \gamma) v'(t) + \gamma m$ for almost every $t \in U_{n, k}$, so, by~\eqref{v'big}, we have that
\begin{align*}
\int_{U_{n, k}} (v'(t) + \gamma w'(t)) \, dt  = \int_{U_{n, k}}((1- \gamma) v'(t) + \gamma m ) \, dt 
& \geq \int_{U_{n, k}}( (1 - \gamma) 2^k - \gamma |m|) \, dt \\
& \geq ((1-\gamma) 2^k - |m| )\meas{U_{n, k}}\\
& \geq (1- \gamma)2^{k-1}\meas{U_{n, k}},
\end{align*}
since $\gamma \leq 1$.  Hence we have, using Jensen's inequality, that
\begin{align*}
\mathscr{L}(v + \gamma w) 
& \geq \int_{U_{n, k}} (\gamma w(t))^2 (v'(t) + \gamma w'(t))^8\, dt \\
& \geq \delta^2  \meas{U_{n, k}} \left( \frac{1}{\meas{U_{n, k}}}\int_{U_{n, k}}(v'(t) + \gamma w'(t))\, dt\right)^{\! 8}\\
& \geq \delta^2 2^{-n - 3k+1} \cdot ( 1- \gamma)^8 2^{8(k-1)}\\
& =\delta^2( 1- \gamma)^{8} 2^{5k -n   -7}.
\end{align*}
Since this holds for all $k \geq K$, we have that $\mathscr{L}(v + \gamma w) = \infty$, for all $\gamma \in (0,1)$.

The Lagrangian we have constructed in theorem~\ref{thm:no var}, however, vanishes along the minimizer, and so is not superlinear.   Gratwick and Preiss~\cite{Gratwick-Preiss-2011} show that it is possible to have a continuous, superlinear Lagrangian with $L_{pp} \geq 2 > 0$ for which the minimizer is nowhere locally differentiable.  That minimizer is, however, Lipschitz.  The example of section~\ref{sec:tonelli-cex} is a non-Lipschitz version of this construction, which gives a minimizer which has upper and lower Dini derivatives of $\pm \infty$ at every point of a dense set.
\begin{theorem}
	\label{thm:no var coercive}
	There exist $T> 0$, $w \in W^{1,2}(-T, T)$, and a continuous Lagrangian $L \colon [-T, T] \times \mathbb{R} \times \mathbb{R} \to [0, \infty)$, superlinear and satisfying $L_{pp} \geq 2> 0$, such that $w$ is a minimizer of $\mathscr{L}$ over $\scra{w(-T)}{w(T)}$, $0 \leq \mathscr{L}(w) < \infty$, but $\mathscr{L}(w + u) = \infty$ for all non-zero  $u \in W_0^{1,\infty}(-T, T)$.
\end{theorem}
\begin{proof}
	We let $T$, $w$, and $\phi$ be as from theorem~\ref{thm:thesis}, the notation of which we retain.  We have to add another term to that Lagrangian.
	For $k \geq 0$ we choose a decreasing sequence of numbers $t_k \in (0, T)$ such that
	\begin{equation*}
	\frac{\tilde{w}(t_k)}{t_k} \geq 2k + 1
	,
	\end{equation*}
	and recall that~\eqref{diffquot} implies that 
	\begin{equation}
	\label{wsteep}
	\frac{w(x_n + t_k) - w(x_n)}{t_k} \geq \frac{\tilde{w}(t_k)}{t_k} - 1 \geq 2k,
	\end{equation}
	for all $n \geq 0$ and $k \geq 0$ large enough, depending on $n$.  Define a convex and superlinear function $\omega \colon \mathbb{R} \to [0, \infty)$ as follows.  Set $\omega(0) \df 0$ and $\omega(1) \df t_1^{-1}$.  Suppose $\omega(l)$ has been defined such that $\omega(l) \geq 2 \omega(l-1)$ for each $1 \leq l \leq k-1$ for $k \geq 2$. Define
	\[
	\omega(k) \df \max\{2 \omega(k-1), k t_k^{-1}\}.
	\]
	This defines $\omega(k)$ for all $k \geq 0$, and we then define $\omega$  to be affine  between the specified endpoints on each interval $[k, k+1]$.  Define $\omega(p) \df \omega(-p)$ for $p \in (-\infty, 0)$.  
	Define $L \colon [-T, T] \times \mathbb{R} \times \mathbb{R} \to [0, \infty)$ by
	\[
	L(t, y, p) \df \phi(t, y - w(t)) + p^2 +  ( y - w(t))^2 \omega(p),
	\]
	which is continuous, superlinear, with $L_{pp} \geq 2 >0$, for which by theorem~\ref{thm:thesis}, $w$ is a minimizer of $\mathscr{L}$ over $\scra{w(-T)}{w(T)}$.  
	
	Let $u \in W_0^{1, \infty}(-T, T)$ be non-zero.  So there exist $n \geq 0$ and $\epsilon > 0$ such that $|u| \geq \epsilon$ on a neighbourhood of $x_n$.  Choosing $K \geq \|u'\|_{L^{\infty}(-T, T)}$, we have for all large $k \geq K$, by~\eqref{wsteep}, that
	\begin{align*}
	\int_{x_n}^{x_n + t_{k}} ( w'(t) + u'(t)) \, dt 
	& \geq  w(x_n + t_k) - w(x_n) - \|u'\|_{L^{\infty}(-T, T)}t_k \\
	& \geq ( 2k - \|u'\|_{L^{\infty}(-T, T)})t_k \\
	& \geq k t_k.
	\end{align*}
	So, using Jensen's inequality and that $\omega$ is non-decreasing on $[0, \infty)$, we have that
	\begin{align*}
	\int_{x_n}^{x_n + t_k} L(t, w+  u, w' + u') \, dt
	& \geq \int_{x_n}^{x_n + t_k} (u(t))^2 \omega (w'(t) + u'(t))\, dt  \\
	& \geq \epsilon^2 t_k\, \omega \! \left( (t_k)^{-1} \int_{x_n}^{x_n + t_k}( w'(t) + u'(t))\, dt\right)\\
	& \geq \epsilon^2 t_k \omega(k) \\
	& \geq \epsilon^2 k,
	\end{align*} 
	where the final inequality follows by the definition of $\omega$.  Since this holds for all large $k \geq K$, $\mathscr{L}( w+ u) = \infty$, as required.
\end{proof}
We might now  speculate whether in any one given problem, it is always possible to approximate the minimum value {\it either} by Lipschitz trajectories {\it or} by adding Lipschitz trajectories to the minimizer.  We know that neither approach alone succeeds in general, but is it possible that both fail simultaneously?
\begin{Question}
	\label{Qu:Lav or var}
	Let $v \in W^{1,1}(a,b)$ be a minimizer of $\mathscr{L}$ over $\scra{v(a)}{v(b)}$.  Suppose that the Lavrentiev phenomenon occurs.  Does there exist a sequence of non-zero $u_j \in W_0^{1, \infty}(a,b)$ such that $\mathscr{L}(v + u_j) \to \mathscr{L}(v)$?
\end{Question}
There seems to be very little reason to think this might be true: inferring a positive approximation result from the presence of the Lavrentiev phenomenon seems eccentric.  (In contrast, the principle of gaining information about the minimizer assuming \textit{non-}occurrence of the Lavrentiev phenomenon is used, for example, by Esposito, Leonetti, and Mingione~\cite{Esposito-Leonetti-Mingione-2004}.)  To show it to be false, we just need to show that the Lavrentiev phenomenon occurs in (a modified version of) one of the examples from theorems~\ref{thm:no var} or~\ref{thm:no var coercive}.
\begin{theorem}
	\label{thm:lav no var}
	There exists $v \in W^{1,1}(0,1)$ and a continuous Lagrangian $L \colon [0,1] \times \mathbb{R} \times \mathbb{R} \to [0, \infty)$, convex in $p$, such that $v$ is a minimizer of $\mathscr{L}$ over $\scra{v(0)}{v(1)}$, $0 \leq \mathscr{L}(v) < \infty$, but $\mathscr{L}(v + u) = \infty$ for all non-zero $u \in W_0^{1,\infty}(0,1)$.  Moreover, the Lavrentiev phenomenon occurs.
\end{theorem}
\begin{proof}
	We show that the example from theorem~\ref{thm:no var} exhibits the Lavrentiev phenomenon.  The argument is based on the presentation given in~\cite{Buttazzo-Giaquinta-Hildebrandt} of the example given by Mani\`a~\cite{Mania-1934}.  We borrow our notation from the proof of theorem~\ref{thm:no var}. Without loss of generality $x_0 = 0$.
	
	Let $u \in W^{1, \infty}(0,1) \cap \scra{v(0)}{v(1)}$.  The definition of $v$ implies that for $t \in U_{0, k}$, $v(t) \geq t \rho_0 (t) \geq 2^k$ and so, since as $t \to 0$ we may choose $k \to \infty$, and $v(0) = 0$, we see that $v'(0) = \infty$.  Since $u(0) = v(0) = 0$ and $u$ is Lipschitz, we must have that $u < v/4$ on a right neighbourhood of $0$.  Since also $u(1) = v(1)$, by the intermediate value theorem, $\{ t \in (0,1) : u(t) = v(t) / 4\} \neq \emptyset$.  Define 
	\[
	\tau_1 \df \sup \{ t \in (0,1): u(t) = v(t) / 4\} < 1.
	\]
	Similarly $\{ t \in (\tau_1 , 1) : u(t) = v(t) / 2\} \neq \emptyset$, so we may define 
	\[
	\tau_2 \df \min \left\{ 2 \tau_1, \inf \{ t \in (\tau_1, 1) : u(t) = v(t) / 2\} \right\}.
	\]
	Choose $k \geq 0$ such that $2^{-3k - 3} \leq \tau_1 \leq 2^{-3k}$, so $\tau_1 \in U_{0,k}$.  Then for $t \in (\tau_1, \tau_2)$, we have by definition of $\tau_2$, monotonicity of $v$, and~\eqref{v'big} that
	\[
	v(t) - u(t) \geq v(t) - v(t)/2 = v(t) /2 \geq v(\tau_1)/2 \geq 2^{k - 1}\tau_1 \geq 2^{k-1} \cdot 2^{-3k - 3} = 2^{-2k - 4},
	\]
	so $(v(t) - u(t))^2 \geq 2^{-4k - 8}$.
	
	If $\tau_2 = 2 \tau_1$, then $\tau_1, \tau_2 \in U_{0,k-1}$ if $k \geq 1$, and so $v' \geq 2^{k-1}$ almost everywhere on $(\tau_1, \tau_2)$.  If $k = 0$, then all we may say is that $\tau_1, \tau_2 \in U_{0, 0} = [0,1)$, on which $v' \geq 2$ by definition of $v$, since $\rho \geq 1$ everywhere.  So in general we may say that $v' \geq 2^{k-1}$ almost everywhere on $(\tau_1, \tau_2)$.  Hence, by definition of $\tau_1$, 
	\begin{align*}
	u(\tau_2) - u(\tau_1) \geq v(\tau_2)/4 - v( \tau_1 ) / 4= 2^{-2}(v(2 \tau_1) - v( \tau_1)) 
	& \geq 2^{-2} \cdot 2^{k-1} \tau_1 \\
	& \geq 2^{k -3} \cdot 2^{-3k -3} \\
	& = 2^{-2k - 6}.
	\end{align*}
	
	Otherwise, we have by the definitions of $\tau_1$ and $\tau_2$, the monotonicity of $v$, and~\eqref{v'big} that
	\[
	u(\tau_2) - u(\tau_1) = v(\tau_2) / 2- v( \tau_1) / 4 \geq v(\tau_1) / 4\geq 2^{-2} \cdot 2^k \tau_1 \geq 2^{-2 + k} \cdot 2^{-3k - 3} = 2^{-2k -5}.
	\]
	
	Hence in both cases we have that $(u(\tau_2) - u(\tau_1))^8 \geq 2^{-16k - 48}$.  So, using Jensen's inequality, we see, since $(\tau_2 - \tau_1)^{-1} \geq \tau_1^{-1} \geq 2^{3k}$ by definition of $\tau_2$, that 
	\begin{align*}
	\mathscr{L}(u) \geq \int_{\tau_1}^{\tau_2} ( u(t) - v(t))^2 (u'(t))^8 \, dt 
	& \geq 2^{-4k - 8} (\tau_ 2 - \tau_1 ) \left( (\tau_2 - \tau_1)^{-1}\int_{\tau_1}^{\tau_2} u'(t)\, dt \right)^{\! 8}\\
	& \geq 2^{-4k - 8} (\tau_2 - \tau_1)^{-7} 2^{-16k - 48} \\
	& \geq 2^{-20k - 56} \cdot 2^{21k}\\
	& \geq 2^{-56}.
	\end{align*}
	Since this number is independent of $k$ and therefore of $u$, we see that the Lavrentiev phenomenon occurs, as claimed.
\end{proof}
The Lagrangian in theorem~\ref{thm:lav no var} can be adapted to have $L_{pp} \geq \epsilon > 0$ and be superlinear, while still exhibiting the Lavrentiev phenomenon.  Notice that 
\[
\int_0^1{\rho_n^2} \leq \sum_{k=0}^{\infty} 2^{2k} \meas{U_{n, k}} \leq \sum_{k=0}^{\infty} 2^{2k} \cdot 2^{-3k - n+1} = 2^{-n+2},
\]
and therefore that $\sum_{n = 0}^{\infty} \rho_n$ converges in $L^2(0, 1)$, thus $v \in W^{1, 2}(0, 1)$. Following~\cite{Buttazzo-Giaquinta-Hildebrandt} (see p.~148), we set
\[
\tilde{L}(t, y, p) = (y - v(t))^2 p^8 + \epsilon p^2,
\]
for some $0 < \epsilon < 2^{-56} \|v'\|_{L^2(0,1)}^{-2}$.  Then $\tilde{L}_{pp} \geq 2 \epsilon > 0$ and $\tilde{L}$ is superlinear, and moreover, letting $\tilde{\mathscr{L}}$ denote the corresponding functional, 
\[
\inf_{w \in \scra{v(0)}{v(1)}} \tilde{\mathscr{L}}(w) \leq \tilde{\mathscr{L}}(v) = \epsilon \|v'\|_{L^2(0, 1)}^2 < 2^{-56} \leq \tilde{\mathscr{L}}(u),
\]
for all $u \in W^{1, \infty}(0, 1) \cap \scra{v(0)}{v(1)}$, so the Lavrentiev phenomenon persists.  However, we lose the easy observation that $v$ is a minimizer, and the  result about Lipschitz variations is no longer clear.  

The example of theorem~\ref{thm:no var coercive} sadly rather readily fails to exhibit the Lavrentiev phenomenon: consider following the near-minimizer $w_n$ everywhere except on small intervals around its singularities, on which one just remains constant until one can pick up the minimizer on the other side of the singularity (this argument is made precise by Gratwick~\cite{Gratwick-2011}).  However, it can be modified, as suggested by the standard computations involved in Mani\`a's example, into an example which does exhibit the Lavrentiev phenomenon, by adding a non-decreasing trajectory with a vertical tangent at $0$.
\begin{theorem}
	\label{thm:lav no var coercive}
	There exist $T> 0$, $w \in W^{1,2}(0,T)$, and a continuous Lagrangian $L \colon [0,T] \times \mathbb{R} \times \mathbb{R} \to [0, \infty)$, superlinear in $p$ and with $L_{pp} \geq 2 > 0$, such that $w$ is a minimizer of $\mathscr{L}$ over $\scra{w(0)}{w(T)}$, $0 \leq \mathscr{L}(w) < \infty$, but $\mathscr{L}(w + u) = \infty$ for all non-zero $u \in W_0^{1,\infty}(0,T)$.  Moreover, the Lavrentiev phenomenon occurs.
\end{theorem}
\begin{proof}
	We adapt the example from theorem~\ref{thm:no var coercive}, and borrow the notation from that proof.  We consider only the interval $[0, T]$, and observe that the function $w$ from this example is a minimizer of $\mathscr{L}$ on $[0,T]$ over $\scra{w(0)}{w(T)}$, and note that $w(0) = w(x_0) = w_0(x_0) = \tilde{w}(0) = 0$ by~\ref{wfixxi} and the definition of $\tilde{w}$.
	
	From~\ref{lipw} and~\ref{wfixxi} we know that 
	\begin{equation}
	\label{wfrom0}
	|w(t)| = |w(t) - w(0)| \leq 2 g(t),
	\end{equation}
	for all $t \in [0, T]$.  So $w(t) + 3 g(t) \geq g(t)$.  This ``$3g$-centred'' version of $w$ will be our new minimizer with respect to its own boundary conditions $w(0) + 3g(0) = 0$ and $w(T) + 3g(T)$.  We modify our Lagrangian from theorem~\ref{thm:no var coercive} to construct a problem which this new function minimizes; to do this we need to add a new weight function containing a term in $g''$.  
	
	Let $\Phi \colon [0, T] \times \mathbb{R} \to [0, \infty)$ be given by
	\[
	\Phi(t, y) \df 
	\begin{cases}
	0 & t= 0, \\
	7|g''(t)| |y| & t \neq 0,\ |y| \leq 6 g(t), \\
	42|g''(t)| g(t) & t \neq 0,\ |y| > 6 g(t).
	\end{cases}
	\]
	Now,
	\[
	g''(t) = \frac{-1}{t \one 1/t} \left(\frac{1}{\one 1/t} + 1\right),
	\]
	so recalling that $T> 0$ was chosen small enough that $\one 1/t \geq 1$ on $(0,T)$, we see that
	\begin{equation}
	\label{gg''}
	|g''(t)g(t)| \leq \frac{2 t \two 1/t}{t \one 1/t} = \frac{2 \two 1/t}{\one 1/t} \to 0\ \textrm{as}\ t \to 0,
	\end{equation}
	so $\Phi$ is continuous.  Now define $F \colon [0, T] \times \mathbb{R} \times \mathbb{R} \to [0, \infty)$ by 
	\[
	F(t, y, p) \df \phi(t, y - (w + 3g)) + \Phi (t, y - (w + 3g)) + p^2,
	\]
	and consider the corresponding functional $\mathscr{F} (u) \df \int_0^T F(t, u, u')$ defined for $u \in W^{1, 1}(0, T)$. 
	We claim that $w + 3g$ is a minimizer over $\scra{0}{(w + 3g)(T)}$ of  $\mathscr{F}$.  Let $u \in \scra{0}{w(T)}$ be such that $u + 3g$ is a minimizer of $\mathscr{F}$ over $\scra{0}{(w + 3g)(T)}$.  
	
	First we claim that $|u(t)| \leq 4g(t)$ on $[0,T]$. This is the same strategy of proof as found in lemma~\ref{lemma1}, so we give no more than a sketch of the argument. 	Suppose for a contradiction that $u(t) > 4g(t)$ on some interval $I$.  Then since $|w(t) + 3g(t)| \leq 5g(t)$ by~\eqref{wfrom0}, we see that $u(t) + 3g(t) >7g(t) \geq w(t) + 3g(t)$ on $I$, where $7g $ is a concave function, and in particular that  $I \subsetneq (0,T)$. Therefore we can find an affine function $l$ such that $u(t) + 3g(t) > l (t)\geq w(t) + 3g(t)$ on some subinterval of $I$.  Defining a new trajectory $u_l \in \scra{0}{(w + 3g)(T)}$  by replacing $u + 3g$ with $l$ on this subinterval of $I$, we see that $u_l$ does not increase the ``weight terms'' $\phi$ and $\Phi$ in $F$, and strictly decreases the gradient term, since affine functions are the unique minimizers of quadratic functionals, so
	\[
	\mathscr{F}(u_l) < \mathscr{F}( u + 3g),
	\]
	which contradicts the choice of $u$ as being such that $u + 3g$ is a minimizer.  Supposing in the other case for a contradiction that $u(t) < -4g(t)$ on some interval $I$, we see that $u(t) + 3g(t) < -g(t) \leq 0 \leq w(t) + 3g(t)$, where $-g$ is a convex function, and we use a similar argument to gain a contradiction.
	
	So indeed $|u(t)| \leq 4g(t)$ on $[0, T]$, hence $|u(t) - w(t)| \leq 6g(t)$ by~\eqref{wfrom0} and thus $\Phi(t, u -w) = 7|g''| |u - w|$ on $(0, T]$, by definition.
	
	We now claim that, extended to have value $0$ at $t = 0$, the function $g'(u - w)$ is absolutely continuous on $[0, T]$, i.e.\ can be written as an indefinite integral on $[0, T]$.   That this definition makes it continuous follows since
	\[
	|g'(t) ( u (t) - w(t))| \leq 6 |g'(t) g(t)| \leq 6 \left(\frac{1 }{\one 1/t} + \two 1/t\right)t \two 1/t  \to 0
	\]
	as $t \to 0$. Clearly $g'(u-w)$ is absolutely continuous on subintervals of $(0,T)$ bounded away from $0$, so by the dominated convergence theorem it suffices to show that $(g'(u-w))' \in L^1(0, T)$.  Now,
	\[
	|(g'(u-w))'| \leq |g''||u-w| + |g'|(|u' + 3g'| + |w' + 3g'|),
	\]
	and by the above we have that $|g''||u-w|\leq 6 |g''||g|$, which is bounded on $[0,T]$ by~\eqref{gg''}, so certainly integrable.  Using Cauchy-Schwartz, we see further that
	\[
	\int_0^T |g'|(| u' + 3g'| + |w' + 3g'|) \leq \left( \int_0^T |g'|^2 \right)^{\! 1/2} \left(\int_0^T ( |u' + 3g'| + |w' + 3g'|)^2\right)^{\! 1/2},
	\]
	which is finite since $g' \in L^2(0, T)$, $w \in W^{1,2}(0, T)$, and since $u + 3g$ is a minimizer of $\mathscr{F}$ by assumption.  So indeed $g'(u - w)$ is absolutely continuous on $[0,T]$.
	
	The minimality of $w$ established in theorem~\ref{thm:thesis} implies that
	\[
	\int_0^T \left( \phi ( t, u - w) + (u')^2 - (w')^2 \right) \geq 0.
	\] 
	So, recalling also the simple pointwise inequality~\eqref{triv}, and integrating $g'(u'-w')$ by parts, we see that
	\begin{align*}
	\mathscr{F}( u + 3g) - \mathscr{F}( w + 3g) 
	& = \int_0^T \left(\phi(t, u - w) + \Phi (t, u- w) + (u' + 3g')^2 - ( w' + 3g')^2 \right)\, dt\\
	& = \int_0^T \left(\phi(t, u-w) + (u')^2 - (w')^2\right)\, dt \\
	& \phantom{=} \quad {}+ \int_0^T \left(\Phi(t, u-w) + 6 g'(u' - w')\right) \, dt\\
	& \geq 0 + 6 [g'(u - w)]_{0}^T + \int_0^T \left(\Phi(t, u-w) - 6 g''(u -w)\right)\, dt \\
	& \geq \int_0^T \left(7|g''||u-w| - 6|g''| |u-w|\right) \, dt \\
	& \geq 0.
	\end{align*}
	So $w + 3g$ is indeed a minimizer of $\mathscr{F}$.
	
	Since $g'$ increases to $\infty$ as we approach $0$, we can find a sequence $r_k > 0$ such that $r_k \downarrow 0$ and $g' \geq k+ 1$ on $(0, r_k)$.  $T>0$ was chosen small enough that we may consistently set 	$r_0 = 2T$.  Define a convex and superlinear function $\Theta \colon \mathbb{R} \to [0, \infty)$ such that
	\begin{equation}
	\label{bigomega}
	\Theta(p) \geq 2^8\|w' + 3g'\|_{L^2(0, T)}^2 p r_{k}^{-3}\ \textrm{for}\ p \geq k/4,
	\end{equation}
	for all $k\geq 0$, as follows.  Set $\Theta(0) \df 0$ and $\Theta (1/4) \df 2^6\|w' + 3g'\|_{L^2(0, T)}^2 r_1^{-3}$.  Suppose that $\Theta (l/4)$ has been defined such that $\Theta (l / 4) \geq 2 \Theta ( (l-1)/4)$ for all $1 \leq l \leq k-1$, for some $k \geq 2$.  Define 
	\[
	\Theta(k/4) \df \max \{ 2 \Theta ((k-1)/4), 2^6\|w' + 3g'\|_{L^2(0, T)}^2 k r_{k}^{-3} \}.
	\]
	This defines $\Theta$ inductively at the points $k/4$, and we extend it to be affine on each interval $[k/4, (k+1)/4]$.  Define $\Theta(p) = \Theta(-p)$ for $p \in (-\infty, 0$).
	
	Define $L(t, y, p) \colon [0, T] \times \mathbb{R} \times \mathbb{R} \to [0, \infty)$ by 
	\begin{align*}
	L(t, y, p) 
	& \df \phi(t, y - (w + 3g)) + \Phi(t, y - (w + 3g)) + p^2 + (y - ( w+ 3g))^2 (\omega(p) + \Theta(p)) \\
	& = F(t, y, p) + ( y - ( w + 3g))^2 (\omega(p) + \Theta(p)).
	\end{align*}
	So $L$ is continuous, superlinear in $p$ and has $L_{pp}\geq 2> 0$, and, since $w + 3g$ is a minimizer of $\mathscr{F}$ over $\scra{0}{(w + 3g)(T)}$, clearly $w + 3g$ is  a minimizer of the associated functional $\mathscr{L}$ over $\scra{0}{(w + 3g)(T)}$.  
	
	By monotonicity of $g$ and~\eqref{wsteep} we have that
	\[
	\frac{(w + 3g)(x_n + t_k) - (w + 3g)(x_n)}{t_k} \geq \frac{w(x_n + t_k) - w(x_n)}{t_k} \geq 2k,
	\] 
	for all $n \geq 0$ and sufficiently large $k \geq 0$, depending on $n$.  Given this, the argument that $\mathscr{L}(w + 3g + u) = \infty$ for all $ u \in W_0^{1, \infty}(0, T)$ follows exactly as in theorem~\ref{thm:no var coercive}.  
	
	It just remains to show that the Lavrentiev phenomenon occurs.  The Mani\`a-style estimates follow exactly the same pattern as before.  Let $u \in W^{1, \infty}(0, T) \cap \scra{0}{(w + 3g)(T)}$.  Since $g'(0) = \infty$, and that $u(0) = w(0) = g(0) = 0$ and $u$ is Lipschitz, we have that $u < g/4$ on some right neighbourhood of $0$.  Since $u(T) = ( w + 3g)(T) \geq g(T)$, the intermediate value theorem implies that $\{ t \in (0, T) : u(t) = g(t) / 4\} \neq \emptyset$.  Define
	\[
	\tau_1 \df \sup \{ t \in (0, T) : u(t) = g(t) / 4 \} < T.
	\]
	Similarly $\{ t \in (\tau_1, T) : u(t) = g(t) / 2\} \neq \emptyset$, so we may define
	\[
	\tau_2 = \min \left\{ 2 \tau_1, \inf \{ t \in (\tau_1, T) : u(t) = g(t) / 2 \} \right\}.
	\]
	So $0 < \tau_2 - \tau_1 \leq \tau_1$.
	Choose $k \geq 0$ such that $r_{k+1} \leq 2 \tau_1 \leq r_k$.  By~\eqref{wfrom0}, the definition of $\tau_2$, the  monotonicity of $g$, and the choice of $r_k$, for all $t \in (\tau_1, \tau_2)$ we have that 
	\[
	(w+ 3g)(t) - u(t) \geq g(t) - g(t) / 2 \geq g( \tau_1)/2 \geq (k + 1) \tau_1 / 2.
	\]
	Thus $((w+3g)(t) - u(t))^2 \geq 2^{-2}(k + 1)^2\tau_1^2$ on $(\tau_1, \tau_2)$.
	
	Now, if $\tau_2 = 2 \tau_1$, then by the definition of $\tau_1$ and choice of $r_k$, we have that
	\[
	u(\tau_2) - u (\tau_1) \geq g( \tau_2) / 4 - g (\tau_1) / 4 = 2^{-2}( g ( 2 \tau_1) - g( \tau_1)) \geq 2^{-2}(k + 1) \tau_1.
	\]
	Otherwise, we have by the definitions of $\tau_1$ and $\tau_2$, the monotonicity of $g$,  and the choice of $r_k$ that
	\[
	u(\tau_2) - u(\tau_1) = g( \tau_2) / 2 - g (\tau_1)/4 \geq g (\tau_1)/ 2 -  g( \tau_1) / 4 = g (\tau_1) / 4 \geq 2^{-2}(k + 1) \tau_1.
	\]
	So in either case we have that 
	\[
	\frac{u ( \tau_2) - u (\tau_1)}{\tau_2 - \tau_1} \geq \frac{u ( \tau_2) - u (\tau_1)}{\tau_1} \geq 2^{-2} ( k + 1).
	\]
	Therefore by Jensen's inequality we have, by the choice of $\Theta$ as satisfying~\eqref{bigomega}, that
	\begin{align*}
	\int_{\tau_1}^{\tau_2}\, \Theta(u'(t))\, dt 
	& \geq ( \tau_2 - \tau_1)\, \Theta \! \left(\frac{1}{\tau_2 - \tau_1} \int_{\tau_1}^{\tau_2} u'(t) \, dt \right)\\
	& = ( \tau_2 - \tau_1) \, \Theta \! \left( \frac{u (\tau_2) - u ( \tau_1)}{\tau_2 - \tau_1}\right) \\
	& \geq ( \tau_2 - \tau_1) 2^8 \| w' + 3 g' \|_{L^2(0, T)}^2 \frac{ u ( \tau_2 ) - u( \tau_1)}{\tau_2 - \tau_1}r_{k+1}^{-3}\\
	& \geq 2^6 \| w' + 3 g' \|_{L^2(0, T)}^2 (k + 1) \tau_1 r_{k+1}^{-3}.
	\end{align*}
	So,  since $\tau_1 \geq r_{k+1}/2$, we have that
	\begin{align*}
	\mathscr{L}(u)
	&  \geq \int_{\tau_1}^{\tau_2} ( u (t) - ( w + 3g)(t))^2 \Theta ( u' (t)) \, dt \\
	& \geq 2^{-2} (k + 1)^2 \tau_1^2 \int_{\tau_1}^{\tau_2} \Theta ( u'(t)) \, dt \\
	& \geq 2^{-2} (k + 1)^2 \tau_1^2 2^6 \| w' + 3 g' \|_{L^2(0, T)}^2 (k + 1) \tau_1 r_{k+1}^{-3} \\
	& = 2^{4} (k + 1)^3 \tau_1^3\| w' + 3 g' \|_{L^2(0, T)}^2 r_{k+1}^{-3}\\
	& \geq 2 \|w' + 3g'\|_{L^2(0, T)}^2.
	\end{align*}
	That is, for all $u \in W^{1, \infty}(0, T) \cap \scra{0}{(w + 3g)(T)}$, we have that
	\[
	\mathscr{L}(u) \geq 2 \|w' + 3g'\|_{L^2(0, T)}^2 > \|w' + 3g'\|_{L^2(0, T)}^2 = \mathscr{L}(w + 3g),
	\]
	which is precisely to say that the Lavrentiev phenomenon occurs, as required.
\end{proof}
\section{Minimal regularity}
\label{sec:regularity}
We know that under our standing assumptions, minimizers need not be everywhere differentiable.  In this section, we deduce properties of the derivatives of minimizers at points at which they do exist.  We show that derivatives must be approximately continuous at points where they exist (theorems~\ref{thm:apcontfinite} and~\ref{thm:apcontinfinite}), and that ``kinks'' may not appear, i.e.\ if both one-sided derivatives exist at a point, they must be equal (corollaries~\ref{cor:nokinkfinite} and~\ref{thm:infinite one side}).

Our results in this section apply to vector-valued trajectories $v \colon [a,b] \to \mathbb{R}^n$.  Our proofs proceed by contradiction, assuming a minimizer $v$ has a derivative which fails to be well-behaved in a certain way, and thereby constructing a competitor trajectory with strictly lower energy, by replacing $v$ with affine pieces on open subintervals of the domain.  
\begin{definition}
	The left and right Dini derivatives, $D^{-}v(t)$ and $D^{+}v(t)$ respectively, of a function $v \in W^{1,1}((a,b);\mathbb{R})$ at a point $t \in [a,b]$ are given by
	\[
	D^{-}v(t) \df \lim_{s \uparrow t} \frac{ v(s) - v(t)}{s-t},\ \text{and}\ D^{+}v(t) \df \lim_{s \downarrow t}\frac{v(s) - v(t)}{s-t},
	\]
	whenever these limits make sense and exist as finite or infinite values.  The left and right derivatives of vector-valued functions $v \in W^{1,1}((a,b); \mathbb{R}^n)$ are formed by taking the vectors of the corresponding left and right derivatives of the components; thus these exist at a point if and only if the corresponding derivatives of each component function exist at that point.  In principle, then, such vectors of derivatives might contain components with infinite values.  We shall clearly distinguish the cases when all the components are finite, and when one or more may be infinite. 
\end{definition}

\begin{definition}[See for example~\cite{Evans-Gariepy}]
	We recall the usual definition of approximate continuity.  Let $f \colon [a,b] \to \mathbb{R}^n$ be measurable.  We say that $f$ is approximately continuous on the left at $t \in (a, b]$ if, for all $c > 0$,  
	\[
	\lim_{s \uparrow t} (t-s)^{-1} \lambda ( \{ r \in (s, t) : \| f(r) - f(t)\| \geq c \} )  = 0;
	\]
	similarly, we say that $f$ is approximately continuous on the right at $t \in [a, b)$ if, for all $c > 0$, 
	\[
	\lim_{s \downarrow t} (s-t)^{-1} \lambda( \{ r \in (t, s) :\|f(r) - f(t)\| \geq c \})  = 0.
	\]
\end{definition}
We retain our standing assumption of continuity of the Lagrangian.  Some further assumption of strict convexity is required to deduce any regularity results.  We impose the following condition on $L$: that for all $R \in [1, \infty)$, there exists $\tau_R > 0$ such that for all $(t, y, p)$ with $\max\{|t| , \|y \| , \|p\|\} \leq R + 1$, there exists a subdifferential $\xi \in \mathbb{R}^n$ of $L(t, y, \cdot)$ at $p$ such that 
\begin{equation}
\label{convexity-assump}
L(t, y, q ) \geq L(t, y, p) + \xi \cdot (q - p) + 2 \tau_R,
\end{equation}
whenever $\| q - p \| \geq R^{-1}$.  This holds in particular if the (partial) Hessian $L_{pp}$ exists and is continuous and strictly positive for all $(t, y, p)$.

The following lemma is our key tool, which we use repeatedly in the remainder of the section.
\begin{lemma}
	\label{key-lemma}
	Let $v \in W^{1,1}((a,b) ; \mathbb{R}^n)$, $R \geq |a| + |b| + \|v\|_{\infty} + 1$, and $\epsilon > 0$. 
	
	Then there exists $\delta > 0$ such that if $(t_1, t_2) \subseteq [a,b]$ satisfies $t_2 - t_1 \leq \delta$ and $\|v(t_2) - v(t_1)\| \leq R(t_2 - t_1)$, then $u \in W^{1,1}((a, b) ; \mathbb{R}^n)$ defined by 
	\[
	u(r) \df 
	\begin{cases}
	v(r) & r \notin (t_1, t_2), \\
	\textrm{affine} & \textrm{otherwise};
	\end{cases}
	\]
	satisfies
	\begin{align*}
	\lefteqn{\int_{t_1}^{t_2} L(s, v(s), v'(s))\, ds } \\
	& \geq \int_{t_1}^{t_2} L(s, u(s), u'(s))\, ds + \tau_R \lambda ( \{ s \in (t_1,t_2) : \| v'(s) - u'\| \geq R^{-1} \})  - \epsilon (t_2 - t_1).
	\end{align*}
\end{lemma}
\begin{proof}
	First we show that the strict convexity and continuity of $L$ conspire to allow us to use a subdifferential of $L(t, y, \cdot)$ in the convexity inequality involving the function $L(s, z, \cdot)$, when $(s, z)$ is near to $(t, y)$.  
	
	We choose $\delta_1 \in (0,1)$ witnessing the uniform continuity for $\min \{ \tau_R/2, \epsilon / 3\}$ of $L$ for $(t, y, p)$ such that $\max\{|t|, \|y\| , \|p\| \}\leq R + 1$.  Let $(t, y), (s, z) \in [a,b] \times \mathbb{R}^n$ be such that $\max\{ |t|, \|y\|\} \leq R$ and $\max\{| s - t| , \|y - z\|\} \leq \delta_1$, and let $p, q \in \mathbb{R}^n$ be such that $\|p\| \leq R$ and $\|q - p\| \geq R^{-1}$. Define $\tilde{q} \df p + R^{-1}\|q - p\|^{-1} (q-p)$, so $\| \tilde{q} - p\| = R^{-1}$ and $\|\tilde{q}\| \leq R + R^{-1}$.  Letting $\xi \in \mathbb{R}^n$ be a subdifferential of $L(t, y, \cdot)$ at $p$ which satisfies~\eqref{convexity-assump}, using continuity twice we see that
	\begin{align*}
	L(s, z, \tilde{q}) \geq L(t, y, \tilde{q}) - \tau_R/2 
	& \geq L(t, y, p) + \xi \cdot ( \tilde{q} - p) + 3\tau_R/2 \\
	&\geq L(s, z, p ) + \xi \cdot ( \tilde{q} - p) + \tau_R.
	\end{align*}
	Convexity of the one-dimensional function $\mu \mapsto L(s, z, p + \mu ( q-p))$  allows us to infer, since $R^{-1}\|q - p\|^{-1} \leq 1$,  that
	\begin{align*}
	L(s, z, q ) - L(s, z, p)  \geq \frac{ L(s, z, \tilde{q}) - L(s, z, p)}{R^{-1}\|q - p\|^{-1}} 
	& \geq R \|q - p\| ( \xi \cdot ( \tilde{q} - p) + \tau_R) \\
	& \geq  \xi \cdot (q - p) + \tau_R.
	\end{align*}
	That is,  $\delta_1 > 0$ is such that for all $(t, y, p)$ with $\max\{|t|,\|y \|, \|p\|\} \leq R$, there exists a subdifferential $\xi \in \mathbb{R}^n$ of $L(t, y, \cdot)$ at $p$ such that 
	\begin{equation}
	\label{convexity-conseq}
	L(s, z, q) \geq L(s, z, p) + \xi \cdot (q - p) + \tau_R,
	\end{equation}
	whenever $\max\{|s-t|,\|y - z\|\} \leq \delta_1$ and $\| q - p\| \geq R^{-1}$.
	
	Since $v$ is absolutely continuous, when seeking to apply this inequality along the graph of the trajectory of $v$ we can reduce the condition on proximity of the $y$ variable (i.e.\ the $v(t)$ term) to a condition only on proximity of the time variable.  Choose $\delta  \in (0, \delta_1 / 2R)$ such that $\int_E \|v'\| \leq \delta_1 / 2$ whenever $E \subseteq [a,b]$ satisfies $\lambda (E) \leq \delta$.  
	
	Now let $(t_1, t_2) \subseteq [a,b]$ be such that $0 < t_2 - t_1 \leq \delta$, and $\|v(t_2) - v(t_1)\| \leq R(t_2 - t_1)$, and let $s \in (t_1, t_2)$. Then for $u \in W^{1,1}((a,b); \mathbb{R}^n)$ as defined in the statement, $u'(s) =  l \df (t_2 - t_1)^{-1} (v(t_2) - v(t_1)) $, and so $\|l\| \leq R$ by assumption.  Then
	\begin{align*}
	\|u (s) - v(t_1)\| & \leq R( s - t_1) \leq  R\delta \leq \delta_1/2;\\
	\|v(s) - v(t_1)\| & \leq \int_{t_1}^s \|v'\| \leq \delta_1/2;\\
	\shortintertext{and so}
	\|u(s) - v(s)\| & \leq \|u(s) - v(t_1)\| + \|v(t_1) - v(s)\| \leq \delta_1.
	\end{align*}
	So $\max\{|s - t_1|, \|u(s)- v(t_1)\|,  \|v(s) - v(t_1)\|, \|u(s) - v(s)\|\} \leq \delta_1$.  Moreover, $\max \{ |t|, \|v\|_{\infty}\} \leq R$.  So there exists a subdifferential $\xi \in \mathbb{R}^n$ of $L(t_1, v(t_1), \cdot)$ at $l$ for which~\eqref{convexity-conseq} holds.
	Suppose first that  $s \in (t_1, t_2)$ is such that $v'(s)$ exists and $\| v'(s) - l\| \geq R^{-1}$.  Then using~\eqref{convexity-conseq} and the continuity of $L$, we have
	\begin{align*}
	L(s, v(s), v'(s)) 
	& \geq L(s, v(s), l) + \xi \cdot ( v'(s) - l) + \tau_R \\
	& \geq L(s, u(s), l) + \xi \cdot (v'(s) - l) + \tau_R - \epsilon/3.
	\end{align*}
	Suppose now that $s \in (t_1, t_2)$ is such that $v'(s)$ exists and $\|v '(s) - l\| < R^{-1}$.  Then $\|v'(s)\| \leq \|l\| + R^{-1} \leq R + 1$, and we may use continuity and a (non-strict) application of our convexity assumption~\eqref{convexity-assump} to see that
	\begin{align*}
	L(s, v(s), v'(s) )
	& \geq L(t_1, v(t_1), v'(s)) - \epsilon/3 \\
	& \geq L(t_1, v(t_1), l) + \xi \cdot (v'(s) - l) - \epsilon/3 \\
	& \geq L(s, u(s), l) + \xi \cdot (v'(s) - l) - 2 \epsilon / 3.
	\end{align*}
	Since almost every $s \in (t_1, t_2)$ falls into one of these two cases, we can now integrate and see that
	\begin{align*}
	\lefteqn{\int_{t_1}^{t_2} L(s, v(s), v'(s)) \, ds }\\
	& \geq \int_{\{ s \in (t_1, t_2) : \|v'(s) - l \| \geq R^{-1} \}} (L(s, u(s), l) + \xi \cdot ( v'(s) - l) + \tau_R - \epsilon/3) \, ds \\
	& \phantom{=} {}+ \int_{\{ s \in (t_1, t_2) : \| v'(s) - l \| < R^{-1}\}}( L(s, u(s) , l) + \xi \cdot ( v'(s) - l) - 2 \epsilon / 3) \, ds\\
	& \geq \int_{t_1}^{t_2} L(s, u(s), l) \, ds \\
	& \phantom{=} {}+ \int_{t_1}^{t_2} \xi \cdot (v'(s) - l)\, ds + \tau_R \lambda ( \{ s \in (t_1, t_2): \|v'(s) - l\| \geq R^{-1}\} ) - \epsilon(t_2 - t_1)\\
	& = \int_{t_1}^{t_2} L(s, u(s), u'(s)) \, ds + \tau_R \lambda ( \{ s \in (t_1, t_2): \|v'(s) - l\| \geq R^{-1}\} ) - \epsilon(t_2 - t_1),
	\end{align*}
	recalling that $l = (t_2 - t_1)^{-1} (v (t_2) - v(t_1))$, and therefore that $\int_{t_1}^{t_2} \xi \cdot (v'(s) - l) \, ds = 0$.
\end{proof}
Armed with this tool, we may swiftly deduce some facts about the behaviour of the derivatives of minimizers.  Assuming for a contradiction some bad behaviour of the derivative of a minimizer, each proof comes down to the ability to insert small affine segments into the trajectory, with slopes which differ significantly from the derivative of the minimizer.  The construction of these affine segments is slightly easier when the range is one-dimensional, but the general case is not particularly difficult, so we content ourselves with attacking immediately the vector-valued case.  The reader who believes that our proofs need not be quite as fussy in the one-dimensional case is quite right. 
\begin{theorem}
	\label{thm:apcontfinite} 
	Let $v \in W^{1,1}((a,b); \mathbb{R}^n)$ be a minimizer of $\mathscr{L}$ over $\scra{v(a)}{v(b)}$, and suppose for some $t \in [a,b]$ that, respectively,  $t \in (a, b]$ and $D^{-}v(t)$ exists and each component is finite; or $t \in [a,b)$ and  $D^{+}v(t)$ exists and each component is finite.
	
	Then $v'$ is approximately continuous on the left, respectively right, at $t$.
\end{theorem}
\begin{proof}
	We consider the case in which $t \in [a, b)$;  the other case is similar.
	
	By translating our domain $[a,b]$, subtracting an affine function from $v$, and making the corresponding corrections to $L$, without loss of generality we may assume that $t= 0 \in [a, b)$, $v(t) = 0$, and $D^{+}v (t) = 0$.  
	
	Suppose for a contradiction that the result is false, so there exist $c, \alpha \in (0,1)$ and arbitrarily small $s > 0$ such that 
	\begin{equation}
	\label{contradiction-assump-1}
	\lambda (\{r \in (0, s): \|v'(r)\| \geq c\}) > \alpha s.
	\end{equation}
	Let $\delta \in (0, b)$ be as given by lemma~\ref{key-lemma} for $R \geq 2c^{-1}$ and $\epsilon  \leq \alpha\tau_R / 8$.  Let $s_0 \in (0, \delta/2 )$ be such that~\eqref{contradiction-assump-1} holds and $s \in (0, s_0)$ implies that
	\begin{equation}
	\label{vquot}
	\frac{\|v(2s)\|}{2s} \leq c/8.
	\end{equation}
	Consider $s \in (0, s_0)$ such that $\|v'(s)\| \geq c$. Then there exist $s^{\pm}$ such that $(s^{-}, s)$ and $(s, s^{+})$ are connected components of the set $\{ r \in (0, 2s_0): \|v(r) - v(s) \| > c | r -s|/ 2\}$.  Note that~\eqref{vquot} implies that $s^{-} > 0$, and that $s^{+} < 2s_0$, since
	\[
	\frac{\| v(2s_0)  -v (s)\|}{| 2s_0 - s|} \leq \frac{\|v(2s_0)\|}{s_0} + \frac{\|v(s)\|}{s_0} \leq 2\frac{\|v(2s_0)\|}{2s_0} + \frac{\|v(s)\|}{s} \leq 3c/8.
	\]
	By the Besicovitch covering theorem we may extract from the collection $\{ (s^{-}, s^{+}) : s \in (0, s_0), \ \|v'(s)\| \geq c\}$ a pairwise disjoint subcollection $\mathcal{I} = \{ (s_i^{-}, s_i^{+})\}_{i = 1}^{\infty}$, say, such that 
	\[
	\meas{\bigcup \mathcal{I} \cap \{r \in (0, s_0) : \|v'(r) \| \geq c\} } \geq \meas{\{r \in (0, s_0) : \|v'(r) \| \geq c\}} /2 >\alpha s_0/2.
	\]
	Define $u \in \scra{v(a)}{v(b)}$ by
	\[
	u(r) \df
	\begin{cases}
	v(r) & r \notin \bigcup_{i=1}^{\infty} (s_i^{-}, s_i) \cup (s_i, s_i^{+}),\\
	\mathrm{affine} & \mathrm{otherwise}.
	\end{cases}
	\]
	Let $i \geq 1$, and let $I_i^{-} \df (s_i^{-}, s_i)$ and $I_i^{+} \df (s_i, s_i^{+})$.  By choice of $s_i^{\pm}$, on $I_i^{\pm}$ we have that $\|u'\| = c/2$, and, furthermore, that $\|v'(r)\| \geq c$ implies that $\|v'(r) - u'\| \geq c/2$.  Hence
	\[
	\meas{ \{ r \in I_i^{\pm} : \|v'(r) - u'\| \geq c/2\}} \geq \meas{ \{ r \in I_i^{\pm} : \|v'(r)\| \geq c\} }.
	\]
	So lemma~\ref{key-lemma} implies that
	\[ 
	\int_{I_i^{\pm}} L(r, v, v')\, dr \geq \int_{I_i^{\pm}}L(r, u , u') \, dr + \tau_R \meas{ \{r \in I_i^{\pm} :  \|v'(r)\| \geq c\}} -  \alpha \tau_R \meas{I_i^{\pm}}/8,
	\]
	and so summing, since $\mathcal{I}$ is pairwise disjoint and $\bigcup \mathcal{I} \subseteq (0, 2s_0)$, gives that
	\begin{align*}
	\lefteqn{\int_{\bigcup \mathcal{I}} L(r, v, v')\, dr }\\
	& \geq \int_{\bigcup \mathcal{I}} L(r, u, u') \, dr + \tau_R \meas{ \bigcup \mathcal{I} \cap\{ r \in (0, s_0) : \|v'(r)\| \geq c\} } - \alpha \tau_R \meas{ \bigcup \mathcal{I} } / 8 \\
	&\geq \int_{\bigcup \mathcal{I}}L(r, u ,u')\, dr + \alpha \tau_R s_0/2 - \alpha \tau_R s_0 / 4 \\
	& = \int_{\bigcup \mathcal{I}}L(r, u, u') + \tau_R \alpha s_0/ 4,
	\end{align*}
	which is a contradiction.
\end{proof}
\begin{theorem}
	\label{thm:apcontinfinite}
	Let $v \in W^{1,1}((a,b); \mathbb{R}^n)$ be a minimizer of $\mathscr{L}$ over $\scra{v(a)}{v(b)}$, and suppose for some $t \in [a,b]$ that, respectively, $t \in (a, b]$ and $D^{-}v(t)$ exists and at least one component is infinite; or $t \in [a, b)$ and $D^{+}v(t)$ exists and at least one component is infinite.
	
	Then $v'$ is approximately continuous on the left, respectively right, at $t$, in the sense that for all $m > 0$, 
	\begin{align*}
	\lim_{s \uparrow t}(t-s)^{-1}\meas{  \{ r \in (s, t) : \|v'(r)\| \leq m \}} & = 0, \ \textrm{respectively}\\
	\lim_{s \downarrow t} (s- t)^{-1}\meas{ \{ r \in (t, s) : \|v'(r)\| \leq m \} } & = 0.
	\end{align*}
\end{theorem}
\begin{proof}
	We consider that case in which $t \in [a, b)$; the other case is similar.
	
	Without loss of generality we may assume that $t = 0 \in [a, b)$ and $v(t) = 0$. We suppose for a contradiction that there exist $m \in (1, \infty)$, $\alpha \in (0,1)$, and  arbitrarily small $s > 0$ such that 
	\begin{equation}
	\label{contra-assump-2}
	\meas{\{r \in (0, s) :  \|v'(r)\| \leq m\} } > \alpha s.
	\end{equation}
	Let $\delta \in (0, b)$ be as given by lemma~\ref{key-lemma} for $R \geq 2m$, and $\epsilon \leq \tau_R \alpha/16$.  Choose $s_0 \in (0, \delta/2)$ such that~\eqref{contra-assump-2} holds and such that $s \in (0,s_0)$ satisfies
	\begin{equation}
	\label{vquot2}
	\frac{\|v(s)\|}{s} \geq 3m.
	\end{equation}
	Consider $s \in (0,s_0)$ such that $\|v'(s)\| \leq m$.  Then there exist $s^{\pm}$ such that $(s^{-}, s)$ and $(s, s^{+})$ are connected components of the set $\{ r \in (0,s_0) : \|v(r) - v(s)\| < 2m |r-s|\}$.  Note that $s^{-} >0$ by~\eqref{vquot2}.  Define $\sigma_s \df  s - s^{-} > 0$.  By the Besicovitch covering theorem we may extract from the collection $\{ (s- \sigma_s, s+ \alpha \sigma_s/8): s \in (0, s_0),\ \|v'(s)\| \leq m \}$ a pairwise disjoint subcollection $\mathcal{I} = \{ (s_i - \sigma_i, s_i +\alpha \sigma_i /8)\}_{i=1}^{\infty}$, say,  such that 
	\[
	\meas{\bigcup \mathcal{I} \cap \{r \in (0, s_0) : \|v'(r)\| \leq m\} } > \alpha s_0/2.
	\]
	For each $i \geq 1$, since $\alpha/8 < 1$ and  $0 < \sigma_i < s_i$, we see that $s_i + \alpha \sigma_i/8 < 2s_i <2s_0$, i.e.\ $\bigcup \mathcal{I} \subseteq (0, 2s_0)$.  Since $\mathcal{I}$ is pairwise disjoint, we see that 
	\[
	\sum_{i=1}^\infty\sigma_i = \lambda \left( \bigcup_{i=1}^\infty (s_i - \sigma_i , s_i) \right) \leq 2s_0,
	\]
	so 
	\begin{align*}
	\lambda\left( \bigcup_{i=1}^{\infty} (s_i, s_i + \alpha \sigma_i / 8) \cap \{ r \in (0, s_0) : \|v'(r)\| \leq m\} \right)
	& \leq  \lambda \left( \bigcup_{i=1}^{\infty}  (s_i, s_i + \alpha \sigma_i/8) \right) \\
	& = \alpha \left(\sum_{i=1}^{\infty}  \sigma_i\right)/8 \\
	& \leq \alpha s_0 /4,
	\end{align*}
	and so
	\begin{equation}
	\label{left big}
	\lambda \left( \bigcup_{i=1}^{\infty}  (s_i - \sigma_i, s_i) \cap \{r \in (0, s_0):  \| v'(r)\| \leq m\} \right) \geq \alpha s_0 /4.
	\end{equation}
	Define $u \in \scra{v(a)}{v(b)}$ by
	\[
	u(r) \df 
	\begin{cases}
	v(r) & r \notin \bigcup_{i=1}^{\infty}  (s_i - \sigma_i, s_i),\\
	\textrm{affine} & \textrm{otherwise}.
	\end{cases}
	\]
	Fix $i \geq 1$.  On $(s_i -\sigma_i, s_i)$, we have that $\|u'\| = 2m$, and, furthermore, that $\| v'(r)\| \leq m$ implies $\|v'(r) - u'\| \geq m$, so
	\[
	\meas{ \{ r \in (s_i- \sigma_i, s_i) : \|v'(r) - u'\| \geq m\} } \geq \meas{ \{ r \in (s_i - \sigma_i, s_i) : \|v'(r)\| \leq m\}}.
	\]
	So lemma~\ref{key-lemma} implies that 
	\begin{align*}
	\lefteqn{ \int_{s_i- \sigma_i}^{s_i} L(r, v, v')\, dr }\\
	& \geq \int_{s_i - \sigma_i}^{s_i} L(r, u ,u')\, dr + \tau_R \meas{ \{r \in (s_i - \sigma_i, s_i) :  \|v'(r)\| \leq m\}} - \alpha \tau_R \sigma_i / 16, 
	\end{align*}
	and so, summing, since $\mathcal{I}$ is pairwise disjoint and $\bigcup \mathcal{I} \subseteq (0, 2s_0)$, gives, by~\eqref{left big},  that
	\allowdisplaybreaks{
		\begin{align*}
		\lefteqn{\int_{\bigcup_{i = 1}^{\infty}  (s_i - \sigma_i, s_i)} L(r, v, v') \, dr}\\
		&\geq \int_{\bigcup_{i = 1}^{\infty}  (s_i - \sigma_i, s_i)} L(r, u, u') \, dr \\
		&\phantom{=} {}+ \tau_R \meas{ \bigcup_{i = 1}^{\infty}  (s_i - \sigma_i, s_i) \cap \{ r \in (0, s_0) : \|v'(r)\| \leq m\} } - \alpha \tau_R \meas{ \bigcup \mathcal{I}} / 16 \\
		& \geq \int_{\bigcup_{i = 1}^{\infty}  (s_i - \sigma_i, s_i)}L(r, u ,u')\, dr +  \alpha \tau_R  s_0 / 4 - \alpha \tau_R s_0 / 8 \\
		& = \int_{\bigcup_{i = 1}^{\infty}  (s_i - \sigma_i, s_i)}L(r, u ,u') \, dr+ \alpha \tau_R  s_0 / 8,
		\end{align*}
	}
	which is a contradiction.
\end{proof}
\begin{corollary}
	\label{cor:nokinkfinite}
	Let $v \in W^{1,1}((a,b); \mathbb{R}^n)$ be a minimizer of $\mathscr{L}$ over $\scra{v(a)}{v(b)}$, and $t \in (a,b)$ be such that $D^{\pm}v(t)$ both exist, and each component of both is finite.
	
	Then $D^{-}v(t) = D^{+}v(t)$.
\end{corollary}
\begin{proof}
	Without loss of generality we may assume that $t=0 \in (a, b)$ and $v(t) = 0$.  We assume for a contradiction that $D^{-}v(0) \neq D^{+}v(0)$, so we may further suppose without loss of generality that $-D^{-}v_1(0) = D^{+}v_1(0) = m > 0$, say.
	
	Let $\delta \in (0, \min\{|a|, |b|\})$ be as given by lemma~\ref{key-lemma} for $R \geq\|D^{+}v(0)\|+ \|D^{-}v(0)\| + m + 2 m^{-1} + 1$, and $\epsilon \leq \tau_R / 4$. 
	
	By theorem~\ref{thm:apcontfinite} we can choose $s_0 \in (0, \delta/ 3)$ such that $s \in (0,s_0)$ satisfies 
	\begin{gather}
	\left\| (3s)^{-1}v(3s) - D^{+}v(0) \right\| <m/2 \ \text{and} \ \left\| (-3s)^{-1}v(-3s) - D^{-}v(0)\right\| < m/2; \ \textrm{and}  \label{vquots}\\
	\begin{cases}
	\meas{ \{ r \in (0, s) : \| v'(r)  - D^{+}v(0)\| \geq m/2 \}} < s/2,&\\
	\meas{ \{ r \in (-s, 0) : \|v'(r) - D^{-}v(0)\| \geq m/2 \}} < s/2.& 
	\end{cases}
	\label{apcont}
	\end{gather}
	Fix $s_1 \in (0, s_0)$.  Now,~\eqref{vquots} implies that $\left| \frac{v_1(-s_1)}{-s_1} + m\right| < m/2$, and hence that $ms_1 / 2< v_1 (-s_1) < 3 m s_1 /2$, and, furthermore, that $\left| \frac{v_1(3s_1)}{3s_1} - m\right| < m/2$, and hence that $v_1(3 s_1) > 3ms_1 / 2$.  So $v_1(0) = 0 < ms_1 / 2<  v_1(-s_1) < 3ms_1/2 < v_1(3s_1)$, and therefore there exists $s_2 \in (0, 3s_1)$ such that $v_1(s_2)= v_1(-s_1)$.  Define $u \in \scra{v(a)}{v(b)}$ by 
	\[
	u(r) \df 
	\begin{cases}
	v(r) & r \notin (-s_1, s_2),\\
	\textrm{affine} & \textrm{otherwise}.
	\end{cases}
	\]
	Then on $(-s_2, s_1)$,  
	\begin{align*}
	\|u'\| \leq \frac{\|v(s_2)\|}{s_2 + s_1} + \frac{\|v(-s_1)\|}{s_2 + s_1}
	& \leq \frac{ \| v (s_2)\|} {s_2} + \frac{\| v(-s_1)\|}{s_1} \\
	& \leq \|D^{+}v(0)\| + m/2 + \|D^{-}v(0)\| +m/2 \\
	& \leq R.
	\end{align*}
	For $r \in (0, s_2)$, we see that $\|v'(r) - D^{+}v(0)\| \leq m/2$ implies that $\|v'(r) - u'\| \geq |v_1'(r) - u_1'(r)| = |v_1'(r)| > m/2$, hence, by~\eqref{apcont}, that 
	\begin{align*}
	\meas{ \{ r \in (0, s_2) : \|v'(r)- u' \| \geq m/2 \} } 
	& \geq \meas{ \{ r \in (0, s_2) : \|v'(r) - D^{+}v(0)\| \leq m/2\}  }\\
	& \geq s_2 / 2.
	\end{align*}
	Similarly, 
	\begin{align*}
	\meas{ \{r \in (-s_1, 0) : \|v'(r) - u'\| \geq m/2 \}  } 
	&  \geq \meas{ \{ r \in (-s_1, 0): \|v'(r) - D^{-}v(0)\| \leq m/2\}} \\
	& \geq s_1 / 2.
	\end{align*}
	So lemma~\ref{key-lemma} implies that 
	\begin{align*}
	\lefteqn{\int_{-s_1}^{s_2}L(r, v, v') \, dr}\\
	& \geq \int_{-s_1}^{s_2} L(r, u, u') \, dr+ \tau_R \meas{ \{ r \in (-s_1, s_2) : \|v'(r) - u'\| \geq m/2\}} - \tau_R (s_1 + s_2)/4\\
	& \geq \int_{-s_1}^{s_2} L(r, u, u') \, dr+ \tau_R (s_1 + s_2)/2 - \tau_R (s_1 + s_2)/ 4 \\
	&  = \int_{-s_1}^{s_2} L(r, u, u')\, dr + \tau_R (s_1 + s_2)/4,
	\end{align*}
	which is a contradiction.
\end{proof}
\begin{corollary}
	\label{thm:infinite one side}
	Let $v \in W^{1,1}((a,b); \mathbb{R}^n)$ be a minimizer of $\mathscr{L}$ over $\scra{v(a)}{v(b)}$, and suppose that $t \in (a,b)$ is such that $D^{\pm}v(t)$ both exist.
	
	Then if at least one component of one one-sided derivative is infinite, then at least one component of the other one-sided derivative is infinite.  Note that this statement does {\it not} assert that there is one coordinate function with infinite left and right derivatives, and in fact I do not know whether such an assertion can be made in general.
\end{corollary}
\begin{proof}
	We suppose for a contradiction that the result is false.  We consider the case in which at least one component of $D^{+}v(t)$ is infinite, but all components of $D^{-}v(t)$ are finite; the other case is similar.
	
	Without loss of generality we may assume that $t= 0 \in (a,b)$, $v(t) = 0$, and $D^{-}v(0) = 0$.
	
	Let $\delta \in (0, \min\{|a|, |b|\})$ be as given by lemma~\ref{key-lemma} for $R \geq 2$ and $\epsilon  \leq \tau_R / 4$.
	By theorems~\ref{thm:apcontfinite} and~\ref{thm:apcontinfinite} we can choose $s_0 \in (0, \delta)$ such that $s \in (0, s_0)$ satisfies
	\begin{gather}
	\frac{\|v(-s)\|}{s} < 1/2,\ \textrm{and}\ \frac{\|v(s)\|}{s} > 3;\label{vquots2}\ \textrm{and}\\
	\begin{cases}
	\meas{ \{ r \in (-s, 0):  \|v'(r)\| \geq 1/2 \}  } < s/2, &\\
	\ \meas{ \{ r \in (0, s) : \|v'(r)\| \leq 2\} } < s/2.&
	\end{cases}
	\label{apcont2}
	\end{gather}
	Fix $s_1 \in (0, s_0)$, and consider the set  $\{ r \in  (-s_1, s_0) : \| v(r) - v(-s_1) \| < |r + s_1| \}$.  By~\eqref{vquots2} we know that $0$ lies in this set, but since 
	\[
	\|v(s_1) - v(-s_1)\| \geq \|v(s_1)\| - \|v(-s_1)\| \geq 3s_1 - s_1/2 > 2s_1,
	\]
	we see that $s_1$ does not.  Therefore there exists $s_2 \in (0, s_1)$ such that $\|v(s_2) - v(-s_1)\| = |s_2 + s_1|$.  Define $u \in \scra{v(a)}{v(b)}$ by
	\[
	u(r) \df
	\begin{cases}
	v(r) & r \notin (-s_1, s_2),\\
	\textrm{affine} & \textrm{otherwise}.
	\end{cases}
	\]
	Then on $(-s_1, s_2)$, we have that $\|u'\| = 1$,  and, furthermore, that $\|v'(r)\| \leq 1/2$ implies that $\|v'(r) - u'\| \geq 1/2$, and $\|v'(r)\| \geq 2$ implies that $\| v'(r) - u'\| \geq 1/2$, so by~\eqref{apcont2}, 
	\begin{align*}
	\meas{ \{r \in (-s_1, 0):  \|v'(r) - u'\| \geq 1/2 \}}& \geq \meas{ \{r \in (-s_1, 0) :  \|v'(r)\| \leq 1/2 \}} \geq s_1/2,
	\shortintertext{and}
	\meas{ \{r \in (0, s_2) :  \|v'(r) - u'\| \geq 1/2 \}} & \geq \meas{ \{ r \in (0, s_2) : \|v'(r)\| \geq 2\} } \geq s_2 / 2.
	\end{align*}
	So lemma~\ref{key-lemma} implies that 
	\begin{align*}
	\lefteqn{\int_{-s_1}^{s_2} L(r, v, v') \, dr}\\
	& \geq \int_{-s_1}^{s_2} L(r, u, u') \, dr + \tau_R \meas{ \{ r \in (-s_1, s_2) : \|v'(r) - u'\| \geq 1/2\} } + \tau_R (s_1 + s_2)/4\\
	& \geq \int_{-s_1}^{s_2} L(r, u, u')\, dr + \tau_R(s_1 + s_2)/2 - \tau_R (s_1 + s_2)/4 \\
	& = \int_{-s_1}^{s_2} L(r, u, u') \, dr + \tau_R(s_1 + s_2)/4,
	\end{align*}
	which is a contradiction.
\end{proof}
More information is available about the behaviour of infinite derivatives if we can locate them only in one coordinate function.
\begin{theorem}
	\label{apcontv1}
	Let $v \in W^{1,1}((a,b); \mathbb{R}^n)$ be a minimizer of $\mathscr{L}$ over $\scra{v(a)}{v(b)}$, and suppose for some $t \in [a,b]$  that, respectively,  $t \in (a, b]$, $D^{-}v_1(t)$ exists as an infinite value, and $v_j$ are Lipschitz in a left-neighbourhood of $t$ for $2 \leq j \leq n$; or that $t \in [a, b)$, $D^{+}v_1(t)$ exists as an infinite value, and $v_j$ are Lipschitz in a right-neighbourhood of $t$ for $2 \leq j \leq n$.  
	
	Then $v_1'$ is approximately continuous on the left, respectively right, at $t$, in the sense that for all $m>0$, 
	\begin{align*}
	\lim_{s \uparrow t} (t-s)^{-1}\meas{ \{ r \in (s,t) : |v_1'(r)| \leq m\}} & = 0,\ \textrm{respectively}\\
	\lim_{s \downarrow t} (s - t)^{-1}\meas{ \{ r \in (t, s) : |v_1'(r)| \leq m \} } & = 0.
	\end{align*}
\end{theorem}
Note that we do not assume that the derivatives of the components $v_j$ for $2 \leq j \leq n$ exist at $t$.
\begin{proof}
	We consider the case in which $t \in [a, b)$; the other case is similar.
	
	Without loss of generality we may assume that $t=0 \in [a, b)$ and $v(t) = 0$.  Suppose for a contradiction that there exist $m \in (1, \infty)$, $\alpha \in (0,1)$, and arbitrarily small $s > 0$ such that
	\begin{equation}
	\label{contraassump3}
	\meas{ \{ r \in (0, s) :  |v_1'(r) | \leq m\}} > \alpha s.
	\end{equation}
	Choose $\eta \in (0, b)$ such that $v_j$ are Lipschitz on $[0, \eta)$ for $2 \leq j \leq n$, and let $\delta \in (0, b)$ be as given by lemma~\ref{key-lemma} for $R \geq 2m + \sum_{j=2}^n \mathrm{Lip}(v_j\vert_{[0, \eta)})$, and $\epsilon \leq \alpha \tau_R / 16$.
	
	Choose $s_0 \in (0, \min \{\delta, \eta\}/2)$ such that~\eqref{contraassump3} holds and $s \in (0, s_0)$ satisfies 
	\begin{equation}
	\label{v1 big}
	|v_1(2s)|/2s > 2R.
	\end{equation}
	Consider $s \in (0, s_0)$ such that $|v_1'(s)|\leq m$.  Then $|v_1(r) - v_1(s)|/ |r-s| < 2m$ for $r$ in some neighbourhood of $s$ contained in $(0, s_0)$, and so 
	\[
	\frac{ \| v(r) - v(s)\|}{|r-s|} < 2m + \sum_{j=2}^n \mathrm{Lip}(v_j \vert_{[0, \eta)})\leq  R,
	\]
	for $r$ in some neighbourhood of $s$ contained in $(0, s_0)$.  So there exist $s^{\pm}$ such that $(s^{-}, s)$ and $(s, s^{+})$ are connected components of the set $\{ r \in (0, s_0) : \|v(r) - v(s)\| < R |r-s|\}$.  Note that  $s^{-} > 0$ by~\eqref{v1 big}.  We now proceed similarly to the proof of theorem~\ref{thm:apcontinfinite}.  Let $\sigma_s \df s - s^{-} > 0$.  By the Besicovitch covering theorem we may extract from the collection $\{(s-\sigma_s, s+ \alpha \sigma_s/8):s \in (0, s_0),\ |v_1'(s)|\leq m\}$ a pairwise disjoint collection $\mathcal{I} = \{ (s_i - \sigma_i, s_i + \alpha \sigma_i/ 8)\}_{i=1}^{\infty}$, say, such that $\bigcup \mathcal{I} \subseteq (0, 2s_0)$, and 
	\begin{equation}
	\label{big cover}
	\lambda \left( \bigcup_{i=1}^{\infty} (s_i - \sigma_i, s_i) \cap \{ r \in (0, s_0) : |v_1'(r) | \leq m\} \right) \geq \alpha s_0 / 4.
	\end{equation}
	Define $u \in \scra{v(a)}{v(b)}$ by
	\[
	u(r) \df 
	\begin{cases}
	v(r) & r \notin \bigcup_{i=1}^{\infty} (s_i - \sigma_i, s_i),\\
	\textrm{affine} & \textrm{otherwise}.
	\end{cases}
	\]
	Fix $1 \geq 1$.  On $(s_i - \sigma_i, s_i)$, we have that $\|u ' \| = R$, and, furthermore, that  $|v_1'(r)| \leq m$ implies that $\|v'(r)\| \leq m + \sum_{j=2}^n \mathrm{Lip}(v_j \vert_{[0, \eta)})$, which implies that $\|v'(r)  - u'\| \geq m$.  So  
	\[
	\meas{ \{ r \in (s_i - \sigma_i, s_i) : \|v'(r) - u'\| \geq m \} } \geq \meas{ \{ r \in (s_i - \sigma_i, s_i):  |v_1'(r) | \leq m\} }.
	\]
	So lemma~\ref{key-lemma} implies that
	\begin{align*}
	\lefteqn{\int_{s_i -\sigma_i}^{s_i} L(r, v, v')\, dr} \\
	&  \geq \int_{s_i - \sigma_i}^{s_i} L(r, u, u')\, dr + \tau_R \meas{ \{ r\in (s_i - \sigma_i, s_i) :  |v_1'(r)| \leq m \}} - \alpha \tau_R \sigma_i / 16,
	\end{align*}
	and so, summing, since $\mathcal{I}$ is pairwise disjoint and $\bigcup \mathcal{I} \subseteq (0, 2s_0)$, gives, by~\eqref{big cover}, that
	\begin{align*}
	\lefteqn{\int_{\bigcup_{i = 1}^{\infty} (s_i - \sigma_i, s_i)} L(r, v, v') \, dr }\\
	& \geq \int_{\bigcup_{i = 1}^{\infty} (s_i - \sigma_i, s_i)} L(r, u, u')\, dr \\
	& \phantom{=} {}+ \tau_R \meas{ \bigcup_{i = 1}^{\infty} (s_i - \sigma_i, s_i)\cap \{ r \in (0, s_0) : |v_1'(r)|\leq m \} } - \alpha \tau_R \meas{ \bigcup \mathcal{I}} / 16\\
	& \geq \int_{\bigcup_{i = 1}^{\infty}(s_i - \sigma_i, s_i) } L(r, u, u')\, dr +  \alpha \tau_R s_0 / 4 - \alpha \tau_R s_0 / 8\\
	& = \int_{\bigcup_{i = 1}^{\infty}(s_i - \sigma_i, s_i)} L(r, u ,u') \, dr+ \alpha \tau_R s_0 / 8,
	\end{align*}
	which is a contradiction.
\end{proof}
\begin{corollary}
	Let $ v \in W^{1,1}((a,b);\mathbb{R}^n)$ be a minimizer of $\mathscr{L}$ over $\scra{v(a)}{v(b)}$, and suppose that $t \in (a,b)$ is such that $D^{\pm}v(t)$ both exist, and $v_j$ is Lipschitz in a neighbourhood of $t$ for $2 \leq j \leq n$.  
	
	Then if one one-sided derivative of $v_1$ is infinite at $t$, then the two one-sided derivatives of $v_1$ are equal to the same infinite value.
\end{corollary}
\begin{proof}
	We suppose that $D^{+}v_1(t) = + \infty$; the other cases are similar.  
	
	Without loss of generality we may assume that $t= 0 \in (a, b)$ and $v(t) = 0$.  By corollary~\ref{thm:infinite one side}, we are just required to prove that $D^{-}v_1(0) > - \infty$.  Suppose for a contradiction that $D^{-}v_1(0) = - \infty$.  Choose $\eta\in (0, \min\{|a|, |b|\})$ such that $v_j$ are Lipschitz on $(-\eta, \eta)$ for $2 \leq j \leq n$.  Let $\delta \in (0, \min\{ |a| , |b|\})$ be as given by lemma~\ref{key-lemma} for $R \geq 1 + \sum_{j=2}^n \mathrm{Lip}(v_j \vert_{(-\eta, \eta)})$ and $\epsilon \leq \tau_R / 4$.   By theorem~\ref{thm:apcontinfinite} we may choose $s_0 \in (0,  \min \{ \eta, \delta\})$ such that $s \in (0, s_0)$ satisfies
	\begin{gather}
	\frac{v_1(s)}{s} > 1 \, \textrm{and}\ \frac{v_1(-s)}{-s} < -1\label{vquots4};\ \textrm{and} \\
	\begin{cases}
	\meas{ \{r \in (-s, 0) :  |v_1'(r)| \leq 3R \} } < s/ 2,&\\
	\meas{ \{ r \in (0, s) :  |v_1'(r)| \leq 3R \} } < s/2. &
	\end{cases}
	\label{apcont5}
	\end{gather}
	By~\eqref{vquots4}, $v_1(s_0) \geq s_0 > 0 = v_1(0)$, so we can choose $s_1 \in (0, s_0)$ such that $v_1(-s_1) < v_1(s_0)$.  Since, by~\eqref{vquots4}, $v_1(0) = 0 < s_1 < v_1(-s_1) < v_1(s_0)$, there exists $s_2 \in (0, s_0)$ such that $v_1(-s_1) = v_1(s_2)$.  Then $\|v(-s_1) - v(s_2)\| \leq \sum_{j = 2}^n \mathrm{Lip}(v_j \vert_{(-\eta, \eta)}) |s_2 + s_1| \leq R |s_2 + s_1|$.  Define $u \in \scra{v(a)}{v(b)}$ by
	\[
	u(r) \df 
	\begin{cases}
	v(r) & r \notin (-s_1, s_2),\\
	\textrm{affine} & \textrm{otherwise}.
	\end{cases}
	\]
	So  on $(-s_1, s_2)$ we have that $\|u'\| \leq R $, and, furthermore, that $|v_1'(r)| \geq 3R$ implies that $\|v'(r)\| \geq 2R $ and hence that $\|v'(r) - u'\| \geq R$, and so, by~\eqref{apcont5}, 
	\begin{align*}
	\meas{ \{ r \in (-s_1, s_2)  : \| v'(r) - u'\| \geq R \}} 
	& \geq \meas{ \{r \in (-s_1, s_2) :  |v_1'(r)| \geq 3R \} }\\
	& \geq (s_1 + s_2)/2.
	\end{align*}
	So lemma~\ref{key-lemma} implies that
	\begin{align*}
	\lefteqn{\int_{-s_1}^{s_2} L(r, v, v') \, dr}\\
	& \geq \int_{-s_1}^{s_2} L(r, u, u') \, dr+ \tau_R \meas{ \{ r \in (-s_1, s_2) : \|v'(r) - u'\| \geq R \} } - \tau_R (s_1 + s_2) / 4 \\
	& \geq \int_{-s_1}^{s_2} L(r, u, u') \, dr+ \tau_R ( s_1 + s_2)/2 - \tau_R (s_1 + s_2)/4 \\
	& = \int_{-s_1}^{s_2} L(r, u ,u')\, dr  + \tau_R (s_1 + s_2)/4,
	\end{align*}
	which is a contradiction.
\end{proof}
It is a rather delicate matter to investigate the behaviour of derivatives of individual coordinate functions once we admit infinite derivatives.  It seems that the presence of arbitrarily steep tangents in one variable can mask a multitude of sins in the others.  The following example demonstrates one such case: a vertical tangent in the second variable allows a cusp point in the first variable.
\begin{example}
	Let $v \in W^{1,1}((-1,1), \mathbb{R}^2)$ be defined by $v(t) = (|t|, \mathrm{sign}(t)|t|^{1/3})$, and let $\tilde{L} \colon [-1,1] \times \mathbb{R}^2 \times \mathbb{R}^2 \to [0, \infty)$ be given by 
	\[
	\tilde{L}(t, y, p) = ((y_1 - v_1(t))^2 + (y_2^3 - t)^2) p_2^6.
	\]
	
	Let $\eta > 0$ be the constant from the usual Mani\`a estimates (see for example~\cite{Buttazzo-Giaquinta-Hildebrandt}), i.e.\ such that 
	\[
	\int_0^1 (u(t)^3 - t)^2 (u'(t))^6 \, dt \geq \eta,
	\]
	for any $u \in W^{1,1}(0,1)$ such that $u(1) = 1$ and $u(t) < t^{1/3}/4$ for some $t > 0$.
	
	Suppose $u \in \scra{v(-1)}{v(1)}$.    If $u_2(t) < t^{1/3}/4$ for some $t \in (0,1)$, then $\int_{-1}^1 \tilde{L}(s, u, u')\, ds \geq \eta$; similarly if $u_2(t) > -t^{1/3}/ 4$ for some $t \in (-1, 0)$, then $\int_{-1}^1 \tilde{L}(s, u, u')\, ds \geq \eta$.
	
	Suppose now that $u_2(t) \geq t^{1/3}/4$ on $(0, 1]$ and $u_2(t) \leq -t^{1/3}/4$ on $[-1, 0)$.  Suppose further that $u_1(0) \neq v_1(0) = 0$.  So there exists some $t \in (0,1)$ such that $|u_1(s) - v_1(s)|> t$ on $(-t, t)$.  Then by Jensen's inequality,
	\begin{align*}
	\int_{-1}^1 \tilde{L}(s, u, u') \, ds \geq \int_{-t}^t (u_1(s) - v_1(s))^2 (u_2'(s))^6 \, ds
	& \geq t^2 \int_{-t}^t (u_2'(s))^6 \, ds \\
	& \geq 2t^3 \left( \frac{ u_2(t) - u_2(-t)}{2t}\right)^{\! 6} \\
	& \geq 2t^3 \left( \frac{ t^{1/3} + t^{1/3}}{8t}\right)^{\! 6}\\
	& \geq 2^{-11}.
	\end{align*}
	So if $u$ is such that $\int_{-1}^1 \tilde{L}(s, u , u') \, ds < \eta$, then $u_1(0) \neq v_1(0)$ implies that $\int_{-1}^1 \tilde{L}(s, u, u')\, ds \geq 2^{-11}$.  So if $u$ is such that $u_1(0) \neq v_1(0)$, then 
	\[
	\int_{-1}^1 \tilde{L}(s, u, u') \, ds\geq \min \{ \eta, 2^{-11}\}.
	\]
	Choose $\sigma \in (1, 3/2)$ and $\epsilon < \min\{\eta, 2^{-11}\}\left(\int_{-1}^1 ( |v_1'|^2 + |v_2'|^{\sigma}\right)^{-1}$, and define $L \colon [-1,1] \times \mathbb{R}^2 \times \mathbb{R}^2 \to [0, \infty)$ by 
	\[
	L(t, y, p) = ( ( y_1 - v_1(t))^2 + ( y_2^3 - t)^2) p_2^6 + \epsilon ( p_1^2 + p_2^{\sigma}).
	\]
	Let $u \in \scra{v(-1)}{v(1)}$ be a minimizer of $\mathscr{L}$, so by the above argument and the choice of $\epsilon$, $\mathscr{L}(u) \leq \mathscr{L}(v) < \min \{ \eta, 2^{-11} \}$, so $u_1(0) = v_1(0)$.  Suppose $I$ is a non-trivial component of $\{s \in (-1, 1) : u_1(s) \neq v_1(s)\}$, then $0 \notin I$, and so $v_1$ is linear on $I$, and therefore is the unique minimizer of the functional $v_1 \mapsto \int_I (v_1')^2$.  Define $\hat{u} \in \scra{v(-1)}{v(1)}$ by 
	\[
	\hat{u}(r) \df 
	\begin{cases} 
	u(r) & r \notin I,\\
	(v_1(r), u_2(r)) & r \in I.
	\end{cases}
	\]
	Then 
	\begin{align*}
	\int_I L(r, u ,u') \, dr
	& = \int_I ((u_1 - v_1)^2 + (u_2^3 - t)^2) (u_2')^6 + \epsilon \left( (u_1')^2 + (u_2')^{\sigma}\right) \, dr\\
	& >  \int_I ( u_2^3 - t)^2 (u_2')^6 + \epsilon \left( (v_1')^2 + (u_2')^{\sigma}\right) \, dr\\
	& = \int_I L(r, \hat{u}, \hat{u}')\, dr,
	\end{align*}
	which is a contradiction.  Hence $u_1 = v_1$, and in particular $D^{\pm}u_1(0) = \pm 1$.
\end{example}

\begin{acknowledgements}
I wish to thank John Ball for first posing and discussing with me the original question of approximation (question~\ref{Qu:JB}  in section~\ref{sec:variations}), which proved surprisingly diverting and served as the starting point for the current paper. I am also grateful to David Preiss for  encouragement, advice, and conversation on these topics.  Finally, I am indebted to an anonymous referee for comments on an earlier version of this article, in particular for suggesting and motivating significant improvements to the presentation.
\end{acknowledgements}
\def\cprime{$'$}

\end{document}